\documentclass{compositio}

\usepackage{amsmath,amsfonts,yhmath,hyperref}
\usepackage{amsthm}
\usepackage{amssymb}
\usepackage[latin9]{inputenc}
\usepackage[nottoc,notlof,notlot]{tocbibind}
\usepackage{mathtools}
\usepackage{enumitem}
\usepackage{tikz-cd}
\usepackage{multirow}

\usepackage{float}
\usepackage{array}
\usepackage{mathrsfs}

\usetikzlibrary{patterns}

\usepackage{pgf,tikz}

\usetikzlibrary{arrows}
\usetikzlibrary[patterns]

\newcommand{\white}{\node[draw,circle, minimum width=0.5cm]}
\newcommand{\black}{\node[draw,circle, minimum width=0.5cm,fill=black]}

\newcommand{\PP}{\mathbb{P}}

\newcommand{\ZZ}{\mathbb{Z}}
\newcommand{\RR}{\mathbb{R}}
\newcommand{\CC}{\mathbb{C}}

\renewcommand{\P}{\mathcal{P}}

\newcommand{\B}{\mathcal{B}}
\newcommand{\M}{\mathcal{M}}

\newcommand{\ang}[1]{\langle #1 \rangle}

\newcommand{\val}{\mathrm{val}}

\newcommand{\curve}{\node[draw,circle, fill=white]}

\newcommand{\bino}[2]{\begin{pmatrix}
#1 \\
#2 \\
\end{pmatrix}}

\newtheorem{theo}{Theorem}[section]

\newtheorem{prop}[theo]{Proposition}
\newtheorem{coro}[theo]{Corollary}
\newtheorem{lem}[theo]{Lemma}

\theoremstyle{definition}
\newtheorem{defi}[theo]{Definition}

\theoremstyle{remark}
\newtheorem{remark}[theo]{Remark}
\newenvironment{rem}[1]{
    \begin{remark}#1}{
    \xqed{\blacklozenge}\end{remark}
}

\theoremstyle{remark}
\newtheorem{example}[theo]{Example}
\newenvironment{expl}[1]{
    \begin{example}#1}{
    \xqed{\lozenge}\end{example}
}

\newcommand{\xqed}[1]{
    \leavevmode\unskip\penalty9999 \hbox{}\nobreak\hfill
    \quad\hbox{\ensuremath{#1}}}

\classification{14N10, 14T90, 05A15, 14H99, 14M25}
\keywords{Enumerative geometry, tropical descendant invariants\\ 
}
 
\begin{document}
 
 
\title{Tropical descendant invariants with line constraints}
\author{Thomas Blomme, Hannah Markwig}

\begin{abstract}
Via correspondence theorems, rational log Gromov--Witten invariants of the plane can be computed in terms of tropical geometry. For many cases, there exists a range of algorithms to compute tropically: for instance, there are (generalized) lattice path counts and floor diagram techniques. So far, the cases for which there exist algorithms do not extend to non-stationary rational descendant log Gromov--Witten invariants, i.e.\ those where Psi-conditions do not have to be matched up with the evaluation of a point. The case of rational descendant log Gromov--Witten invariants satisfying point conditions (without Psi-conditions) and one Psi-condition of any power combined with a line plays a particularly important role, since it shows up in mirror symmetry as contributions to coefficients of the $J$-function. We provide recursive formulas to compute those numbers via tropical methods. Our method is inspired by the tropical proof of the WDVV equations. We also extend our study to counts involving two lines, both paired up with a Psi-condition, appearing with power one. 
\end{abstract}

\maketitle

\tableofcontents

\section{Introduction}

Tropical geometry provides powerful methods for the study of enumerative geometry and Gromov--Witten theory. Roughly, the general approach requires two steps: given an enumerative problem, we first need a \emph{correspondence theorem} stating the equality of the complex count with a certain tropical count, where as usual, tropical objects have to be counted with multiplicity reflecting the number of complex objects degenerating to a fixed tropical object under tropicalization. Once such a correspondence theorem exists, the second problem is to provide algorithms within tropical geometry that count the tropical objects appearing in the correspondence theorem.
These two steps already appeared in the pioneering work of Mikhalkin \cite{mikhalkin2005enumerative}, where he proved a correspondence theorem for counts of plane curves satisfying point conditions, and the lattice path algorithm to compute such numbers on the tropical side.

Since then, many instances of correspondence theorems have been proved and many tropical algorithms to compute tropical enumerative numbers have been developed, see e.g.\ \cite{NS06, IKS05, BBM10, CJM10, Blo19, FM09, Gro16, BruMa, BBM11, BBBM13, Bou19, ABR11, Tyo09}.

For the plane, counts of curves satisfying point conditions are equal to Gromov--Witten invariants. Many researchers have worked on correspondences involving Gromov--Witten invariants. One reason for that is the appearance of Gromov--Witten invariants in mirror symmetry. Their tropical analogues accordingly play a role in the Gross--Siebert mirror symmetry program involving tropical methods \cite{GS06, GS07, Gro09}. Throughout the years, it became clear that tropical enumerative invariants most naturally reflect log Gromov--Witten invariants. Log Gromov--Witten theory is a recent and active field of research \cite{AC11, Che10, Che14, GS13, Ran19}.

In this paper, we focus on the case of rational descendant log Gromov--Witten invariants of the plane.
For rational curves, the correspondences with the tropical world are most powerful, as they can be carried out on the level of moduli spaces \cite{Ran15}. 

There is a moduli stack with logarithmic structure
\[
\overline{M}^{\mathsf{log}}_{0,n}(\PP^2,d),
\]
parametrizing log stable maps of degree $d$ to $\PP^2$ whose source is a rational curve with $n$ marked points.

This moduli space is a virtually smooth Deligne--Mumford stack and it carries a virtual fundamental class denoted by $[1]^\mathsf{log}$ in degree $3d-1 +n$. For each of the  $n$ marked points, there are evaluation morphisms
\[
\mathrm{ev}_i: \overline{M}^{\mathsf{log}}_{0,n}(\PP^2,d)\to \PP^2.
\]

For each of the $n$ marks there is a cotangent line bundle, whose first Chern class is denoted $\psi_i$.

\begin{defi}\label{def-logdescGWI}
The \emph{rational descendant log Gromov--Witten invariant} is defined as the following intersection number on $\overline{M}^{\mathsf{log}}_{0,n}(\PP^2, d)$:
\begin{equation}\label{eq-lgw}
\langle  \tau_{k_1}(\Xi_1)\ldots\tau_{k_n}(\Xi_n)\rangle^\mathsf{log}_{0,d}=
\int_{[1]^\mathsf{log}} \prod_{j=1}^n \psi_j^{k_j}\mathrm{ev}_j^\ast([H_j]),
\end{equation}
where $\Xi_j$ stands for either the class of a line $L$ or the class of a point $P$.
\end{defi}

Analogously, one can define tropical rational log Gromov--Witten invariants as an intersection product on a tropical moduli space parametrizing tropical stable maps, see Section \ref{subsec-tropicaldesc}.
We denote the \emph{tropical rational descendant log Gromov--Witten}  by
$$\ang{\psi^{k_1}\Xi_1,\cdots,\psi^{k_n}\Xi_n}_d,$$
where $\Xi_j$ is either a tropical line $L$ or a point $P$.

\begin{theo}\emph{(Correspondence Theorem, \cite{Gro18, MR16}.)}
The rational descendant log Gromov--Witten invariant equals its tropical counterpart, \textit{i.e.}\ we have
$$\langle  \tau_{k_1}(\Xi_1)\ldots\tau_{k_n}(\Xi_n)\rangle^\mathsf{log}_{0,d}=\ang{\psi^{k_1}\Xi_1,\cdots,\psi^{k_n}\Xi_n}_d.$$
\end{theo}

Note that the correspondence theorem of \cite{Gro18} is stated in terms of $\psi$-classes which are obtained as pullbacks from the moduli space of stable curves under the morphism which forgets the map. The corresponding invariants are sometimes also called ancestors instead of descendants. Proposition
3.4 of \cite{MR16} shows that for log Gromov--Witten invariants of toric varieties, this is equivalent to the usual version of $\psi$-classes.

For the case of tropical stationary rational descendant log Gromov--Witten invariants, \textit{i.e.}\ when $\psi$-conditions are only paired up with points ($\Xi_i=P$ if $k_i>0$) there exists a recursion similar to the WDVV equations which  computes  those numbers \cite{MR08}. Furthermore, there exists a generalized lattice path algorithm to compute those numbers \cite{MR08}. Also floor diagram techniques have been successfully applied for the computation of tropical stationary rational descendant log Gromov--Witten invariants \cite{BGM10}.

For the case where $\psi$-conditions are allowed to be paired up with line conditions, no tropical algorithm to compute these numbers is known. To the best of our knowledge, neither does there exists an algorithm to compute the corresponding rational descendant log Gromov--Witten invariants without relying on the correspondence theorem and tropical methods.

We close this gap by providing recursive formulas for two cases.

To simplify the notation further, we drop insertions of the form $ \tau_{0}(P)$ resp.\ $\psi^{0}P$, i.e.\ simple point conditions. With this simplified notation, the two cases we consider can be stated as follows.

We provide recursions that suffice to compute tropical rational log Gromov--Witten invariants of the form
\begin{enumerate}
\item $\langle \psi^{k}L\rangle_d$, where $k$ can be chosen arbitrarily. That is, we consider counts of rational tropical curves that satisfy point conditions and that satisfy a $\psi$-condition of arbitrary power together with a tropical line condition. This means we count rational tropical curves passing through points and having a vertex of valency $k+2$ (where the marked point itself is not counted towards the valency) on a tropical line. 
These results can be found in Section \ref{sec-WDVV1}.
\item $\langle\psi L, \psi L\rangle_d$. That is, we consider counts of rational tropical curves satisfying point conditions and two simple $\psi$-conditions with a line each. Put differently, we count rational tropical curves passing through points and having two trivalent marked vertices (where the marked point itself is not counted towards the valency) each on a tropical line. These results can be found in Section \ref{sec-WDVV2}.
\end{enumerate}

Our methods from Section \ref{sec-WDVV1} can also be applied for tropical rational descendant log Gromov--Witten invariants of the form $\ang{\psi^{k_1}P,\cdots,\psi^{k_{n-1}}P,\psi^{k_n}L}_d$, where each $k_i\geq 0$. In order to avoid heavy notation in the recursive formula, we choose to restrict the presentation to the case where $k_i=0$ for $i=1,\ldots,n-1$.

Another motivation to focus on this case comes from its importance for mirror symmetry of $\PP^2$, investigated by Gross in \cite{Gro09}. Our tropical invariants of the form  $\langle \psi^{k}L\rangle_d$ appear in Definition 3.4 (2)(a) in \cite{Gro09}. 
Later, Gross' result was generalized by Overholser \cite{Ove15} to include also more generalized rational descendant Gromov--Witten invariants. Our invariants appear as contributions to coefficients of the $J$-function which is based on ordinary Gromov-Witten invariants rather than their log versions. The relation between ordinary and log Gromov-Witten invariants in this setting is discussed in \cite{Gro09, Ove15, MR16}. Combinatorially, the tropical curves we count also show up in Definition 3.4 (3)(c) in \cite{Gro09} in another contribution to the $J$-function, however with slightly different multiplicity, which arises due to the relation between ordinary and log Gromov-Witten invariants.

Our approach to establish the recursions is motivated by the tropical proof of Kontsevich's formula and the WDVV equations \cite{GM053, MR08}.
For small values of $k$ resp.\ $d$ in  $\langle \psi^{k}L\rangle_d$ resp.\ for $\langle\psi L, \psi L\rangle_d$ for any $d$, we also provide another, more direct, method to compute our tropical rational descendant log Gromov--Witten invariants and use it to test it against our other formulas, confirming the results.

This paper is organized as follows. In Section \ref{sec-tropical}, we introduce parametrized tropical curves (a.k.a. tropical stable maps) and their moduli spaces. We also introduce tropical $\psi$-classes and their intersections. In Section \ref{sec-enumerative}, we introduce tropical enumerative problems which appear in our recursive formulas, in addition to the tropical rational descendant log Gromov--Witten invariants we focus on in this paper. In Section \ref{sec-explicit}, we provide more direct computations for small values of $k$ resp.\ $d$ resp.\ for $\langle\psi L, \psi L\rangle_d$, which we later use to test our recursive formulas. In Section \ref{sec-WDVV1}, we present the recursion for tropical rational descendant log Gromov--Witten invariants of the form $\langle \psi^{k}L\rangle_d$ and in Section \ref{sec-WDVV2} for tropical rational descendant log Gromov--Witten invariants of the form  $\langle\psi L, \psi L\rangle_d$.

\medskip

\textit{Acknowledgements.} We would like to thank Andreas Gross and Dhruv Ranganathan for interesting discussions. We thank an anonymous referee for useful comments to improve an earlier version of this paper. We acknowledge support from the SFB-TRR 195 \emph{Symbolic Tools in Mathematics and their Application} of the German Research Foundation (DFG) Project-ID 286237555, TRR 195, and the SNSF grant 204125.

\section{Tropical rational stable curves and maps and their moduli spaces}\label{sec-tropical}

We  briefly recall the notion of rational tropical curve and the definition of their moduli space. The latter was introduced by G. Mikhalkin in \cite{mikhalkin2007moduli}. We also refer to \cite{gathmann2009tropical} for more details.

	\subsection{Tropical curves}

		\subsubsection{Abstract curves.} Let $\overline{\Gamma}$ be finite graph without cycles and without bivalent vertices. The set of $1$-valent vertices is denoted by $\Gamma_\infty$, and its complement $\Gamma=\overline{\Gamma}\backslash\Gamma_\infty$ is endowed with a complete metric. Such a metric is defined up to isometry by declaring edges adjacent to a removed $1$-valent vertex to be isometric to $[0;+\infty[$ (such edges are called \textit{ends}), and the other edges to be isometric to some $[0;l]$ (such edges are called \textit{bounded}), and $l>0$ is called the \textit{length} of the edge.
		
		The ends of $\Gamma$ are indexed by the $1$-valent vertices of $\overline{\Gamma}$. The homeomorphism type of $\Gamma$, which is obtained after forgetting the metric, is called the \textit{combinatorial type} of $\Gamma$. For a fixed number of ends $m=|\Gamma_\infty|$, there are a finite number of combinatorial types, and the ends can be labelled by $[\![1;m]\!]$.

		\subsubsection{Parametrized rational tropical curves.} We consider parametrized rational tropical curves inside $\RR^2$. These are sometimes also called tropical stable maps. When the context is clear, we drop the parametrized and just speak about tropical curves.
		
		\begin{defi}
		A parametrized rational tropical curve in $\RR^2$ is a map $h:\Gamma\to\RR^2$ where:
		\begin{enumerate}[label=(\roman*)]
		\item $\Gamma$ is an abstract rational tropical curve,
		\item $h$ is affine with integer slope on each edge of $\Gamma$: the slope of $h$ along an oriented edge $e$, denoted by $\partial_e h$, lies in $\ZZ^2$,
		\item $h$ is balanced: at each vertex $V$ of $\Gamma$, $\sum_{e\ni V}\partial_e h=0$, where the edges are oriented outside $V$.
		\end{enumerate}
		For any edge of $\Gamma$, the integral length of $\partial_e h\in\ZZ^2$ is called the weight $w_e$ of $e$.
		\end{defi}

		\begin{defi}
		Let $h:\Gamma\to\RR^2$ be a parametrized rational tropical curve.
		\begin{itemize}[label=$\ast$]
		\item The edges with slope $0$ are called \textit{contracted edges}.
		\item The contracted ends are called marked points.
		\item The collection $\Delta$ of non-zero slopes of the ends of $\Gamma$ is called the \textit{degree}.
		\end{itemize}
		\end{defi}
		
		For a parametrized tropical curve $h:\Gamma\to\RR^2$, it is possible to label the ends by $[\![1;|\Delta|]\!]\sqcup[\![1;n]\!]$, where $\Delta$ is the degree, and $n$ is the number of marked points. We now define several specific degrees/classes.
		\begin{itemize}[label=-]
		\item A tropical curve is said to be of degree $dL$ (or just $d$) if the degree is
		$$\Delta_d=\{ (0,-1)^d,(-1,0)^d,(1,1)^d\}.$$
		Such curves correspond to degree $d$ curves inside $\PP^2$.
		\item A tropical curve is said to be of degree $dL-kE$ if the degree is
		$$\Delta_{d,k}=\{(0,-1)^{d-k},(-1,0)^{d-k},(-1,-1)^k,(1,1)^d\}.$$
		Such curves corresponds to curves inside the blow-up of $\CC P^2$ at one point.
		\item A tropical curve is said to be of degree $dL-k_1E_1-k_2E_2$ if the degree is
		$$\Delta_{d,k_1,k_2}=\{(0,-1)^{d-k_1-k_2},(-1,0)^{d-k_2},(-1,-1)^{k_1},(1,0)^{k_2},(1,1)^{d-k_2}\}.$$
		Such curves correspond to curves in the blow-up of $\PP^2$ at two distinct points.
		\end{itemize}
		
		\begin{figure}
		\begin{center}
		\begin{tabular}{ccc}
		\begin{tikzpicture}[line cap=round,line join=round,>=triangle 45,x=0.6cm,y=0.6cm]
		\draw[line width=1pt] (0,2)--++(1,0)--++(1,-1)--++(2,0)--++(2,2);
		\draw[line width=1pt] (0,3)--++(1,0)--++(2,2);
		\draw[line width=1pt] (2,0)--++(0,1);
		\draw[line width=1pt] (4,0)--++(0,1);
		\draw[line width=1pt] (1,2)--++(0,1);
		\end{tikzpicture} & \begin{tikzpicture}[line cap=round,line join=round,>=triangle 45,x=0.6cm,y=0.6cm]
		\draw[line width=1pt] (0,1)--++(1,0)--++(1,1)--++(1,-1)--++(1,0)--++(3,3);
		\draw[line width=1pt] (0,4)--++(1,0)--++(1,-1)--++(2,2);
		\draw[line width=1pt] (0,5)--++(1,0)--++(2,2);
		\draw[line width=1pt] (1,0)--++(0,1);
		\draw[line width=1pt] (3,0)--++(0,1);
		\draw[line width=1pt] (4,0)--++(0,1);
		\draw[line width=1pt] (2,2)-- node[midway,right] {$2$} ++(0,1);
		\draw[line width=1pt] (1,4)--++(0,1);
		\end{tikzpicture} & \begin{tikzpicture}[line cap=round,line join=round,>=triangle 45,x=0.6cm,y=0.6cm]
		\draw[line width=1pt] (0,1)--++(2,0)--++(1,1)--++(1,0)--++(2,2);
		\draw[line width=1pt] (0,2)--++(1,0)--++(1,1)--++(1,0)--++(2,2);
		\draw[line width=1pt] (0,4)--++(2,0)--++(2,2);
		\draw[line width=1pt] (1,0)--++(0,2);
		\draw[line width=1pt] (2,0)--++(0,1);
		\draw[line width=1pt] (4,0)--++(0,2);
		\draw[line width=1pt] (3,2)--++(0,1);
		\draw[line width=1pt] (2,3)--++(0,1);
		\end{tikzpicture} \\
		$2L$ & $3L$ & $3L$ \\
		\begin{tikzpicture}[line cap=round,line join=round,>=triangle 45,x=0.6cm,y=0.6cm]
		\draw[line width=1pt] (0,0)--++(2,2)--++(1,0)--++(2,2);
		\draw[line width=1pt] (0,3)--++(2,0)--++(2,2);
		\draw[line width=1pt] (3,0)--++(0,2);
		\draw[line width=1pt] (2,2)--++(0,1);
		\end{tikzpicture} & \begin{tikzpicture}[line cap=round,line join=round,>=triangle 45,x=0.6cm,y=0.6cm]
		\draw[line width=1pt] (0,0)--++(3,3)--++(1,0)--++(2,2);
		\draw[line width=1pt] (0,1)--++(3,0)--++(1,1)--++(1,0)--++(2,2);
		\draw[line width=1pt] (0,4)--++(3,0)--++(2,2);
		\draw[line width=1pt] (3,0)--++(0,1);
		\draw[line width=1pt] (5,0)--++(0,2);
		\draw[line width=1pt] (4,2)--++(0,1);
		\draw[line width=1pt] (3,3)--++(0,1);
		\end{tikzpicture} & \begin{tikzpicture}[line cap=round,line join=round,>=triangle 45,x=0.6cm,y=0.6cm]
		\draw[line width=1pt] (0,0)--++(2,2)--++(1,0)--++(1,-1)--++(2,0);
		\draw[line width=1pt] (0,3)--++(3,0)--++(2,2);
		\draw[line width=1pt] (0,4)--++(2,0)--++(2,2);
		\draw[line width=1pt] (4,0)--++(0,1);
		\draw[line width=1pt] (2,2)--++(0,2);
		\draw[line width=1pt] (3,2)--++(0,1);
		\end{tikzpicture} \\
		$2L-E$ & $3L-E$ & $3L-E_1-E_2$ \\
		\end{tabular}
		\caption{\label{figure examples of tropical curves}Examples of tropical curves inside $\RR^2$ along with their degree.}
		\end{center}
		\end{figure}
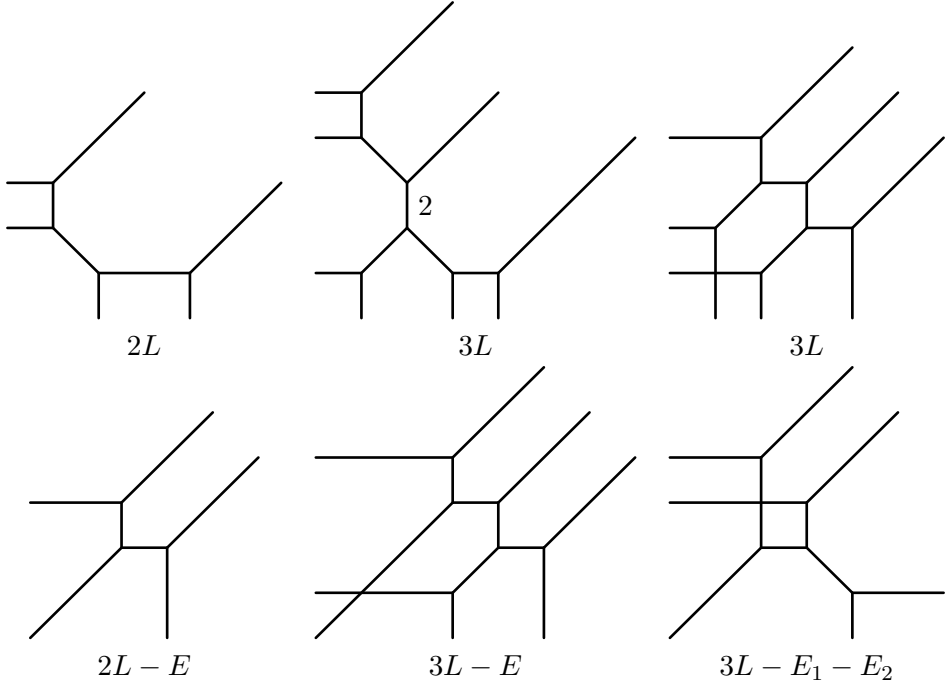				
		
		\begin{expl}
		In Figure \ref{figure examples of tropical curves} we can see several examples of tropical curves along with their class. The ``$2$" next to the edge of the second curve denotes the weight of the edge when it is not equal to $1$.
		\end{expl}

	\subsection{Moduli space of rational tropical curves and maps}
	
	In this section, we recall the definition of the moduli space of rational tropical curves, along with the definition of $\psi$-cycles and the maps defined on these moduli spaces.
	
		\subsubsection{Moduli spaces of tropical curves.} The moduli space of abstract rational tropical curves with $m$ ends is denoted by $\M_{0,m}$. It is introduced in \cite{mikhalkin2007moduli}, and constructed as follows: to each combinatorial type corresponds an orthant $\RR_{\geqslant 0}^{|E|}$, where $E$ is the set of bounded edges. The coordinates correspond to the lengths that we assign to the edges of the graph. The orthants are then glued together, as the boundary of an orthant corresponds to the graph where we contract the edges of length $0$. The space $\M_{0,m}$ is an abstract cone complex that can be embedded inside some vector space, endowing it with a lattice structure that makes it into a \textit{balanced fan} \cite{gathmann2009tropical}.
		
		The fan $\M_{0,m}$ is a fan of pure dimension $m-3$. The top-dimensional cones are in bijection with the combinatorial types for which the underlying graph is trivalent: each vertex has exactly three neighbours.
		
		\begin{expl}
		The moduli space $\M_{0,4}$ (with ends labelled by $[\![1;4]\!]$) is of pure dimension $1$, it has exactly three top-dimensional cones corresponding to the three trivalent graphs with four ends: $12//34$, $13//24$ and $14//23$. The apex of these three cones corresponds to the graph with a unique vertex adjacent to four ends.  The coordinate on each of the cone is the length of the unique bounded edge.
		\end{expl}
		
		The moduli space of parametrized rational tropical curves of degree $\Delta$ and $n$ marked points is denoted by $\M_{0,n}(\RR^2,\Delta)$. It is in bijection  with $\RR^2\times\M_{0,n+|\Delta|}$: a map $h:\Gamma\to\RR^2$ is uniquely determined by the data of $\Gamma$, the slopes of the ends (\textit{i.e.} $\Delta$), and the image of one chosen vertex of $\Gamma$ inside $\RR^2$ \cite{gathmann2009tropical}. The splitting as $\RR^2\times\M_{0,n+|\Delta|}$ is not unique, as it depends on the choice of a vertex.

		\subsubsection{Tropical $\psi$-classes.} Tropical $\psi$-classes on the moduli space of abstract rational tropical curves $\M_{0,m}$ have been introduced in \cite{mikhalkin2007moduli}, their intersections have been computed in \cite{kerber2009intersecting}. We refer to \cite{MR08} for a more thorough introduction on tropical $\psi$-classes, where the definition is generalized to the moduli space of parametrized tropical curves $\M_{0,n}(\RR^2,\Delta)$, using the tropical intersection theory methods from \cite{allermann2010first}. Each monomial $\psi_1^{k_1}\cdots\psi_n^{k_n}$ defines a weighted subfan of $\M_{0,n}(\RR^2,\Delta)$, described as follows in \cite{kerber2009intersecting, MR08}.
		
		\begin{defi}
		The subfan $\M_{0,n}(\RR^2,\Delta)\cap\psi_1^{k_1}\cdots\psi_n^{k_n}$ is the subfan of pure dimension $|\Delta|-1-\sum k_i+n$  containing as top-dimensional cones the ones corresponding to the combinatorial types satisfying the following conditions: for each vertex $V$, let $I_V\subset[\![1;n]\!]$ be the subset of marked points adjacent to $V$,
		$$\val(V)=3+\sum_{i\in I_V}k_i,$$
		with a weight equal to
		$$\frac{\prod_V\left( \sum_{i\in I_V}k_i\right)!}{\prod_i k_i!}.$$
		\end{defi}
		
		\begin{expl}
		The cones in the fan $\M_{0,n}(\RR^2,\Delta)\cap\psi_1^{k}$ correspond to the graphs where the vertex adjacent to the first marked point has valency $k+3$. As the end is contracted, we only see a $(k+2)$-valent vertex inside $\RR^2$. All the weights are $1$. See for instance the first curve on Figure \ref{figure example tropical curve psi constraint}.
		
		The cones in the fan $\M_{0,n}(\RR^2,\Delta)\cap\psi_1^{k_1}\psi_2^{k_2}$ correspond to graphs with marked points $1$ and $2$ adjacent to different vertices of respective valency $k_1+3$ and $k_2+3$, but also the one where $1$ and $2$ are adjacent to the same vertex of valency $k_1+k_2+3$. Cones of the latter form are counted with a weight $\frac{(k_1+k_2)!}{k_1!k_2!}$. See the remaining curves on Figure \ref{figure example tropical curve psi constraint}.
		\end{expl}
		
		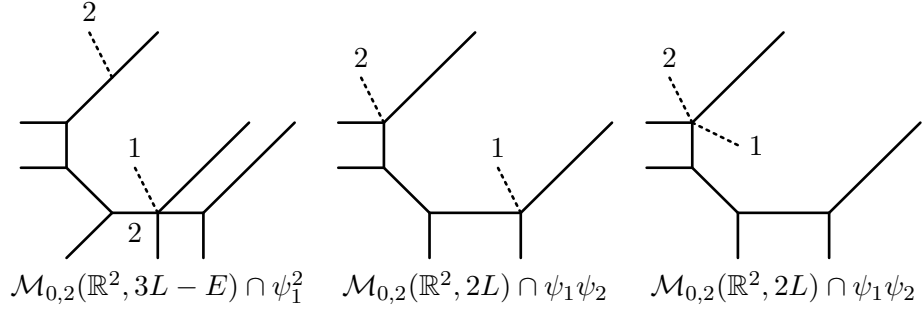
\begin{figure}
		\begin{center}
		\begin{tabular}{ccc}
		\begin{tikzpicture}[line cap=round,line join=round,>=triangle 45,x=0.6cm,y=0.6cm]
		\draw[line width=1pt] (1,0)--++(1,1)-- node[midway,below] {$2$} ++(1,0)--++(1,0)--++(2,2);
		\draw[line width=1pt] (3,0)--++(0,1)--++(2,2);
		\draw[line width=1pt] (2,1)--++(-1,1)--++(0,1)--++(2,2);
		\draw[line width=1pt] (4,0)--++(0,1);
		\draw[line width=1pt] (0,2)--++(1,0);
		\draw[line width=1pt] (0,3)--++(1,0);
		
		\draw[line width=1pt,dotted] (3,1)--++(-0.5,1) node[above] {$1$};
		\draw[line width=1pt,dotted] (2,4)--++(-0.5,1) node[above] {$2$};
		\end{tikzpicture} & \begin{tikzpicture}[line cap=round,line join=round,>=triangle 45,x=0.6cm,y=0.6cm]
		\draw[line width=1pt] (0,2)--++(1,0)--++(1,-1)--++(2,0)--++(2,2);
		\draw[line width=1pt] (0,3)--++(1,0)--++(2,2);
		\draw[line width=1pt] (2,0)--++(0,1);
		\draw[line width=1pt] (4,0)--++(0,1);
		\draw[line width=1pt] (1,2)--++(0,1);
		\draw[line width=1pt,dotted] (4,1)--++(-0.5,1) node[above] {$1$};
		\draw[line width=1pt,dotted] (1,3)--++(-0.5,1) node[above] {$2$};
		\end{tikzpicture} & \begin{tikzpicture}[line cap=round,line join=round,>=triangle 45,x=0.6cm,y=0.6cm]
		\draw[line width=1pt] (0,2)--++(1,0)--++(1,-1)--++(2,0)--++(2,2);
		\draw[line width=1pt] (0,3)--++(1,0)--++(2,2);
		\draw[line width=1pt] (2,0)--++(0,1);
		\draw[line width=1pt] (4,0)--++(0,1);
		\draw[line width=1pt] (1,2)--++(0,1);
		\draw[line width=1pt,dotted] (1,3)--++(1,-0.5) node[right] {$1$};
		\draw[line width=1pt,dotted] (1,3)--++(-0.5,1) node[above] {$2$};
		\end{tikzpicture} \\
		$\M_{0,2}(\RR^2,3L-E)\cap\psi_1^2$ & $\M_{0,2}(\RR^2,2L)\cap\psi_1\psi_2$ & $\M_{0,2}(\RR^2,2L)\cap\psi_1\psi_2$ \\
		\end{tabular}
		\caption{\label{figure example tropical curve psi constraint}Tropical curves satisfying some $\psi$-constraint.}
		\end{center}
		\end{figure}
		
		Results from \cite{kerber2009intersecting} assert that the subfan defined by $\psi_1^{k_1}\cdots\psi_n^{k_n}$ is balanced in the sense of \cite{allermann2010first}.

		\subsubsection{Forgetful maps.} We have well-defined maps between the moduli spaces of tropical curves. The first are obtained by forgetting some of the ends of the curve, and then \textit{stabilizing}, which means delete the univalent and bivalent vertices that may appear through this deletion. For any subset $I$ of $[\![1;m]\!]$ of size $m'$, we have a map to $\M_{0,m'}$ only remembering the ends indexed by $I$:
		$$\mathrm{ft}_I:\M_{0,m}\to\M_{0,m'}.$$
		
		We also have forgetful maps $\M_{0,n}(\RR^2,\Delta)\to\M_{0,n+|\Delta|}$ obtained by forgetting the map to $\RR^2$. It is then possible to compose with a forgetful map to some $\M_{0,m}$.
		
		\begin{expl}
		We especially need the forgetful maps of the following form
		$$\mathrm{ft}_4:\M_{0,n}(\RR^2,\Delta)\to\M_{0,4},$$
		obtained by only remembering the relative position of four ends. Such a map is referred as \textit{tropical cross-ratio}. Concretely, the cross-ratio is just the length of the path between the ends not forgotten by the map. On Figure \ref{figure example of cross-ratio and forgetful map} we see a curve in $\M_{0,4}(\RR^2,3L)\cap\psi_1\psi_2$. The image by the forgetful map is depicted in red and blue. The combinatorial type in $\M_{0,4}$ is $12//34$, and the coordinate along the ray is the length of the blue path. Notice that for the blue edge of weight $2$, the slope of $h$ is $(0,2)$. Thus, the length contributing to the cross-ratio is not the length appearing inside $\RR^2$ but half of it.
		\end{expl}
		
		\begin{figure}
		\begin{center}
		\begin{tikzpicture}[line cap=round,line join=round,>=triangle 45,x=0.6cm,y=0.6cm]
		\draw[line width=1pt] (-1,1)--++(2,0); \draw[line width=1pt,red] (1,1)--++(1,1)--++(1,-1)--++(1,0); \draw[line width=1pt] (4,1)--++(3,3);
		\draw[line width=1pt] (-1,4)--++(1,0); \draw[line width=1pt,red] (0,4)--++(1,0);\draw[line width=1pt,blue] (1,4)--++(1,-1);\draw[line width=1pt] (2,3)--++(2,2);
		\draw[line width=1pt] (-1,5)--++(2,0)++(1,1)--++(1,1); \draw[line width=1pt,red] (1,5)--++(1,1);
		\draw[line width=1pt] (1,0)--++(0,1);
		\draw[line width=1pt] (3,0)--++(0,1);
		\draw[line width=1pt] (4,0)--++(0,1);
		\draw[line width=1pt,blue] (2,2)-- node[midway,right] {$2$} ++(0,1);
		\draw[line width=1pt,red] (1,4)--++(0,1);
		\draw[line width=1pt,dotted,red] (1,1)--++(-0.5,1) node[above] {$1$};
		\draw[line width=1pt,dotted,red] (4,1)--++(-0.5,1) node[above] {$2$};
		\draw[line width=1pt,dotted,red] (2,6)--++(-0.5,1) node[above] {$3$};
		\draw[line width=1pt,dotted,red] (0,4)--++(-0.5,-1) node[below] {$4$};
		\end{tikzpicture}
		\caption{\label{figure example of cross-ratio and forgetful map}Tropical cross-ratio between four marked points on a tropical curve.}
		\end{center}
		\end{figure}

		\subsubsection{Evaluation maps.} Beside the forgetful maps, we also have maps defined on the moduli space of parametrized tropical curves that evaluate the position of the ends and marked points: for each end $e$, we have
		$$\begin{array}{rccl}
\mathrm{ev}_e:& \M_{0,n}(\RR^2,\Delta) & \longrightarrow & \RR^2/\ang{n_e} \\
 & (h,\Gamma) & \longmapsto & h(e) \\		
		\end{array},$$
		where $w_e n_e$ is the slope of the end $e$, so that $h(e)$ is well-defined as it does not depend on the choice of point inside $e$ that we evaluate. In particular, for a marked point, \textit{i.e.} $n_e=0$, the evaluation map takes values inside $\RR^2$. The evaluation maps and forgetful maps are linear map on any cone of $\M_{0,n}(\RR^2,\Delta)$.
		
		If $n_e=0$, the codomain $\RR^2$ contains the lattice $\ZZ^2$. If $n_e\neq 0$, the quotient space $\RR^2/\ang{n_e}$ also contains the lattice $\ZZ^2/\ang{n_e}$. An isomorphism with $\ZZ\subset\RR$ is given by taking $\det(n_e,-):\RR^2/\ang{n_e}\to\RR$.

\section{Enumerative problems and enumerative invariants}\label{sec-enumerative}

	\subsection{Tropical enumerative problems appearing in our recursions}\label{subsec-enum}
	
	We recall several tropical enumerative problems that lead to different enumerative invariants. For each of them, we have two slightly different points of view on the same approach. First, we can view the problem enumeratively as counting tropical curves satisfying a number of constraints, with a suitable multiplicity so that the count does not depend on the choice of the constraints. As noticed in \cite{allermann2010first}, the invariant is in fact equal to the intersection number between the image of the moduli space of curves by an adhoc evaluation map, and a cycle $\Xi$ corresponding to the constraints. The multiplicity of a tropical curve is then equal to the intersection index at the point of the moduli space corresponding to the curve.
	
	The examples in this section introduce the notations for the various enumerative invariants needed for the statement of the recursive formulas and computations of descendant invariants with line constraints.

		\subsubsection{Invariants with point constraints.}\label{subsec-enumpoints}
		
		We consider degree $\Delta\subset\ZZ^2$ rational tropical curves with $n$ marked points. We consider the evaluation map:
		$$\mathrm{ev}:\M_{0,n}(\RR^2,\Delta)\to(\RR^2)^n,$$
		that evaluates the position of the marked points and ends. For $n=|\Delta|-1$, let $\Xi=\prod_1^n \{P_i\}$ be a cycle, where $(P_i)$ is a generic configuration of points in $\RR^2$. As domain and codomain share the same dimension, and $\Xi$ is $0$-dimensional, the image and $\Xi$ are of complementary dimension. We then intersect the image of the evaluation map with $\Xi$ to get an invariant, which amounts to look at the preimages under $\mathrm{ev}$. It corresponds to a count of degree $\Delta$ rational curves passing through a generic configuration of $|\Delta|-1$ points inside $\RR^2$, with a multiplicity. We speak about $P$-constraint. If $\Xi$ is chosen generically, the curves are trivalent rational tropical curves, and the multiplicity $m_\Gamma$ is then as given by the correspondence theorem from \cite{mikhalkin2005enumerative}:
		$$m_\Gamma=\prod_V m_V,$$
		where $m_V=|\det(a_V,b_V)|$, and $a_V$, $b_V$ are the outgoing slopes of two of the edges adjacent to $V$. The enumerative invariant counting the tropical curves in $\M_{0,n}(\RR^2,\Delta)$ which are in the inverse image of $\Xi$ under $\mathrm{ev}$ with multiplicity is denoted by $N_\Delta$.
		
		
		\begin{rem}
		By only caring about the marked points and not the ends of non-zero slope, we usually forget about their labelling: the enumerative invariant  is the intersection number divided by the number of labellings of the ends. If $\Delta=\{n_1^{d_1},\cdots,n_r^{d_r}\}$, this number is equal to $\prod_{i=1}^r d_i !$. For instance, for $\Delta_d$, we divide by $(d!)^3$.
		\end{rem}
		
		\begin{expl}
		For $\Delta_d$, we count rational curves of degree $d$ passing through $3d-1$ points in generic position, this number is denoted by $N_d$. For instance, after the division by $(d!)^3$,
		$$N_1=N_2=1,\ N_3=12.$$
		\end{expl}
		
		\begin{expl}
		For $\Delta_{d,k}$, we count curves of degree $dL-kE$, passing through $3d-k-1$ points in generic position. The number is denoted by $N_{dL-kE}$. For instance,
		$$N_{2L-E}=1,\ N_{3L-E}=12,\ N_{3L-2E}=1.$$
		\end{expl}
		
		\begin{expl}
		For $\Delta_{d,k_1,k_2}$, we count curves of degree $dL-k_1E_1-k_2 E_2$ passing through $3d-k_1-k_2-1$ points in generic position. The number is denoted by $N_{dL-k_1E_1-k_2 E_2}$. For instance,
		$$N_{3L-E_1-E_2}=12,\ N_{3L-2E_1-E_2}=1.$$
		\end{expl}
		
		\subsubsection{Variation on the point constraints.} We also have the following variant of the enumerative problem where some of the marked points have to lie on $1$-dimensional cycles inside $\RR^2$. In that case, we choose $n=|\Delta|-1+r$, $C_1,\dots,C_r$ tropical curves inside $\RR^2$, and we take the cycle of constraints to be
		$$\Xi=\prod_1^{|\Delta|-1}\{P_i\}\times\prod_1^r C_i,$$
		which amounts to count tropical curves passing through $|\Delta|-1$ points in generic position, and intersect each of the curves $C_i$ somewhere. The solutions are easily obtained from the curves passing through $(P_i)$, as the remaining marked points just have to be chosen at the intersection points between each solution to the point problem, and the curves $C_i$. For each choice of intersection points, the multiplicity is equal to $\prod m_V$ times the product of intersection indices at the chosen intersection points. In the end, the enumerative invariant is equal to
		$$N_\Delta\cdot \prod_1^r (\Delta\cdot C_i) .$$
		
		\subsubsection{Relative invariants.} The following problem is another variation on the preceding one where we now care on the position of some of the unbounded ends. Let $I\subset\Delta$ a subset of ends, we take $n=|\Delta|-|I|-1$ and
		$$\mathrm{ev}:\M_{0,n}(\RR^2,\Delta)\to(\RR^2)^n\times\prod_I \RR^2/\ang{n_i},$$
	$$\Xi=\prod_1^n \{P_i\}\times\prod_I \{\mu_i\}.$$
	It amounts to count curves satisfying the point constraints and having ends indexed by $I$ fixed. The multiplicity is also given by
	$$m_\Gamma=\frac{1}{\prod w_i}\prod_V m_V,$$
	where the product on the denominator is the product of the weights of the ends whose position is evaluated. This correction factor to the multiplicity, which only depends on the degree and constraints (\textit{i.e.} not on the curves) and already appears in \cite{gross2010tropical} and \cite{gathmann2007caporaso}, is necessary to have a correspondence statement when dealing with tangency conditions. In particular, we also get a tropical invariant if we use the multiplicity $m_\Gamma=\prod m_V$ instead, which is the convention we adopt in the rest of the paper. We now introduce two notations which will be useful in the statement of the recursive formula from Theorem \ref{theorem recursive formula r=1}.
	
	\begin{defi} We set
	\begin{itemize}[label=-]
	\item Let $N_d(w)$ be the number of tropical curves of degree $\{(1,1)^d,(-1,0)^d,(0,-1)^{d-w},(0,-w)\}$ passing through a generic configuration of $3d-w$ points in generic position with multiplicity $m_\Gamma=\prod m_V$, with the choice of a bottom end of weight $w$.
	\item Let $\widetilde{N}_d(w)$ be the number of tropical curves of degree $\{(1,1)^d,(-1,0)^d,(0,-1)^{d-w},(0,-w)\}$ passing through a generic configuration of $3d-w-1$ points in generic position and a fixed end of weight $w$, with multiplicity $m_\Gamma=\prod m_V$.
	\end{itemize}
	\end{defi}
	
	This definition also naturally extends to tropical degrees which are not just $dL$ but the other degrees as well.
	
	\begin{rem}
	For $w\neq 1$, $N_d(w)$ is just a usual relative invariants computed by the Caporaso-Harris formula from \cite{gathmann2007caporaso}, while $\widetilde{N}_d(w)$ differs from the usual relative invariant by a factor $w$, since we dropped the denominator from the multiplicity in our convention. For $w=1$, the only difference is that here we choose a bottom end for the computation of $N_d(1)$, of which there are $d$, so that $N_d(1)=d\cdot N_d$.
	\end{rem}
	
	\begin{expl}
	By counting curves of degree $2$ but having a bottom end of weight $2$ instead of two ends of weight $1$, we have $N_2(2)=2$. The usual relative invariant counting conics passing through $3$ points and a fixed tangency to a line would be $1$, but due to our choice of convention deleting the denominator, we get $\widetilde{N}_2(2)=2$. Similarly, by counting curves of degree $3$ having bottom ends of weight $1$ and $2$ instead of $1$ three times, we can compute $N_3(2)=22$ and $\widetilde{N}_3(2)=20$.
	\end{expl}
	
	\begin{rem}
	It is also possible to mix this situation with the other variant by adding tropical curves constraints.
	\end{rem}
		
		\subsubsection{Descendant invariants with point insertions.} We finish by considering the evaluation map restricted to the subfan of $\M_{0,n}(\RR^2,\Delta)\cap\psi_1^{k_1}\cdots\psi_n^{k_n}$, which is of pure dimension $|\Delta|+n-1-\sum k_i$. We still have the evaluation map:
		$$\mathrm{ev}:\M_{0,n}(\RR^2,\Delta)\cap\psi_1^{k_1}\cdots\psi_n^{k_n}\longrightarrow (\RR^2)^n.$$
		If $n=|\Delta|-1-\sum k_i$, we can consider the cycle of constraints
	$$\Xi=\prod_1^n \{ P_i\},$$		
		which is of complementary dimension (here $0$). A point constraint coupled to a $\psi$-class constraint is called a $\psi^k P$-constraint: the tropical curves in the inverse image of the evaluation map have a vertex of fixed valency $k+3$ (counting the contracted marked end) at a specific point in $\RR^2$. By genericity of $\Xi$, no cone corresponding to curves having more than one marked point adjacent to the same vertex can contribute a solution, since marked points are mapped to different points inside $\RR^2$. The enumerative problem amounts to look for rational curves having vertices of prescribed valency at fixed points inside $\RR^2$. The multiplicity of a tropical curve is equal to
		$$m_\Gamma=\prod_V m_V,$$
		where the product is over the trivalent vertices not adjacent to any marking. The invariant is denoted by
		$$\ang{\psi^{k_1}P,\cdots,\psi^{k_r}P}_\Delta,$$
		where we forget about the point constraints (\textit{i.e.} $k_i=0$): we assume that there are the right number of point constraints so that the dimension count makes sense.
		
		\begin{expl}
		The notation $\ang{\psi^k P}_d$ counts degree $d$ curves with a $\psi^k P$-constraint passing through an additional $3d-2-k$ points in generic position.
		
		The notation $\ang{\psi^{k_1} P,\psi^{k_2} P}_d$ counts degree $d$ curves with a $\psi^{k_1} P$-constraint, a $\psi^{k_2}P$-constraint passing through an additional $3d-2-k_1-k_2$ points in generic position.
		\end{expl}
		
		These invariants with only $\psi^k P$-constraints have already been considered in \cite{MR08}, where a recursive formula of WDVV type and a generalized lattice path algorithm is provided. In \cite{BGM10}, a floor diagram algorithm to compute them has been developed.

	\begin{rem}
	Notice that due to the higher-valency conditions imposed by the $\psi$-classes, it is now possible to get parametrized tropical curves having non-trivial automorphism group. For instance, two ends adjacent to the same vertex can be exchanged. Thus, by forgetting about the labelling of the ends, \textit{i.e.} dividing by $(d!)^3$ in the case of degree $d$ curves, we count the tropical curves $h:\Gamma\to\RR^2$ with a coefficient equal to $\frac{1}{|\mathrm{Aut}\Gamma|}$, because there are not always $(d!)^3$ different possible labellings anymore.
	\end{rem}
	
	\begin{figure}
	\begin{center}
	\begin{tabular}{cc}
	\begin{tikzpicture}[line cap=round,line join=round,>=triangle 45,x=0.6cm,y=0.6cm]
		\draw[line width=1pt] (0,2)--++(1,0)--++(1,-1)--++(2,2);
		\draw[line width=1pt] (0,3)--++(1,0)--++(2,2);
		\draw[line width=1pt] (2,0)--++(0,1);
		\draw[line width=1pt] (2.1,0)--++(0,-1);
		\draw[line width=1pt] (1.9,0)--++(0,-1);
		\draw[line width=1pt] (1,2)--++(0,1);
		\draw (2,0) node {$\bullet$};
		\draw (1.5,1.5) node {$+$};
		\draw (1,2.5) node {$\times$};
		\draw (2,4) node {$+$};
		\end{tikzpicture} & \begin{tikzpicture}[line cap=round,line join=round,>=triangle 45,x=0.6cm,y=0.6cm]
		\draw[line width=1pt] (0,2)--++(1,0)--++(1,-1)--++(2,0)--++(2,2);
		\draw[line width=1pt] (0,3)--++(1,0)--++(2,2);
		\draw[line width=1pt] (2,0)--++(0,1);
		\draw[line width=1pt] (4,0)--++(0,1);
		\draw[line width=1pt] (1,2)--++(0,1);
		\draw (1,3) node {$\bullet$};
		\draw (2,0.5) node {$\times$};
		\draw (3,1) node {$\times$};
		\draw (5,2) node {$+$};
		\end{tikzpicture} \\
	\end{tabular}
	\caption{\label{figure curve with automorphism}Tropical curves contributing for $\ang{\psi P}_2$ for two different choices of point constraints.}
	\end{center}
	\end{figure}
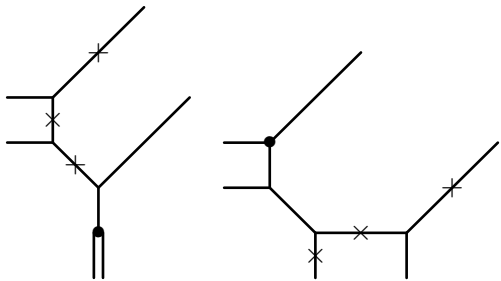
	
	\begin{expl}
	In Figure \ref{figure curve with automorphism} we can see the unique degree $2$ tropical curve satisfying three $P$-constraints and a unique $\psi P$-constraint, for two different choices of point constraints. The first curve has a vertex of multiplicity $2$, but also an automorphism since the two bottom ends can be exchanged. Thus, each of them has multiplicity $1$ and we get that $\ang{\psi P}_2=1$.
	\end{expl}

	\subsection{Descendant invariants with line insertions}\label{subsec-tropicaldesc}

We now consider the descendant invariants where we authorize the $\psi$-classes to be coupled with line constraints, line meaning the standard tropical line inside $\RR^2$. We call it $\psi^k L$-constraints. It means that we consider the evaluation map
$$\mathrm{ev}:\M_{0,n}(\RR^2,d)\cap\psi_1^{k_1}\cdots\psi_n^{k_n}\to(\RR^2)^n,$$
and we intersect it with a cycle of the form
$$\Xi=\prod_1^n \Xi_i,$$
where $\Xi_i\subset\RR^2$ is either a point or a tropical line. We assume the cycle is of complementary dimension, so that the intersection number yields an enumerative invariant. This invariant is usually denoted by $\ang{\tau_{k_1}(d_1)\cdots\tau_{k_n}(d_n)}^{\log}_{0,d}$, where $d_i=\mathrm{codim}\Xi_i$, but with our notations, it is denoted by
$$\ang{\psi^{k_1}\Xi_1,\cdots,\psi^{k_n}\Xi_n}_d,$$
where $\Xi_i$ is either $L$ or $P$.

We now explain how to compute the multiplicity of a tropical curve that satisfies the constraints: it belongs to $\M_{0,n}(\RR^2,d)\cap\psi_1^{k_1}\cdots\psi_n^{k_n}$, and the $i$-th marked point is mapped to the cycle $\Xi_i$, which is a point $P_i$ or a tropical line $L_i$ inside $\RR^2$. Let $\Xi=\prod_i\Xi_i\subset(\RR^2)^n$ be a generic cycle of constraints. For simplicity, we assume that every $\psi$-constraint is coupled with a line, so that we do not have any $\psi^k P$ -constraint with $k\geqslant 1$. The statements on the relative position of marked points on the curve can easily be adapted to the setting with $\psi^k P$-constraints, and the multiplicities are changed by making the product over vertices going only over the trivalent vertices not adjacent to any marking.

\subsubsection{Case with one line.} \label{subsec-caseoneline} Assume there is only one line constraint coupled with a $\psi$-condition: we consider $\ang{\psi^k L}_d$. We have the evaluation map
$$\mathrm{ev}:\M_{0,n+1}(\RR^2,d)\cap\psi_1^k\to(\RR^2)^n\times\RR^2,$$
where the first $n$ marked points are $P$-constraints, and the last marked point is coupled to the $\psi^k L$-constraint: $\Xi=\prod_1^n\{P_i\}\times L$. By dimension count, we have $n+k=3d-1$.

\begin{lem}\label{lemma position points k=1}
In the above setting, let $h:\Gamma\to\RR^2$ be a tropical solution to the problem. Let $\Gamma\backslash\P$ be the complement of marked points. There is a unique bounded component, the $\psi^k L$-marking (\textit{i.e.} the marked point subject to the $\psi^k L$-constraint) is adjacent to it. Every other component contains a unique end.
\end{lem}

\begin{proof}
By genericity of the constraints, the marked point associated to the $\psi^kL$-constraint has valency $k+2$ ($k+3$ including the end to which it corresponds), so that it disconnects $\Gamma$ into $k+2$ components. Every other marked point is on a trivalent vertex, i.e.\ its image lies on the interior of an edge, so that it increases by $1$ the number of connected components. In the end, we get $k+2+n=3d+1$ connected components.

A connected component cannot contain more than two ends, otherwise it is possible to deform the path linking one to the other, yielding a $1$-parameter family of solutions, which is impossible as $\Xi$ is of complementary dimension and has been chosen generically. Hence, we get exactly $3d$ connected components containing each an end. The remaining component has to be bounded.

If the $\psi^k L$-marking was not adjacent to the bounded component, only $P$-constraints would lie on its boundary. This is impossible by the genericity assumption: let $P_i\in\RR^2$ be the marked points on the boundary of the component, and $n_i$ the outing slopes of the edges on which they lie. By tropical Menelaus theorem \cite{mikhalkin2017quantum}, we must have
$$\sum_i \det(n_i,P_i)=0.$$
As the choice of $\Xi$ is generic, this equation is not satisfied by the $P_i$. So the $\psi^kL$-constraint is on the boundary of the unique bounded component.
\end{proof}

\begin{figure}
\begin{center}
\begin{tikzpicture}[line cap=round,line join=round,>=triangle 45,x=0.6cm,y=0.6cm]
		\draw[line width=1pt,blue] (3,1)--++(2,2)++(-2,-2)--++(0,-1)++(0,1)--++(-3,0);

		\draw[line width=1pt] (0,2)--++(1,0)--++(1,-1)--++(2,0)--++(2,2);
		\draw[line width=1pt] (0,3)--++(1,0)--++(2,2);
		\draw[line width=1pt] (2,0)--++(0,1);
		\draw[line width=1pt] (4,0)--++(0,1);
		\draw[line width=1pt] (1,2)--++(0,1);
		\draw (2,1) node {$\bullet$};
		\draw (0.5,2) node {$\times$};
		\draw (1,2.5) node {$\times$};
		\draw (5,2) node {$+$};
		\draw (1.5,3.5) node {$+$};
		\end{tikzpicture}
	\caption{\label{figure example position of marked points k=1}Tropical curve contributing to $\ang{\psi L}_2$.}
\end{center}
\end{figure}
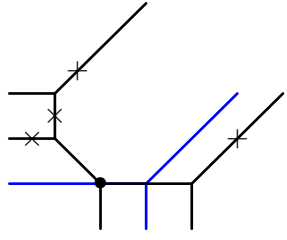

\begin{expl}
On Figure \ref{figure example position of marked points k=1} we have a degree $2$ curve passing through $4$ points with an additional line constraint. The line of the $\psi L$-constraint is drawn in blue. We can see that there is a unique bounded component in the complement of marked points.
\end{expl}

The multiplicity is given by the following proposition.

\begin{prop}\label{prop-mult1}
Let $h:\Gamma\to\RR^2$ be a solution for the tropical enumerative problem defined at the beginning of Subsection \ref{subsec-caseoneline}, let $V$ be the vertex adjacent to the $\psi^k L$-marking, let $u$ be the slope of the line $L$ at the marked point, and let $v$ be the slope of $h$ on the edge adjacent to $V$ in the unique bounded component. Then
$$m_\Gamma=|\det(u,v)|\prod_V m_V,$$
where $V$ goes over the trivalent unmarked vertices.
\end{prop}

\begin{proof}
Let $h:\Gamma\to\RR^2$ be a solution to the enumerative problem. By definition, the multiplicity is the intersection index between the image of the evaluation map and $\Xi$ at $(h,\Gamma)$. As $\Xi=\{P_1\}\times\cdots\times\{P_n\}\times L$, near $(h,\Gamma)$, $\Xi$ is just a (usual) line directed by the vector $(0,\cdots,0,u)\in(\RR^2)^n\times\RR^2$. Projecting along this direction, the intersection index is the determinant of the evaluation map
	$$\mathrm{ev}_\Gamma:\RR^{|E|}\to(\RR^2)^n\times\RR.$$
	We compute this determinant, proceeding as follows:
		\begin{itemize}[label=$\ast$]
		\item We split $\Gamma$ at the $P$-constraints and replace the $P$-constraint by a constraint on each appearing end. This does not change the determinant as moving points on a tropical plane curve leaves the multiplicity invariant, see e.g.\ \cite{gathmann2007numbers}.
		\item For each trivalent vertex adjacent to two ends whose position is fixed, we prune the vertex, which brings a factor of $m_V=|\det(a_V,b_V)|$, where $a_V$ and $b_V$ are the slopes of the fixed ends. In the evaluation matrix, such a trivalent vertex corresponds to a block of size $2\times 2$ having the determinant above, and as the determinant of a block matrix is the product of the determinants of the blocks, we can take out this factor and continue.
		\end{itemize}
		These two steps are enough to compute the multiplicity in the absence of $\psi^k L$-constraint, as every unbounded component gets completely pruned through the process. Due to the presence of the $\psi^k L$-constraint, we have two new problems: what happens for the bounded component, and what happens for the unbounded components adjacent to the $\psi^k L$-marking ? These are handled by the following new step.
		\begin{itemize}[label=$\ast$]
		\item If the $\psi^k L$-marking is adjacent to an end of slope $v$ whose position is fixed, we can prune the end, replacing every other adjacent edge by an unbounded end. This brings a factor of $|\det(v,u)|$ to the determinant, as can be checked by setting up the local evaluation matrix.
		\end{itemize}				
In the end, we get $|\det(v,u)|\prod m_V$.
\end{proof}

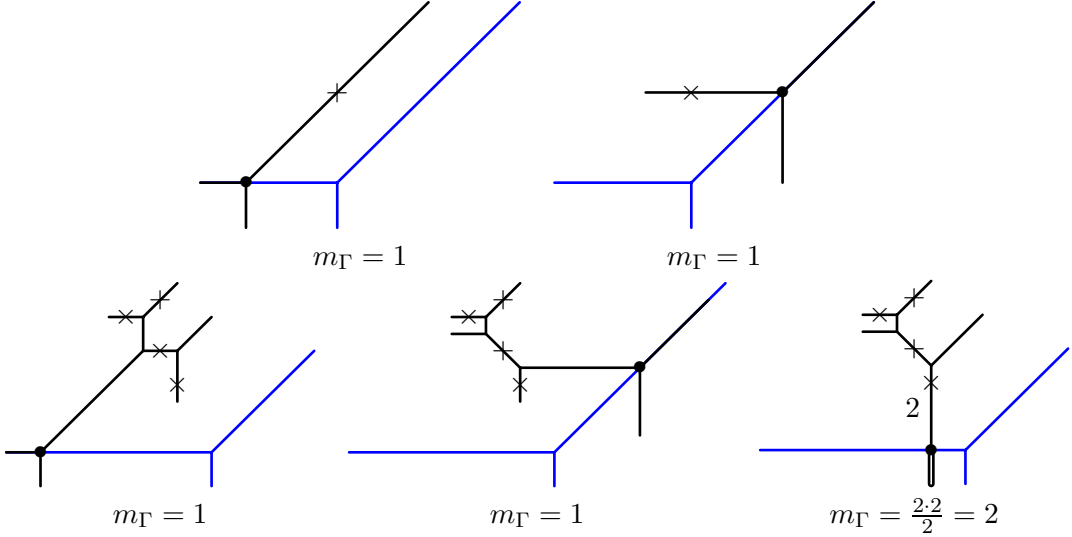
\begin{figure}
\begin{center}
\begin{tabular}{cc}
	\begin{tikzpicture}[line cap=round,line join=round,>=triangle 45,x=0.6cm,y=0.6cm]
		\draw[line width=1pt,blue] (3,1)--++(4,4)++(-4,-4)--++(0,-1)++(0,1)--++(-3,0);

		\draw[line width=1pt] (0,1)--++(1,0) node {$\bullet$} --++(0,-1)++(0,1)--++(4,4);
		\draw (3,3) node {$+$};
		\end{tikzpicture} & \begin{tikzpicture}[line cap=round,line join=round,>=triangle 45,x=0.6cm,y=0.6cm]
		\draw[line width=1pt,blue] (3,1)--++(4,4)++(-4,-4)--++(0,-1)++(0,1)--++(-3,0);

		\draw[line width=1pt] (2,3)--(5,3) node {$\bullet$} --++(0,-2)++(0,2)--++(2,2);
		\draw (3,3) node {$\times$};
		\end{tikzpicture} \\
		$m_\Gamma=1$ & $m_\Gamma=1$ \\
\end{tabular}

\begin{tabular}{ccc}
\begin{tikzpicture}[line cap=round,line join=round,>=triangle 45,x=0.45cm,y=0.45cm]
		\draw[line width=1pt,blue] (0,1)--++(6,0)++(0,-1)--++(0,1)--++(3,3);
		
		\draw[line width=1pt] (0,1)--++(1,0) node {$\bullet$} ++(0,-1)--++(0,1)--++(3,3)--++(1,0)--++(0,-1.5)++(0,1.5)--++(1,1);
		\draw[line width=1pt] (3,5)--++(1,0)--++(1,1)++(-1,-1)--++(0,-1);
		\draw (5,3) node {$\times$};
		\draw (4.5,4) node {$\times$};
		\draw (3.5,5) node {$\times$};
		\draw (4.5,5.5) node {$+$};
		\end{tikzpicture} & \begin{tikzpicture}[line cap=round,line join=round,>=triangle 45,x=0.45cm,y=0.45cm]
		\draw[line width=1pt,blue] (0,1)--++(6,0)++(0,-1)--++(0,1)--++(5,5);
		
		\draw[line width=1pt] (3,4.5)--++(1,0)--++(1,-1)--++(0,-1)++(0,1)--(8.5,3.5) node {$\bullet$} --++(2,2)++(-2,-2)--++(0,-2);
		\draw[line width=1pt] (3,5)--++(1,0)--++(1,1)++(-1,-1)--++(0,-0.5);
		\draw (5,3) node {$\times$};
		\draw (4.5,4) node {$+$};
		\draw (3.5,5) node {$\times$};
		\draw (4.5,5.5) node {$+$};
		\end{tikzpicture} & \begin{tikzpicture}[line cap=round,line join=round,>=triangle 45,x=0.45cm,y=0.45cm]
		\draw[line width=1pt,blue] (0,1)--++(6,0)++(0,-1)--++(0,1)--++(3,3);
		\draw[line width=1pt,double] (5,1)--++(0,-1);		
		\draw[line width=1pt] (3,4.5)--++(1,0)--++(1,-1)--++(1.5,1.5)++(-1.5,-1.5)--(5,1) node[midway,left] {$2$} node {$\bullet$};

		\draw[line width=1pt] (3,5)--++(1,0)--++(1,1)++(-1,-1)--++(0,-0.5);
		\draw (5,3) node {$\times$};
		\draw (4.5,4) node {$+$};
		\draw (3.5,5) node {$\times$};
		\draw (4.5,5.5) node {$+$};
		\end{tikzpicture} \\
		$m_\Gamma=1$ & $m_\Gamma=1$ & $m_\Gamma=\frac{2\cdot 2}{2}=2$ \\
\end{tabular}
\caption{\label{figure example computation r=1}In the first row, the two tropical lines contributing to $\ang{\psi L}_1$, in the second row the three conics contributing to $\ang{\psi L}_2$, see Example \ref{ex-Psiline}.}
\end{center}
\end{figure}

\begin{expl}\label{ex-Psiline}
In figure \ref{figure example computation r=1} we can see the curves contributing to the invariants $\ang{\psi L}_1$ and $\ang{\psi L}_2$. The lines in the first row both have multiplicity $1$ since there are no unmarked trivalent vertex, and in each case the determinant between the slope of the line $L$ (in blue) and the edge toward the bounded component is $1$. We get $\ang{\psi L}_1=2$.

For the conics, similarly, the first two conics have multiplicity $1$. The third conic on the contrary has multiplicity $2$: there is a vertex of multiplicity $2$, the determinant between the blue line and the bounded component is also $2$, but we have to account for the automorphism that exchanges the two bottom ends: $\frac{2\cdot 2}{2}=2$. we get $\ang{\psi L}_2=4$.
\end{expl}

We finish the case of a unique $\psi^k L$-constraint by stating the divisor equation, which allows to get rid of $\psi^0 L$-constraints.

\begin{lem}\label{lem-divisoreq}
We have the following identity:
$$\ang{L,\psi^k L}_d=d\ang{\psi^k L}_d+\ang{\psi^{k-1}P}_d.$$
\end{lem}

\begin{proof}
On the left, we count curves of degree $d$ passing through $3d-1-k$ points and subject to a $\psi^k L$-constraint. However, there is an additional marked point coming from the intersection with a line. This marked point can be adjacent or not to the vertex adjacent to the $\psi^k L$-point, leading to the two terms on the right-hand side.
\end{proof}

\begin{rem}
For $k=0$, we recover the identity $\ang{L,L}_d=d\ang{L}_d=d^2 N_d$.
\end{rem}

\subsubsection{Case with two lines.} \label{subsec-casetwolines} Assume there are exactly two line conditions $L$ and $L'$, coupled with a $\psi^{k_1}L$- resp.\ $\psi^{k_2}L$-constraint. Let $n$ be the number of $P$-constraints. By dimension count, we have $n+k_1+k_2=3d-1$. Contrarily to case with a unique $\psi^k L$-constraint, and contrarily to the case with only $\psi^kP$-constraints, we can now have solutions where different markings are adjacent to the same vertex. For the invariant $\ang{\psi^{k_1}L,\psi^{k_2}L}_d$, these correspond to degree $d$ curves having a $(k_1+k_2+3)$-valent vertex at the intersection point between the line constraints $L$ and $L'$. This contributes exactly $\frac{(k_1+k_2)!}{k_1!k_2!}\ang{\psi^{k_1+k_2-1}P}_d$. We now care about the solutions where the two $\psi^k L$-markings are not adjacent to the same vertex. The relative position of the marked points on a solution is then described by the following lemma.

\begin{lem}\label{lemma position points r=2}
Let $h:\Gamma\to\RR^2$ be a solution to the tropical counting problem defined at the beginning of Subsection \ref{subsec-casetwolines}. Assume that the marked point with the $\psi^{k_1}L$- resp.\ $\psi^{k_2}L$-constraint are not adjacent to the same vertex. Let $\Gamma\backslash\P$ be the complement of marked points. There are exactly two bounded components, every other component contains a unique end. Moreover:
\begin{itemize}[label=$\ast$]
\item Each bounded component has a $\psi^{k_i}L$-marking on its boundary.
\item Each of the $\psi^{k_i} L$-markings is adjacent to a bounded component.
\end{itemize}
\end{lem}

\begin{proof}
By the genericity assumption on the constraints, the curve is trivalent out of the $\psi^{k_i} L$-markings, and the complement of the marked points has thus exactly
$$1+n+(k_1+1)+(k_2+1)=3d+2,$$
connected components. As a connected component cannot contain more than two ends, otherwise we would have a $1$-parameter family of solutions, $3d$ of them are unbounded, and two of them are bounded.

As in the case of a unique $\psi^kL$-constraint, by genericity, a bounded component cannot be adjacent to only $P$-constraints, so that a bounded component has at least one $\psi^{k_i} L$-constraint on its boundary.

If a $\psi^{k_i}L$-constraint was adjacent only to unbounded components, it would be possible to deform by moving the $\psi^{k_i} L$-marking along $L$ (or $L'$), and simultaneously the paths from the $\psi^{k_i}L$-marking to the end in each of the adjacent components, yielding a $1$-parameter family of solutions, which is impossible by genericity.
\end{proof}

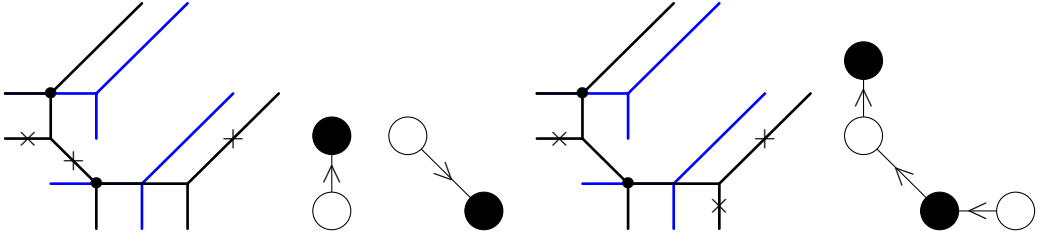
\begin{figure}
\begin{center}
\begin{tabular}{cccc}
\begin{tikzpicture}[line cap=round,line join=round,>=triangle 45,x=0.6cm,y=0.6cm]
		\draw[line width=1pt,blue] (3,1)--++(2,2)++(-2,-2)--++(0,-1)++(0,1)--++(-2,0);
		\draw[line width=1pt,blue] (2,3)--++(2,2)++(-2,-2)--++(0,-1)++(0,1)--++(-2,0);

		\draw[line width=1pt] (0,2)--++(1,0)--++(1,-1)--++(2,0)--++(2,2);
		\draw[line width=1pt] (0,3)--++(1,0)--++(2,2);
		\draw[line width=1pt] (2,0)--++(0,1);
		\draw[line width=1pt] (4,0)--++(0,1);
		\draw[line width=1pt] (1,2)--++(0,1);
		\draw (2,1) node {$\bullet$};
		\draw (1,3) node {$\bullet$};
		\draw (0.5,2) node {$\times$};
		\draw (1.5,1.5) node {$+$};
		\draw (5,2) node {$+$};
		\end{tikzpicture} & \begin{tikzpicture}[line cap=round,line join=round,>=triangle 45,x=1cm,y=1cm]
\white (W1) at (0,0) {};
\black (B1) at (0,1) {};
\white (W2) at (1,1) {};
\black (B2) at (2,0) {};
\draw (W1) to node[midway,sloped] {$>$} (B1);
\draw (W2) to node[midway,sloped] {$>$} (B2);
\end{tikzpicture} & \begin{tikzpicture}[line cap=round,line join=round,>=triangle 45,x=0.6cm,y=0.6cm]
		\draw[line width=1pt,blue] (3,1)--++(2,2)++(-2,-2)--++(0,-1)++(0,1)--++(-2,0);
		\draw[line width=1pt,blue] (2,3)--++(2,2)++(-2,-2)--++(0,-1)++(0,1)--++(-2,0);

		\draw[line width=1pt] (0,2)--++(1,0)--++(1,-1)--++(2,0)--++(2,2);
		\draw[line width=1pt] (0,3)--++(1,0)--++(2,2);
		\draw[line width=1pt] (2,0)--++(0,1);
		\draw[line width=1pt] (4,0)--++(0,1);
		\draw[line width=1pt] (1,2)--++(0,1);
		\draw (2,1) node {$\bullet$};
		\draw (1,3) node {$\bullet$};
		\draw (0.5,2) node {$\times$};
		\draw (4,0.5) node {$\times$};
		\draw (5,2) node {$+$};
		\end{tikzpicture} & \begin{tikzpicture}[line cap=round,line join=round,>=triangle 45,x=1cm,y=1cm]
\white (W1) at (2,0) {};
\black (B1) at (1,0) {};
\white (W2) at (0,1) {};
\black (B2) at (0,2) {};
\draw (W1) to node[midway,sloped] {$<$} (B1) to node[midway,sloped] {$<$} (W2) to node[midway,sloped] {$>$} (B2);
\end{tikzpicture} \\
\end{tabular}
\caption{\label{figure example position of marked points r=2}Two conics contributing to $\ang{\psi L,\psi L}_2$ for different choices of constraints. For the first one, each $\psi L$-marking is adjacent to a unique bounded component. For the second one, one of the $\psi L$-markings is adjacent to the two bounded components. On the right of each curve is depicted the graph $\B(\Gamma)$ along with the orientation given by Lemma \ref{lem orientation bipartite graph}. White vertices are the bounded components, black vertices are the $\psi^k L$-markings.}
\end{center}
\end{figure}

\begin{expl}
In Figure \ref{figure example position of marked points r=2} we see two conics subject to three $P$-constraints and two $\psi L$-constraints. We can see that there is indeed two bounded components in the complement of marked points. Each $\psi L$-marking is adjacent to at least one of them, and each of them has at least one on its boundary.
\end{expl}

The multiplicity is given by the following proposition. Its proof is similar to the case of a unique $\psi^k L$-constraint, except there are now two bounded components. We have two possibilities:
	\begin{itemize}[label=-]
	\item Each $\psi^kL$-marking is adjacent to a unique bounded component.
	\item One of the $\psi^k L$-marking is adjacent to the two bounded components.
	\end{itemize}
	In each case it is possible to assign a bounded component to each $\psi^k L$-marking: in the first case, it is the unique adjacent one. In the second case, one $\psi^k L$-marking is adjacent to a unique bounded component, which we assign to it, and the other $\psi^k L$-marking is thus assigned the other bounded component.

\begin{prop}\label{prop-mult2}
Let $h:\Gamma\to\RR^2$ be a solution to the tropical counting problem defined at the beginning of Subsection \ref{subsec-casetwolines}. Assume that the marked point with the $\psi^{k_1}L$- resp.\ $\psi^{k_2}L$-constraint are not adjacent to the same vertex. Le $V_i$ be the vertex adjacent to the $\psi^{k_i} L$-marking, let $u_i$  be the slope of $L$ (resp. $L'$) at $V_i$, and let $v_i$  be the slope of $h$ on the edge adjacent to $V_i$  going into the assigned bounded component. The multiplicity is
$$m_\Gamma=|\det(u_1,v_1)||\det(u_2,v_2)|\prod_V m_V.$$
\end{prop}

\begin{proof}
As in the case of a unique $\psi^k L$-constraint, the multiplicity can be computed as the determinant of the evaluation map, and we use the toolbox with the following steps:
	\begin{itemize}[label=$\ast$]
	\item We remove the $P$-constraints by replacing each $P$-marking with a pair of ends whose position is fixed.
	\item For each trivalent vertex adjacent to two ends with fixed position, we prune the curve by removing the vertex, which yields a contribution of $m_V$.
	\item During the previous step, each bounded component adjacent to a unique $\psi^{k_i} L$-marking gets pruned. We then prune the vertex and replace every other adjacent edge by an end with fixed position, which yields a contribution of $|\det(u_i,v_i)|$.
	\end{itemize}
	The previous steps finish the computation of the determinant. The last step is applied twice, once for each of the $\psi^{k_i} L$-markings. Its applications is simultaneous for both $\psi^{k_i} L$-markings if each of them is adjacent to a unique bounded component. In the other case, it gets applied once to prune of the the $\psi^{k_1} L$-constraint, then the other bounded component gets pruned, and a second application deletes the $\psi^{k_2} L$-marking. In the end, we get the announced multiplicity.
\end{proof}

\begin{expl}
Each of the conics in Figure \ref{figure example position of marked points r=2} has multiplicity $1$.
\end{expl}

\subsubsection{General case.} We now consider an arbitrary number $r$ of $\psi^k L$ constraints, and $n$ $P$-constraints, so that we have
$$n+\sum_1^r k_i=3d-1.$$
As before, we have the solutions where two among the $\psi^k L$-constraints are realized at the same vertex, yielding a contribution of
$$\sum_{i<j}\frac{(k_i+k_j)!}{k_i!k_j!}\ang{\psi^{k_1}L,\cdots,\widehat{\psi^{k_i}L},\cdots,\widehat{\psi^{k_j}L},\cdots,\psi^{k_r}L,\psi^{k_i+k_j-1}P}_d,$$
to $\ang{\psi^{k_1}L,\cdots,\psi^{k_r}L}_d$. By induction, we can assume that the $\psi^{k_i} L$-markings are adjacent to different vertices (the addition of $\psi^k P$-constraints is no problem for the setting but just involves more notation). Let $h:\Gamma\to\RR^2$ be such a solution. We can form a bipartite graph $\B(\Gamma)$ out of $\Gamma$: the white vertices are the bounded components of $\Gamma\backslash\P$, and the black vertices are the $r$ $\psi^{k_i} L$-markings. Two vertices are linked by an edge if the corresponding $\psi^{k_i} L$-marking is adjacent to the corresponding bounded component. The following lemma describes the relative position of the marked points on a solution $h:\Gamma\to\RR^2$ in terms of $\B(\Gamma)$.

\begin{lem}
Let $h:\Gamma\to\RR^2$ be a solution to the tropical enumeratve problem described above with the $\psi^{k_i} L$-markings adjacent to different vertices. We have the following:
\begin{enumerate}[label=(\roman*)]
\item There are exactly $3d$ unbounded components in $\Gamma\backslash\P$, and $r$ bounded components.
\item Each unbounded connected component of $\Gamma\backslash\P$ contains exactly one unbounded end.
\item Every vertex of $\B(\Gamma)$ has valency at least one. 
\end{enumerate}
\end{lem}

\begin{proof}
By genericity of the constraints, any solution is trivalent out of the $\psi^{k_i} L$-markings, which are of known valency in $\Gamma$. The complement of marked points has thus exactly
$$1+n+\sum_1^r(k_i+1)=3d+r$$
connected components. No connected component can contain more than one unbounded end, otherwise it would be possible to deform the string linking one to the other. Thus, there are $3d$ bounded components since there are $3d$ ends in $\Gamma$. Hence, there are $r$ bounded components.

A white vertex of $\B(\Gamma)$ corresponds to a bounded component of $\Gamma\backslash\P$. It has at least one $\psi^{k_i} L$-constraint on its boundary, otherwise the point configuration is non-generic, as in the case of a unique $\psi^{k_i}L$-constraint.

A black vertex corresponding to a $\psi^{k_i} L$-constraint has to be adjacent to at least one bounded component, otherwise it is possible to deform the position of the $\psi^{k_i} L$-point along with the paths linking it to infinity inside each adjacent unbounded component.
\end{proof}

This description recovers the result for $r=1$ and $r=2$. For $r=1$ there is a unique bipartite graph with two vertices, and for $r=2$, there are two possible bipartite graphs.
	
	We can then put an orientation on the graph $\B(\Gamma)$ using the following lemma. Recall that as $\Gamma$ is without cycle, so is $\B(\Gamma)$.

\begin{lem}\label{lem orientation bipartite graph}
	Each vertex of $\B(\Gamma)$ is adjacent to at most one leaf (of the other color). Moreover, there exists a unique orientation on $\B(\Gamma)$ that satisfies the following:
	\begin{itemize}[label=-]
	\item Each white vertex (\textit{i.e.} bounded component) has exactly one outward edge.
	\item Each black vertex (\textit{i.e.} $\psi^{k_i} L$-point) has exactly one inward edge.
\end{itemize}
\end{lem}

\begin{proof}
Assume there is a white vertex adjacent to two black leaves. Then there is a string relating the two corresponding $\psi^{k_i} L$-markings through the bounded component. Moreover, every other connected component adjacent to one of the $\psi^{k_i} L$-markings has a unique end. Thus, it is possible to deform the strings by moving the $\psi^{k_i} L$-constraints along their line, getting a $1$-parameter family of solutions, which is impossible.

Assume there is a black vertex adjacent to two white leaves. Then line constraint of the $\psi^{k_i} L$-marking and the two adjacent bounded edges corresponding to the two leaves meet at a common point. However, this cannot happen if the constraints are chosen generically as it gives a relation on the coordinate of the $P$-markings adjacent to the bounded components.

To put an orientation, we proceed by induction on the number of vertices. If $\B(\Gamma)$ has no vertices, the result is true. If $\B(\Gamma)$ is not empty, the set of leaves is non-empty either, as $\B(\Gamma)$ is without cycle. 
\begin{itemize}[label=$\ast$]
\item Consider a black leaf. The unique adjacent edge has to be oriented toward the black vertex. Moreover, by the first part of the proposition, it is the unique leaf adjacent to the white neighbour. Thus, every other edge adjacent to the white neighbour has to be oriented toward the white vertex.
\item For a white leaf, just switch black and white, and the orientation of the edges.
\item As the edges adjacent to the white (resp. black) neighbour of a black (resp. white) leaf are oriented inward (resp. outward), they are not of the right orientation to be the unique inward (resp. outward) edge of their other black (resp. white) neighbours. So we can remove the leaves and their only neighbour and proceed by induction.
\end{itemize}
\end{proof}

\begin{figure}
\begin{center}
\begin{tabular}{ccc}
\begin{tikzpicture}[line cap=round,line join=round,>=triangle 45,x=0.9cm,y=0.9cm]
		\draw[line width=1pt,blue] (6,3)--++(-0.5,0)++(0.5,0)--++(0,-0.5)++(0,0.5)--++(2,2);
		
		\draw[line width=1pt] (3,4.5)--++(1,0)--++(1,-1)--++(0,-1)++(0,1)--(6.5,3.5) node {$\bullet$} --++(2,2)++(-2,-2)--++(0,-1);
		\draw[line width=1pt] (3,5)--++(1,0)--++(1,1)++(-1,-1)--++(0,-0.5);
		\draw (5,3) node {$\times$};
		\draw (4.5,4) node {$+$};
		\draw (3.5,5) node {$\times$};
		\draw (4.5,5.5) node {$+$};
		\end{tikzpicture} & \begin{tikzpicture}[line cap=round,line join=round,>=triangle 45,x=1cm,y=1cm]
		\draw[line width=1pt] (1,2)--++(0,1)--++(-1,0)++(1,0)--++(1,1)--++(0,1)--++(-1,0)++(1,0)--++(1,1)++(-1,-1)++(0,-1)--++(1,0)--++(1,1)++(-1,-1)--++(0,-2);
		\draw (1,3) node {$\bullet$};
		\draw (3,4) node {$\bullet$};
		\draw (2,5) node {$\bullet$};
		\draw (0.5,3) node {$\times$};
		\draw (3.5,4.5) node {$+$};
		\draw[line width=1pt,blue] (1,3.5)--++(0.5,0.5)++(-0.5,-0.5)--++(0,-0.5)++(0,0.5)--++(-0.5,0);
		\draw[line width=1pt,blue] (2.5,5)--++(0.5,0.5)++(-0.5,-0.5)--++(0,-0.5)++(0,0.5)--++(-0.5,0);
		\draw[line width=1pt,blue] (3.5,4)--++(0.5,0.5)++(-0.5,-0.5)--++(0,-0.5)++(0,0.5)--++(-0.5,0);
		\end{tikzpicture} & \begin{tikzpicture}[line cap=round,line join=round,>=triangle 45,x=1cm,y=1cm]
		\draw[line width=1pt] (1,0)--++(0,3)--++(-1,0)++(1,0)--++(1,1)--++(0,1)--++(-1,0)++(1,0)--++(1,1)++(-1,-1)++(0,-1)--++(1,0)--++(1,1)++(-1,-1)--++(0,-2)--++(-3,0)++(3,0)--++(1,-1)--++(0,-1)++(0,1)--++(1,0)--++(0,-1)++(0,1)--++(1,1);
		\draw (1,3) node {$\bullet$};
		\draw (5,1) node {$\bullet$};
		\draw (2,5) node {$\bullet$};
		\draw (3,4) node {$\bullet$};
		\draw (0.5,3) node {$\times$};
		\draw (2,2) node {$\times$};
		\draw (4,0.5) node {$\times$};
		\draw (3.5,4.5) node {$+$};
		\draw[line width=1pt,blue] (1,3.5)--++(0.5,0.5)++(-0.5,-0.5)--++(0,-0.5)++(0,0.5)--++(-0.5,0);
		\draw[line width=1pt,blue] (2.5,5)--++(0.5,0.5)++(-0.5,-0.5)--++(0,-0.5)++(0,0.5)--++(-0.5,0);
		\draw[line width=1pt,blue] (3.5,4)--++(0.5,0.5)++(-0.5,-0.5)--++(0,-0.5)++(0,0.5)--++(-0.5,0);
		\draw[line width=1pt,blue] (5,1.5)--++(0.5,0.5)++(-0.5,-0.5)--++(0,-0.5)++(0,0.5)--++(-0.5,0);
		\end{tikzpicture} \\
\begin{tikzpicture}[line cap=round,line join=round,>=triangle 45,x=1cm,y=1cm]
\white (W1) at (0,0) {};
\black (B1) at (1,1) {};
\draw (W1) to node[midway,sloped] {$>$} (B1);
\end{tikzpicture} & \begin{tikzpicture}[line cap=round,line join=round,>=triangle 45,x=1cm,y=1cm]
\white (W1) at (-1,0) {};
\white (W2) at (0,0) {};
\white (W3) at (1,0) {};
\black (B1) at (-0.5,-1) {};
\black (B2) at (0.5,-1) {};
\black (B3) at (0,1) {};
\draw (W2) to node[midway,sloped] {$>$} (B3);
\draw (W1) to node[midway,sloped] {$>$} (B1) to node[midway,sloped] {$>$} (W2) to node[midway,sloped] {$<$} (B2) to node[midway,sloped] {$<$} (W3);
\end{tikzpicture} & \begin{tikzpicture}[line cap=round,line join=round,>=triangle 45,x=1cm,y=1cm]
\white (W1) at (0,0) {};
\white (W2) at (1,0) {};
\white (W3) at (3,1) {};
\white (W4) at (3,-1) {};
\black (B1) at (0,-1) {};
\black (B2) at (1,1) {};
\black (B3) at (2,0) {};
\black (B4) at (3,0) {};
\draw (W1) to node[midway,sloped] {$>$} (B1) to node[midway,sloped] {$>$} (W2) to node[midway,sloped] {$>$} (B2);
\draw (W2) to node[midway,sloped] {$<$} (B3) to node[midway,sloped] {$<$} (W3);
\draw (B3) to node[midway,sloped] {$>$} (W4) to node[midway,sloped] {$>$} (B4);
\end{tikzpicture} \\
\end{tabular}
\caption{\label{figure bipartite graph} In the first row some tropical curves subject to $P$-constraints and $\psi L$-constraints. Below the associated bipartite graphs $\B(\Gamma)$ along with the orientation given by Lemma \ref{lem orientation bipartite graph}.}
\end{center}
\end{figure}
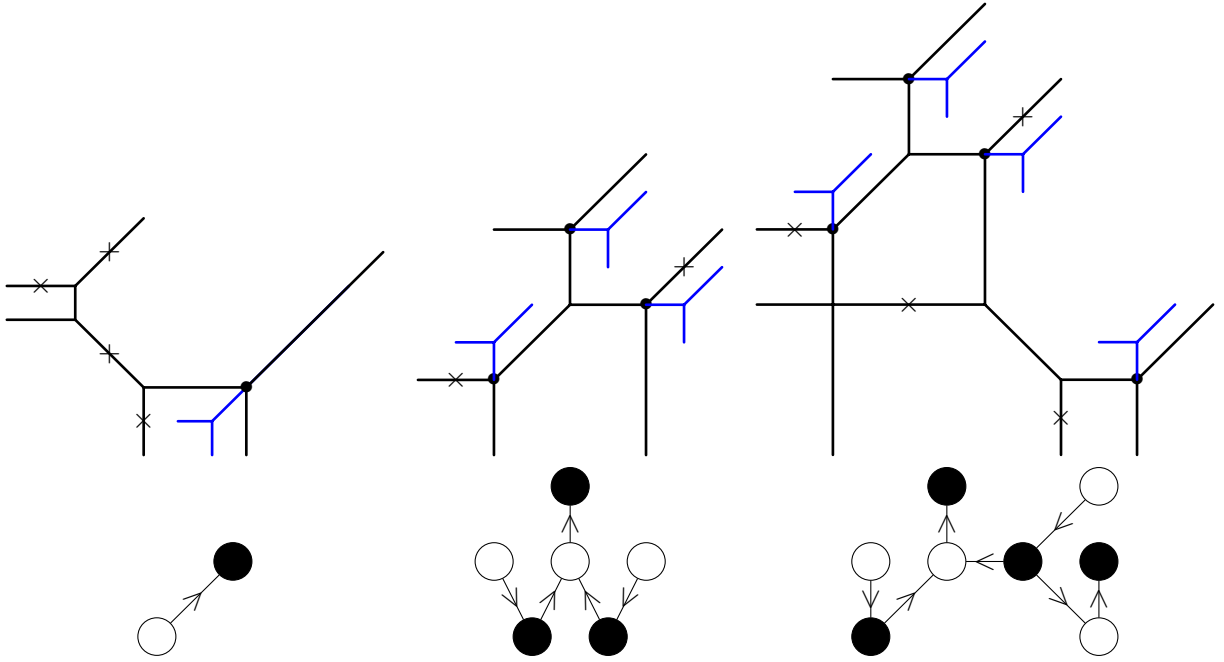

\begin{expl}
In Figure \ref{figure bipartite graph} we can see bipartite graphs along with the orientation given by Lemma \ref{lem orientation bipartite graph}. In Figure \ref{figure example position of marked points r=2} we can also see two conics subject to three $P$-constraints and two $\psi L$-constraints along with the associated $\B(\Gamma)$.
\end{expl}
			
	\begin{prop}
	Let $h:\Gamma\to\RR^2$ be a solution to the problem, with the $\psi^{k_i} L$-markings adjacent to distinct vertices. For each $V_i$ adjacent to a $\psi^{k_i} L$-marking, let $u_i$ be the slope of $\Xi_i$ at $V_i$, and let $v_i$ be the slope of $h$ on the edge of $\Gamma$ adjacent to $V_i$ and belonging to the bounded component corresponding to the unique inward edge of $\B(\Gamma)$. The multiplicity is given by
	$$m_\Gamma=\prod_i|\det(u_i,v_i)|\prod_V m_V.$$
	\end{prop}
	
	\begin{proof}
	As in the case of one and two $\psi^{k_i} L$-constraints, the multiplicity is given by the determinant of the evaluation map. We use the toolbox to compute the determinant, proceeding as follows:
	\begin{itemize}[label=$\ast$]
	\item Remove every $P$-constraint by splitting $\Gamma$ at each $P$-marking, replacing the vertex by a pair of unbounded ends with fixed position.
	\item For every trivalent vertex adjacent to a pair of fixed ends, we prune the vertex by replacing it with an end of fixed position. This contributes $m_V$ to the determinant.
	\item For a $\psi^{k_i} L$-marking adjacent to a unique end of fixed position, we prune $\Gamma$ by removing the $\psi^{k_i} L$-marking, replacing every other adjacent edge by an end with fixed position. This contributes $|\det(u_i,v_i)|$ for the $\psi^{k_i} L$-marking $V_i$.
	\end{itemize}
	The second step progressively prunes the bounded components corresponding to the white leaves of $\B(\Gamma)$. Then, the third step prunes the black neighbour of such a white leaf. We thus continue by induction. The process terminates yielding the expected multiplicity.
	\end{proof}

\section{Explicit computations}\label{sec-explicit}

	\subsection{Case of a unique line constraint}
	
	In this section we give an explicit formula to compute the invariants $\ang{\psi L}_d$ and $\ang{\psi^2 L}_d$. The latter reduces the computation to point invariants and some relative invariants, which can be computed using for instance the lattice path algorithm from \cite{mikhalkin2005enumerative}.
	
	To compute the invariants, we proceed as in the proof of the Caporaso-Harris formula \cite{gathmann2007caporaso} by choosing a very specific choice of constraints. Namely, we choose a tropical line $L$ whose vertex is far from the points in the configuration (see also the beginning of the proof of Theorem 4.3 in \cite{gathmann2007caporaso}, and in particular the picture). We start with some generic results.
	
	\begin{lem}\label{lemma meeting the point region}
	Assume a tropical curve contributes to $\ang{\psi^k L}_d$ for a choice of $L$ with vertex far from the points in the configuration. We have the following.
	\begin{enumerate}[label=(\roman*)]
	\item The extremity of the bounded edge adjacent to the $\psi^k L$-marking in the bounded component of the complement of marked points lies in the point region.
	\item For an unbounded component adjacent to the $\psi^k L$-marking, it is one of the following type, as depicted on Figure \ref{fig-contribution patterns}:
		\begin{enumerate}[label=(\alph*)]
		\item it starts with a bounded edge going inside the point region, it thus has the same slope as the bounded edge adjacent to the $\psi^k L$-marking in the bounded component,
		\item an end adjacent to the $\psi^k L$-marking,
		\item a path from the marking to an end such that at each trivalent vertex along the way, the remaining edge goes inside the point region.
		\end{enumerate}
	\end{enumerate}		
	\end{lem}
	
	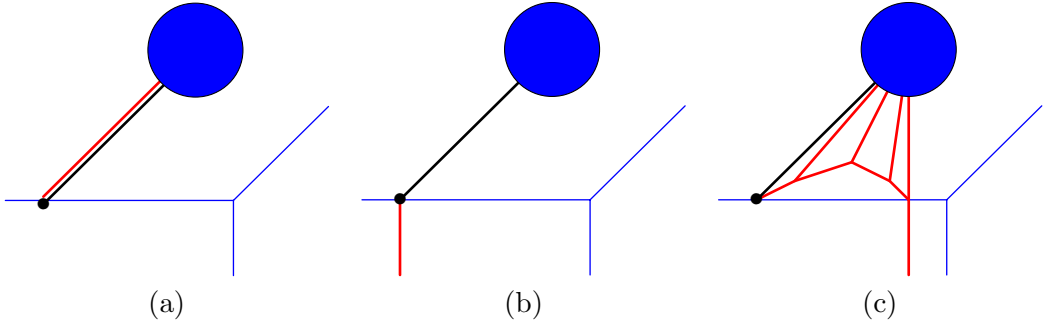
\begin{figure}
	\begin{center}
	\begin{tabular}{ccc}
	\begin{tikzpicture}[line cap=round,line join=round,>=triangle 45,x=0.25cm,y=0.25cm]
\draw [line width=0.5pt,blue] (0,0)--++(5,5) (0,0)--++ (-12,0) (0,0)--++(0,-4);
\draw[line width=1pt,red] (-10,0.2)--++(8,8);
\draw[line width=1pt] (-2,7.8)--++(-8,-8) node {$\bullet$};
\draw[fill=blue] (-2,8) circle [radius=2.5];
\end{tikzpicture} & \begin{tikzpicture}[line cap=round,line join=round,>=triangle 45,x=0.25cm,y=0.25cm]
\draw [line width=0.5pt,blue] (0,0)--++(5,5) (0,0)--++ (-12,0) (0,0)--++(0,-4);
\draw[line width=1pt,red] (-10,0)--++(0,-4);
\draw[line width=1pt] (-2,8)--++(-8,-8) node {$\bullet$};

\draw[fill=blue] (-2,8) circle [radius=2.5];
\end{tikzpicture} & \begin{tikzpicture}[line cap=round,line join=round,>=triangle 45,x=0.25cm,y=0.25cm]
\draw [line width=0.5pt,blue] (0,0)--++(5,5) (0,0)--++ (-12,0) (0,0)--++(0,-4);
\draw[line width=1pt,red] (-10,0)--(-8,1)--(-5,2)--(-3,1)--(-2,0)--++(0,-4);
\draw[line width=1pt,red] (-2,8)--(-8,1) (-2,8)--(-5,2) (-2,8)--(-3,1) (-2,8)--(-2,0);
\draw[line width=1pt] (-2,8)--++(-8,-8) node {$\bullet$};

\draw[fill=blue] (-2,8) circle [radius=2.5];
\end{tikzpicture} \\
(a) & (b) & (c) \\
	\end{tabular}
	\caption{\label{fig-contribution patterns} Deformation patterns occuring in Lemma \ref{lemma meeting the point region}. The blue disk is the region where the marked points lie. The bounded component adjacent to the $\psi^k L$-marking is in black, and the considered adjacent unbounded component in red. }
	\end{center}
	\end{figure}
	
	\begin{proof}
	Let us consider the unique bounded component in the complement of marked points adjacent to the $\psi^k L$-marking, and the bounded edge adjacent to the $\psi^k L$-marking. If its extremity was not in the point region, then the classical line that contains it splits $\RR^2$ in two halves only one of them containing the point region. The trivalent vertex at its extremity has one of its edges going inside the half that does not contain the point region. There is at least an end in this region by balancing, contradicting the fact that the component is bounded. Thus, the extremity lies in the point region. This concludes the proof of (i).
	
	For (ii), consider an unbounded component adjacent to the $\psi^k L$-marking that is not an end. If the slope is the same as the edge going in the bounded component, its extremity has to lie in the point region, otherwise we can apply the previous reasoning to the lines that contain the other two adjacent edges to its extremity: each of them cuts $\RR^2$ in two halves only one of them containing the point region, and we would get an end in each direction by balancing. We get curves in case (a).
	
	Assume that we are not in (a): the slope is different, and that the edge adjacent to the $\psi^k L$-marking is not an end, so that we are not in (b). We still consider the line containing the edge adjacent to the $\psi^k L$-marking: it splits $\RR^2$ into two halves, and the half not containing the point region has to contain an end. As there is a unique end, the remaining edge going into the half that contains the point region is bounded and has to go inside the point region: if not, we would get a second end in the component proceeding as in the previous paragraph. Following the edges going into the half not containing the point region, we get curves as described in (c) and pictured in Figure \ref{fig-contribution patterns}(c).
	\end{proof}
	
		\subsubsection{A formula for $\ang{\psi L}_d$, i.e.\ the $k=1$ case.} First, we assume that the power of the unique $\psi^k L$-constraint is $1$.
		
		\begin{prop}\label{prop-k=1}
		Using the notation from Section \ref{subsec-enum}, we have
		$$\boxed{ \ang{\psi L}_d= 2N_d+N_d(2). }$$
		\end{prop}			
	
	\begin{proof}
	We choose a specific configuration of constraints where the vertex of the line $L$ is far from the points of the configuration. The set of tropical curves passing through the point configuration is a $1$-dimensional family, and each curve passing through the points contains a unique string: a path from one end to another not passing through the marked points. The $1$-parameter family is obtained by deforming the string. According to \cite{GM053}, the only deformations that are unbounded are the ones where the string is only comprised of two ends adjacent to the same vertex. Because of our choice of constraints, we need such an unbounded deformation in order to reach the line $L$ with the vertex of the string, thus satisfyig the $\psi L$ condition. A priori, there are $6$ possibilities for a string formed by two ends, depending on whether the ends have the same slope or not. Given a tropical line $L$ far from the point constraints, only three of these possibilities yield solutions by meeting the line as required, as depicted in Figure \ref{figure patterns for k=1}. Their contribution are respectively $N_{dL-E}=N_d$, $\frac{2}{2}N_d(2)$ and $N_{dL-E}$.
	\end{proof}
	
	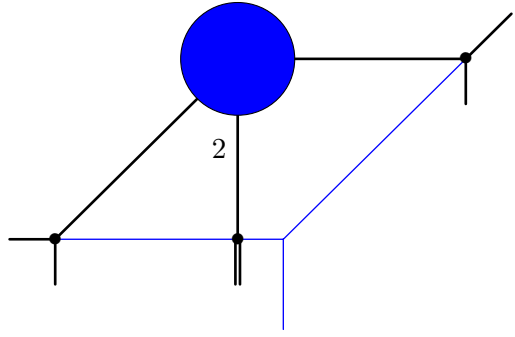
\begin{figure}
	\begin{center}
	\begin{tikzpicture}[line cap=round,line join=round,>=triangle 45,x=0.3cm,y=0.3cm]
\draw [line width=0.5pt,blue] (0,0)--++(10,10) (0,0)--++ (-12,0) (0,0)--++(0,-4);
\draw[line width=1pt] (-2,8)--++(-8,-8)--++(-2,0)++(2,0) node {$\bullet$} --++(0,-2);
\draw[line width=1pt] (-2,8)--++(10,0)--++(2,2)++(-2,-2) node {$\bullet$} --++(0,-2);
\draw[line width=1pt] (-2,8)-- node[midway,left] {$2$} ++(0,-8) node {$\bullet$};
\draw[line width=1pt] (-2.1,0)--++(0,-2);
\draw[line width=1pt] (-1.9,0)--++(0,-2);
\draw[fill=blue] (-2,8) circle [radius=2.5];
\end{tikzpicture}
	\caption{\label{figure patterns for k=1}Patterns occurring for tropical curves contributing to $\ang{\psi L}_d$ with a tropical line far from the points.}
	\end{center}
	\end{figure}
	
	\begin{expl}\label{ex-k=1}
	With the formula from Proposition \ref{prop-k=1}, we can compute the following values:
	\begin{itemize}[label=$\ast$]
	\item $\ang{\psi L}_1=2\cdot 1+0=2$,
	\item $\ang{\psi L}_2=2\cdot 1 + 2=4$,
	\item $\ang{\psi L}_3=2\cdot 12 + 36=60$.
	\end{itemize}
	\end{expl}

		\subsubsection{A formula for $\ang{\psi^2 L}_d$, i.e.\ the $k=2$ case.} We now give a formula for $\ang{\psi^2 L}_d$. We use the degree
		$$\square_d=\{(1,1)^d,(0,-1)^{d-1},(-1,0)^{d-2},(-2,-1)\},$$
		along with its symmetric ones obtained by switching the coordinates and the ends of directions $(0,-1)$ and $(1,1)$.
		
		\begin{prop}\label{prop-k=2}
		We have
		$$\boxed{ \ang{\psi^2 L}_d=3N_{\square_d}+\frac{1}{2}N_d(3)+\frac{1}{2}\sum_{d_1+d_2=d}\bino{3d-3}{3d_1-1}d_1 N_{d_1}N_{d_2} . }$$
		\end{prop}
		
		\begin{proof}
		We again choose a tropical line constraint $L$ with vertex far from the points of the configuration. As $k=2$, we have four adjacent edges leaving the $\psi^2 L$-vertex. The list of possibilities for the type of the $\psi^2 L$-vertex is depicted on Figure \ref{figure patterns k=2}. The two patterns at the far left and far right each contribute for $\frac{1}{2}N_{\square_d}$. The other two contribute for $\frac{2}{2}N_{\square_d}$ since the determinant between the bounded component and $L$ is not $1$ but $2$. The pattern with a vertical edge of weight $3$ contributes for $\frac{3}{3!}N_d(3)$. The last pattern with parallel edges contributes the following:
		$$\frac{1}{2}\sum_{d_1+d_2=d}\bino{3d-3}{3d_1-1}d_1N_{d_1}N_{d_2}.$$
		The $\frac{1}{2}$ is for automorphisms, the sum is over the splittings of the degree between the bounded component and the unbounded one. We dispatch the marked points. Then we have $N_{d_1}$ choices for the bounded component, $d_1$ choices for the end that intersects $L$, and $\widetilde{N}_{d_2}(1)=N_{d_2}$ choices for the second component.
		\end{proof}
		
		\begin{figure}
		\begin{center}
		\begin{tikzpicture}[line cap=round,line join=round,>=triangle 45,x=0.3cm,y=0.3cm]
\draw [line width=0.5pt,blue] (0,0)--++(20,20) (0,0)--++ (-20,0) (0,0)--++(0,-4);

\draw[line width=1pt] (-2,8)--++(-16,-8) node {$\bullet$}--++(0,-2)++(0,2)++(0,0.1)--++(-2,0)++(0,-0.2)--++(2,0);

\draw[line width=1pt] (-2,8)--++(-4,-8) node {$\bullet$}--++(-2,0)++(2,0)++(0.1,0)--++(0,-2)++(-0.2,0)--++(0,2);

\draw[line width=1pt] (-2,8)--++(5,-5) node {$\bullet$}--++(2,2)++(-2,-2)++(0.1,0)--++(0,-2)++(-0.2,0)--++(0,2);

\draw[line width=1pt] (-2,8)--++(20,10) node {$\bullet$} ++(-0.1,0.1)--++(2,2)++(0.2,-0.2)--++(-2,-2)--++(0,-2);

\draw[line width=1pt] (-2.5,7)-- node[midway,left] {$3$} ++(0,-7) node {$\bullet$};
\draw[line width=1pt] (-2.7,0)--++(0,-2);
\draw[line width=1pt] (-2.5,0)--++(0,-2);
\draw[line width=1pt] (-2.3,0)--++(0,-2);

\draw[line width=1pt] (-1.6,7)--++(0,-9)++(0.2,0)--++(0,9);
\draw (-1.5,0) node {$\bullet$};

\draw[fill=blue] (-2,8) circle [radius=2.5];
\end{tikzpicture}
		\caption{\label{figure patterns k=2} Patterns occurring for tropical curves contributing to $\ang{\psi^2 L}_d$ with a tropical line far from the points.}
		\end{center}
		\end{figure}
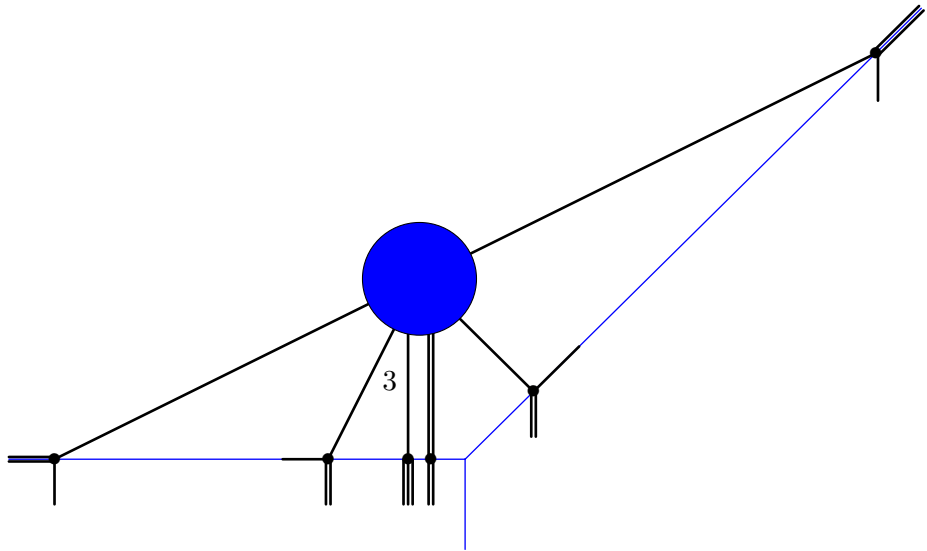
		
		\begin{expl}\label{ex-k=2}
		Using the formula of Proposition \ref{prop-k=2} , we can compute:
		\begin{itemize}[label=$\ast$]
		\item $\ang{\psi^2 L}_2=\frac{9}{2}$,
		\item $\ang{\psi^2 L}_3=54$.
		\end{itemize}
		\end{expl}
		
		\subsubsection{Further cases.}\label{subsec-furthercases} The explicit computation of $\ang{\psi^k L}_d$ for $k\geqslant 3$ can also be carried out by this method provided one is able to list the different patterns that can occur at the $\psi^k L$-vertex. This list becomes longer and longer with patterns of complicated shape when $k$ increases because the slope of the adjacent edges can take more and more values. See for instance Figure \ref{figure patterns for k=3} that presents some patterns for $k=3$.
		
		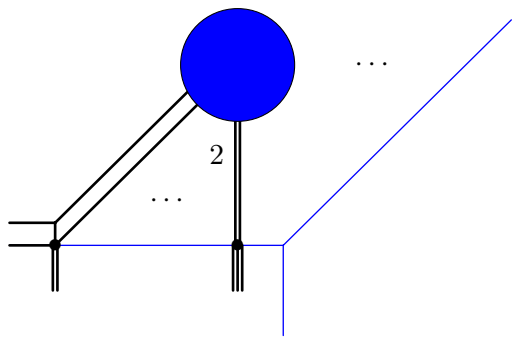
\begin{figure}
	\begin{center}
	\begin{tikzpicture}[line cap=round,line join=round,>=triangle 45,x=0.3cm,y=0.3cm]
\draw [line width=0.5pt,blue] (0,0)--++(10,10) (0,0)--++ (-12,0) (0,0)--++(0,-4);

\draw[line width=1pt] (-2,8)--++(-8,-8)--++(-2,0)++(2,0) node {$\bullet$} ++(0.1,0)--++(0,-2)++(-0.2,0)--++(0,2)++(0.1,0)--++(0,1)--++(-2,0)++(2,0)--++(8,8);

\draw[line width=1pt] (-2.1,8)-- node[midway,left] {$2$} ++(0,-8) ++(0.1,0) node {$\bullet$};
\draw[line width=1pt] (-1.9,8)--++(0,-8);
\draw[line width=1pt] (-2.2,0)--++(0,-2);
\draw[line width=1pt] (-2,0)--++(0,-2);
\draw[line width=1pt] (-1.8,0)--++(0,-2);

\draw (4,8) node {$\cdots$};
\draw (-5,2) node {$\cdots$};

\draw[fill=blue] (-2,8) circle [radius=2.5];
\end{tikzpicture}
	\caption{\label{figure patterns for k=3}Some patterns occurring for tropical curves contributing to $\ang{\psi^3 L}_d$ with a tropical line far from the points.}
	\end{center}
	\end{figure}

	\subsection{Case of two line constraints}
	
	We now consider the invariants $\ang{\psi L,\psi L}_d$ obtained by counting rational curves of degree $d$ subject to two $\psi L$-constraints and $3d-3$ $P$-constraints. Recall that now, it is possible to get curves where the two markings subject to a $\psi$-constraint are adjacent to the same vertex. Fortunately, this forces the markings to lie at the unique intersection point between the two line constraints. Thus, it just contributes a summand of $2\ang{\psi P}_d$ that can be considered separately. We now give a closed expression for these invariants in terms of relative invariants satisfying only $P$-constraints. Below, $N_d(2^2)$ is the number of tropical curves of degree $dL$ having two bottom ends of weight $2$ passing through a generic configuration of $3d-3$ points, counted with multiplicity $m_\Gamma=\prod m_V$.

	\begin{theo}\label{thm-2lines}
	We have the following formula:
	$$\boxed{ \begin{array}{r>{\displaystyle}l}
	\ang{\psi L,\psi L}_d= & 4	N_{dL-2E} + 2 N_{d}(2^2) + 4 N_{dL-E}(2) +2N_d \\
	 & + 10 N_{\square_d}+3N_d(3)+\sum_{d_1+d_2=d}\bino{3d-3}{3d_1-1}d_1N_{d_1}N_{d_2} \\
	 & + 2\ang{\psi P}_d. \\
	\end{array} }$$
	\end{theo}
	
\begin{proof}
We compute the invariant by choosing a configuration of constraints as follows: the point constraints are close to each other and far from the vertices of the line constraints which are south-east of the point region. What happens inside the point region, depicted as a blue disk in the Figures, is not relevant: it just consists in rational tropical curves of some degree satisfying the $P$-constraints. The last term of the formula is just the contribution of the curves where the two $\psi$-marking are adjacent to the same vertex.

As we have two $\psi L$-constraints, the set of degree $d$ rational curves satisfying only the point constraints form a $2$-dimensional family. We have two possibilities:
\begin{itemize}[label=-]
\item either we have two strings that can deform independently: the complement of marked points has two connected components each containing two ends,
\item or we have a ``$2$-dimensional string" meaning a unique component of the complement of marked points contains three ends.
\end{itemize}
The strings are fixed by the two $\psi L$-constraints. As the lines have been chosen far from the point region, we only care about the combinatorial types for which the deformation of the strings is not bounded. In the case of two $1$-dimensional strings, such deformations have been described in \cite{GM053}. For a $2$-dimensional string, the deformations are depicted in Figure \ref{figure patterns two psiL one strings}. We now study the two cases.

\begin{itemize}[label=$\circ$]
\item First, concerning the curve with two strings, the patterns are presented on Figure \ref{figure patterns two psiL two strings}. For each of the string, its unique moving vertex has to land on one of the line constraints.
	\begin{itemize}[label=-]
	\item The patterns presented on $(a1)$, $(a1')$, $(a2)$ and $(a2')$ each contribute for $N_{dL-2E}$ because there are $N_{dL-2E}$ tropical curves that we can draw inside the blue region having infinite edges of the right slope that meet the line constraints where we can cut it and place a vertex, as depicted on Figure \ref{figure patterns two psiL two strings}.
	\item The patterns presented on $(b1)$ and $(b1')$ each contribute for $N_{d}(2^2)$. This accounts for the $2$ coming from the determinant between the bounded edges adjacent to the $\psi L$ and $\psi L'$ vertices, and the $\frac{1}{2}$ for the automorphisms of each of the pairs of ends.
	\item The patterns presented on $(b2)$, $(b2')$, $(c2)$ and $(c2')$ each contribute for $N_{dL-E}(2)$.
	\item The patterns presented on $(c1)$ and $(c1')$ each contribute for $N_{dL-E_1-E_2}=N_d$.
	\end{itemize}
The total contribution is
$$4	N_{dL-2E} + 2 N_{d}(2^2) + 4 N_{dL-E}(2) +2N_d.$$
\item Next, we compute the contribution of the curves with a $2$-dimensional string. The possible shapes of $2$-dimensional strings are depicted on Figure \ref{figure patterns two psiL one strings}. We have $8$ possible shapes. The three shapes on the right are the symmetric of the shapes on the left.
	\begin{enumerate}[label=(\alph*)]
	\item The first pattern is completed by a degree $\square_d$ curve inside the blue region. The contribution coming from the $\psi L$-vertices to the multiplicity is just $1$, so the contribution is $N_{\square_d}$ (times $2$ for the symmetric).
	\item The second pattern is completed by a curve of degree $\square'_d$ (which is $\square_d$ with coordinates switched). This time, the contribution of the $\psi L$-vertices and automorphisms is $\frac{2\cdot 2}{2}=2$, so we get $2N_{\square_d}$ (times $2$ for the symmetric).
	\item For the third pattern, we get $2N_{\square_d}$ (times $2$ for the symmetric).
	\item For the fourth pattern, we have two parallel edges leaving the vertex. Considering the position of the marked points on the curve, only one of the edges goes into a bounded component. Thus, the picture is completed by two tropical curves $C_1$ and $C_2$ of respective degrees $d_1+d_2=d$. The first curve $C_1$ is subject to $P$-constraints and gives $\bino{3d-3}{3d_1-1}\cdot N_{d_1}$ solutions. One of its $d_1$ intersection points with the line yields the pattern, and $C_2$ is thus subject to the remaining $3d_2-2$ $P$-constraints, and has to pass through the intersection point with the line, yielding $N_{d_2}(1)=N_{d_2}$. In total, we get:
	$$\sum_{d_1+d_2=d}\bino{3d-3}{3d_1-1}d_1N_{d_1}N_{d_2}.$$
	The contribution of $\psi L$-vertices and automorphisms is $\frac{2}{2}=1$.
	\item For the last pattern, it is completed by a curve of degree $d$ having a bottom end of weight $3$. The contribution of the $\psi L$-points and automorphisms is $\frac{3\cdot 2}{2}=3$.
	\end{enumerate}
	In total, the curves with a $2$-dimensional string give
	$$10 N_{\square_d}+3N_d(3)+\sum_{d_1+d_2=d}\bino{3d-3}{3d_1-1}d_1N_{d_1}N_{d_2}.$$
\end{itemize}
\end{proof}


\newcommand{\ConstrDoubleString}{\clip(-14,-6) rectangle (14,12);
\draw [line width=0.5pt,blue] (0,0)--++(10,10) (0,0)--++ (-12,0) (0,0)--++(0,-8);
\draw [line width=0.5pt,blue] (2,-2)--++(11,11) (2,-2)--++ (-14,0) (2,-2)--++(0,-8);
\draw[fill=blue] (-2,8) circle [radius=2.5];}


\newcommand{\stral}{\draw [line width=1pt] (-11,0)--++(-2,0) (-11,0)--++(0,-3) (-11,0)--++(7.2,7.2)}
\newcommand{\strbl}{\draw [line width=1pt] (-12,-2)--++(-2,0) (-12,-2)--++(0,-3) (-12,-2)--++(8.8,8.8)}

\newcommand{\strac}{\draw [line width=1pt,double] (-2,0)--++(0,-4);
\draw [line width=1pt,thick] (-2,0)-- node[midway,right] {$2$} ++(0,6)}
\newcommand{\strbc}{\draw [line width=1pt,double] (-2,-2)--++(0,-2);
\draw [line width=1pt,thick] (-2,-2)-- node[midway,right] {$2$} ++(0,8)}

\newcommand{\strar}{\draw [line width=1pt] (8,8)--++(0,-3) (8,8)--++(2,2) (8,8)--++(-8,0)}
\newcommand{\strbr}{\draw [line width=1pt] (12,8)--++(0,-3) (12,8)--++(2,2) (12,8)--++(-12,0)}

\begin{figure}
\begin{center}
\begin{tabular}{cc|cc}
\begin{tikzpicture}[line cap=round,line join=round,>=triangle 45,x=0.125cm,y=0.125cm]
\stral;
\draw [line width=1pt] (-10,-2)--++(-2,0) (-10,-2)--++(0,-3) (-10,-2)--++(7.5,7.5);
\ConstrDoubleString
\end{tikzpicture} & \begin{tikzpicture}[line cap=round,line join=round,>=triangle 45,x=0.125cm,y=0.125cm]
\draw [line width=1pt] (-8,0)--++(-4,0) (-8,0)--++(0,-3) (-8,0)--++(5.5,5.5);
\strbl;
\ConstrDoubleString
\end{tikzpicture} & \begin{tikzpicture}[line cap=round,line join=round,>=triangle 45,x=0.125cm,y=0.125cm]
\strar;
\draw [line width=1pt] (10,6)--++(2,2) (10,6)--++(0,-3) (10,6)--++(-10,0);
\ConstrDoubleString
\end{tikzpicture} & \begin{tikzpicture}[line cap=round,line join=round,>=triangle 45,x=0.125cm,y=0.125cm]
\strbr;
\draw [line width=1pt] (6,6)--++(3,3) (6,6)--++(0,-3) (6,6)--++(-6,0);
\ConstrDoubleString
\end{tikzpicture} \\
$(a1)$ & $(a1')$ & $(a2)$ & $(a2')$ \\
\hline
\begin{tikzpicture}[line cap=round,line join=round,>=triangle 45,x=0.125cm,y=0.125cm]
\strac;
\draw [line width=1pt,double] (-4,-2)--++(0,-2);
\draw [line width=1pt,thick] (-4,-2)-- node[midway,left] {$2$} ++(0,10);
\ConstrDoubleString
\end{tikzpicture} & \begin{tikzpicture}[line cap=round,line join=round,>=triangle 45,x=0.125cm,y=0.125cm]
\strbc;
\draw [line width=1pt,double] (-4,0)--++(0,-4);
\draw [line width=1pt,thick] (-4,0)-- node[midway,left] {$2$} ++(0,8);
\ConstrDoubleString
\end{tikzpicture} & \begin{tikzpicture}[line cap=round,line join=round,>=triangle 45,x=0.125cm,y=0.125cm]
\stral;\strbc;
\ConstrDoubleString
\end{tikzpicture} & \begin{tikzpicture}[line cap=round,line join=round,>=triangle 45,x=0.125cm,y=0.125cm]
\strbl;\strac;
\ConstrDoubleString
\end{tikzpicture} \\
$(b1)$ & $(b1')$ & $(b2)$ & $(b2')$ \\
\hline
\begin{tikzpicture}[line cap=round,line join=round,>=triangle 45,x=0.125cm,y=0.125cm]
\stral;\strbr;\ConstrDoubleString
\end{tikzpicture} & \begin{tikzpicture}[line cap=round,line join=round,>=triangle 45,x=0.125cm,y=0.125cm]
\strbl;\strar;\ConstrDoubleString
\end{tikzpicture} & \begin{tikzpicture}[line cap=round,line join=round,>=triangle 45,x=0.125cm,y=0.125cm]
\strar;\strbc;\ConstrDoubleString
\end{tikzpicture} & \begin{tikzpicture}[line cap=round,line join=round,>=triangle 45,x=0.125cm,y=0.125cm]
\strbr;\strac;\ConstrDoubleString
\end{tikzpicture} \\
$(c1)$ & $(c1')$ & $(c2)$ & $(c2')$ \\
\end{tabular}
\caption{\label{figure patterns two psiL two strings} Shape of the curves contributing to $\ang{\psi L,\psi L}$ with two strings.}
\end{center}
\end{figure}
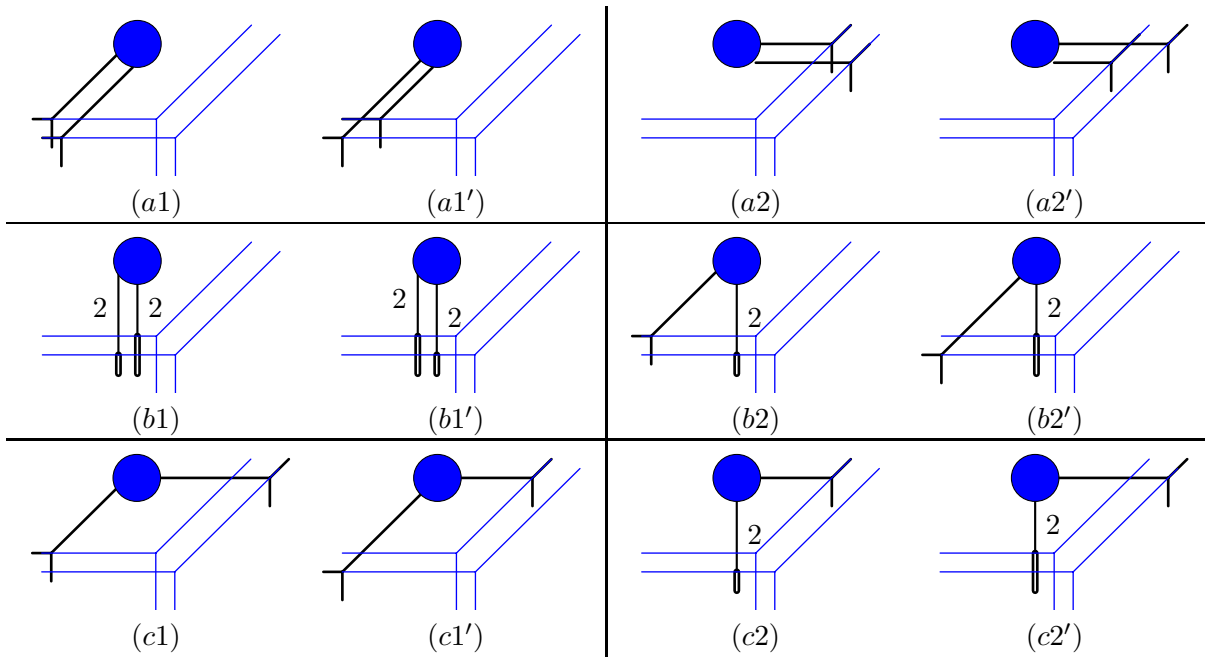

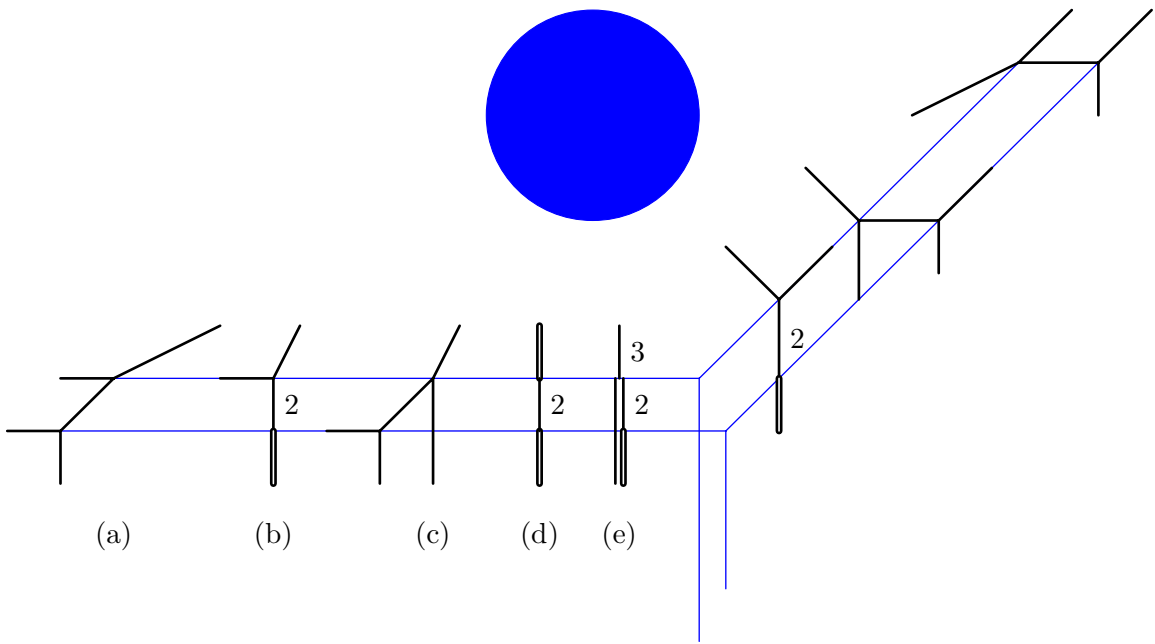
\begin{figure}
\begin{center}
\begin{tikzpicture}[line cap=round,line join=round,>=triangle 45,x=0.35cm,y=0.35cm]
\draw [line width=0.5pt,blue] (0,0)--++(14,14) (0,0)--++ (-24,0) (0,0)--++(0,-10);
\draw [line width=0.5pt,blue] (1,-2)--++(14,14) (1,-2)--++ (-27,0) (1,-2)--++(0,-6);
\draw[blue,fill=blue] (-4,10) circle [radius=4];

\draw[line width=1pt] (-22,0)--++(-2,0) (-22,0)--++(-2,-2)--++(-2,0)++(2,0)--++(0,-2) (-22,0)--++(4,2);
\node (A) at (-22,-6) {(a)};

\draw[line width=1pt] (-16,0)--++(-2,0) ++(2,0)-- node[midway,right] {$2$} ++(0,-2);
\draw[line width=1pt,double] (-16,-2)--++(0,-2);
\draw[line width=1pt] (-16,0)--++(1,2);
\node (B) at (-16,-6) {(b)};

\draw[line width=1pt] (-10,0)--++(0,-4) (-10,0)--++(-2,-2)--++(-2,0)++(2,0)--++(0,-2) (-10,0)--++(1,2);
\node (C) at (-10,-6) {(c)};

\draw[line width=1pt,double] (-6,2)--++(0,-2)++(0,-2)--++(0,-2);
\draw[line width=1pt] (-6,0)-- node[midway,right] {$2$} ++(0,-2);
\node (D) at (-6,-6) {(d)};

\draw[line width=1pt] (-3,0)-- node[midway,right] {$3$} ++(0,2);
\draw[line width=1pt] (-2.85,0)-- node[midway,right] {$2$} ++(0,-2);
\draw[line width=1pt] (-3.15,0)--++(0,-4);
\draw[line width=1pt,double] (-2.85,-2)--++(0,-2);
\node (E) at (-3,-6) {(e)};

\draw[line width=1pt] (3,3)--++(-2,2)++(2,-2)--++(2,2)++(-2,-2)-- node[midway,right] {$2$} ++(0,-3);
\draw[line width=1pt,double] (3,0)--++(0,-2);

\draw[line width=1pt] (6,6)--++(-2,2)++(2,-2)--++(0,-3)++(0,3)--++(3,0)--++(2,2)++(-2,-2)--++(0,-2);

\draw[line width=1pt] (12,12)--++(-4,-2)++(4,2)--++(2,2)++(-2,-2)--++(3,0)--++(2,2)++(-2,-2)--++(0,-2);
\end{tikzpicture}
\caption{\label{figure patterns two psiL one strings} Shape of the curves contributing to $\ang{\psi L,\psi L}$ with one string.}
\end{center}
\end{figure}

\begin{expl}\label{ex-twolines1}
Take $d=2$. The formula from Theorem \ref{thm-2lines} gives
\begin{align*}
\ang{\psi L,\psi L}_2= & 0+0+0+2 \\
	& + 10 + 0 + \bino{3}{2}1\cdot 1\cdot 1 \\
	& + 2\cdot 1 \\
	= & 15+2=17 .\\
\end{align*}
We can also recover it directly by finding the solutions.
\end{expl}

\begin{expl}\label{ex-twolines2}
For $d=3$, the formula from Theorem \ref{thm-2lines} yields
\begin{align*}
\ang{\psi L,\psi L}_3= & 4+0+4\cdot 16+2\cdot 12 \\
	& + 10\cdot 10 + 3\cdot 21 + \bino{6}{2}1\cdot 1\cdot 1+\bino{6}{5}2\cdot 1\cdot 1 \\
	& + 2\cdot 10 \\
	= & 282+20 =302.\\
\end{align*}
\end{expl}

\section{WDVV equation for $\ang{\psi^k L}$}\label{sec-WDVV1}

In this section we prove a WDVV-type equation for the invariants $\ang{\psi^k L}_{d}$. The method to obtain such a formula is similar to the one used in \cite{GM053} to prove Kontsevich's formula: the formula is obtained by considering a suitable evaluation map coupled with a forgetful morphism to $\M_{0,4}$, and intersect it with two different set of constraints, obtaining an equality for which one side contains our wanted invariant. We consider the following evaluation map
$$f:\M_{0,n+3}(\RR^2,d)\cap\psi_1^k\longrightarrow \RR^2\times\RR^2\times\RR^2\times(\RR^2)^n\times\M_{0,4}.$$
We aim to intersect its image with a cycle of the form $\Xi=L\times L_1\times L_2\times\{P_1\}\times\cdots\times\{P_n\}\times\{\lambda\}$, where $L,L_1$ and $L_2$ are generic tropical lines, $P_1,\dots,P_n$ are generic points inside $\RR^2$ and $\lambda\in\M_{0,4}$ is a very big cross-ratio inside some combinatorial type of $\M_{0,4}$. We choose $n$ such that $n+k=3d-2$, so that the domain has dimension equal to the codimension of $\Xi$. The marked points remembered by the forgetful map are the markings associated to $L_1,L_2,L$ and $P_1$. We care about the combinatorial types $L_1L_2//LP$ and $L_1P//L_2L$. In \cite{GM053}, the big cross-ratio is enough to ensure that the curve corresponding to the intersection points between $\Xi$ and the image of $f$ acquires a contracted edge. Unfortunately, this is not the case in our computation and the formula thus witnesses the appearance of a correcting term to palliate this problem.

	\subsection{Curves with a big cross-ratio}

As advertised, we wish to compute the intersection between the image of $f$ and $\Xi$ for two big choices of cross-ratio. The cycle for the combinatorial type $L_1L_2//LP$ is denoted $\Xi_A$ and the cycle for the combinatorial type $L_1P//L_2L$ is denoted by $\Xi_B$. If we do not specify the cycle, we just write $\Xi$. The following lemma describes the curves corresponding to the intersection points. We denote by $Q$ the marking subject to the $\psi^k L$-constraint.

	\begin{lem}\label{lemma shape of curves for big cross-ratio r=1}
	In the setting described above, if the cross-ratio $\lambda$ is big enough, a curve in $f(\M_{0,n+3}(\RR^2,d)\cap\psi_Q^k)\cap\Xi$ satisfies one of the following:
	\begin{enumerate}[label=(\roman*)]
	\item it has a contracted edge contributing to the cross-ratio, and the image can be written as the union of two tropical curves,
	\item the $\psi^k L$-point lies on some end of $L$, and the adjacent connected components of $\Gamma\backslash\P$ are unbounded. Assuming the end is the bottom one (by symmetry), they are of one of the following types:
		\begin{enumerate}[label=(\alph*)]
		\item an end of slope $(0,-1)$ or $(1,1)$,
		\item a string going from an end of slope $(-1,0)$ to $Q$ with slope $(1,-w)$, each adjacent edge is vertical and is completed in a degree $d_i$ curve,
		\item a vertical edge  of slope $(0,w)$, completed in a tropical curve with a bottom end of weight $w$.
		\end{enumerate}
		In the case of $\Xi_A$, the markings corresponding to $L_1$ and $L_2$ lie in the same component, which is not the one in which $P_1$ lies.
		In the case of $\Xi_B$, the marking $L_1$ lies in the same component as $P_1$, and the marking associated to $L_2$ is either in a different component, or on the string.
	\end{enumerate}
	\end{lem}
	
	All these possibilities are depicted on Figure \ref{figure shape of strings for big cross-ratio}. A global deformation is obtained by taking a collection of patterns of type (a), (b) or (c) as in Figure \ref{figure shape of strings for big cross-ratio}.
	
	\begin{figure}
	\begin{center}
	\begin{tabular}{ccc}
	\begin{tikzpicture}[line cap=round,line join=round,>=triangle 45,x=0.65cm,y=0.65cm]
\draw[line width=1pt,blue] (0,0)--++(-3,0)++(3,0)--++(3,3)++(-3,-3)--++(0,-6)--++(0,-2);
\draw (0,-6) node {$\bullet$} node[below left] {$Q$};
\draw (0,-7) node {$\circ$};
\draw[blue,fill=blue] (-1.5,1.5) circle[radius=1];
\draw[line width=1pt] (0,-6)--++(3,3);
\draw[line width=1pt,dotted] (0,-7)--++(3,3);
\end{tikzpicture} & \begin{tikzpicture}[line cap=round,line join=round,>=triangle 45,x=0.65cm,y=0.65cm]
\draw[line width=1pt,blue] (0,0)--++(-3,0)++(3,0)--++(3,3)++(-3,-3)--++(0,-6)--++(0,-2);
\draw (0,-7) node {$\bullet$} node[below right] {$Q$};
\draw (0,-8) node {$\circ$};
\draw[blue,fill=blue] (-1.5,1.5) circle[radius=1];
\curve (1) at (-1,1) {$\Gamma_{i_3}$};
\curve (2) at (-2,0) {$\Gamma_{i_2}$};
\curve (3) at (-3,1) {$\Gamma_{i_1}$};
\draw[line width=1pt] (0,-7)--++(-1,3)--++(-1,2)--++(-1,1)--++(-2,0);
\draw[line width=1pt,dotted] (0,-8)--++(-1,3)--++(-1,2)--++(-1,1)--++(-2,0);
\draw (1) to (-1,-4);
\draw (2) to (-2,-2);
\draw (3) to (-3,-1);
\draw[dotted] (-1,-5) to (-1,-4);
\draw[dotted] (-2,-3) to (-2,-2);
\draw[dotted] (-3,-2) to (-3,-1);
\end{tikzpicture} & \begin{tikzpicture}[line cap=round,line join=round,>=triangle 45,x=0.65cm,y=0.65cm]
\draw[line width=1pt,blue] (0,0)--++(-3,0)++(3,0)--++(3,3)++(-3,-3)--++(0,-6)--++(0,-2);
\draw (0,-6) node {$\bullet$} node[below right] {$Q$};
\draw (0,-7) node {$\circ$};
\draw[blue,fill=blue] (-1.5,1.5) circle[radius=1];
\draw[line width=1pt,dotted] (0,-7)--++(0,1);
\curve (1) at (0,1) {$\Gamma_j$};
\draw[line width=1pt] (0,-6) to (1);
\end{tikzpicture} \\
(a) & (b) & (c) \\
	\end{tabular}
	\caption{\label{figure shape of strings for big cross-ratio}The possible shapes of strings emanating from $Q$ along with a small deformation.}
	\end{center}
	\end{figure}
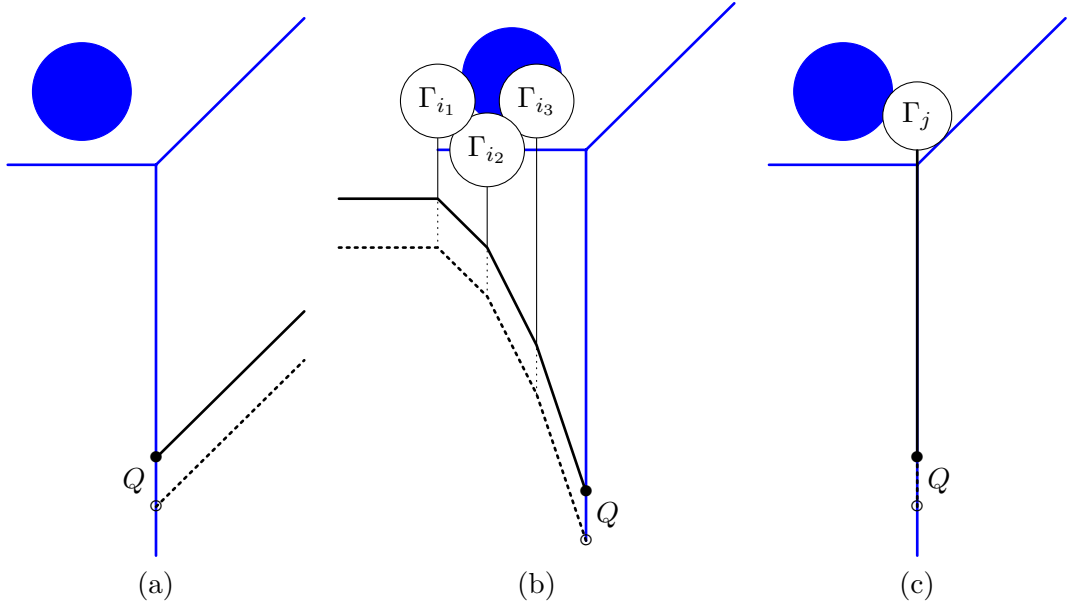
	
	\begin{proof}
	We consider a tropical solution. We momentarily forget about the $\M_{0,4}$-condition, and the marked points corresponding to the lines $L_1,L_2$ as well (they can be added afterwards), getting a  curve in $\M_{0,n+1}(\RR^2,d)\cap\psi_Q^k$ satisfying the $P$-constraints and the $\psi^k L$-constraint. The solution can now be deformed in a $1$-parameter family of curves, which are subject to point constraints and a $\psi^k L$-constraint. We precisely lack one point constraint to get a finite number of solutions.
	
	\begin{enumerate}[label=(\roman*)]
	\item Assume the combinatorial type has a contracted edge: its length can vary without affecting the point constraints. Adding back the $L_1,L_2$ and cross-ratio constraints, by genericity, the length needs to be fixed by the cross-ratio constraint. Removing the contracted edge from the parametrization, the tropical curve is split into two tropical curves. This finishes the first part of the statement.
	
	\item We now assume there is no contracted edge. We wish to describe the position of the point constraints on the curve. We artificially add a generic point constraint, so that the initial solution can no longer be deformed, and the position of points on the curve is now described by Lemma \ref{lemma position points k=1}. Removing this additional point constraint, the position of the original point constraints on the curve is one of the following cases:
	\begin{enumerate}[label=(\Alph*)]
	\item \textbf{The complement of marked points has a component with two ends.} This occurs if the additinal point constraint is not adjacent to the bounded component whose existence is ensured by Lemma \ref{lemma position points k=1}. In that case, the deformation given by the $1$-parameter family is obtained by deforming the string linking both ends inside the component as in \cite{GM053}. Results from \cite{GM053} ensure that for such combinatorial types, the image under $\mathrm{ft}_4$ inside $\M_{0,4}$ is either bounded, or the string is formed only by two ends with a unique adjacent bounded edge. Combinatorial types of the latter form do not contribute any solution as the cross-ratio condition is not sufficient to fix the curve. In any case, for $\lambda$ big enough, there is no solution.
	\item \textbf{There is no bounded component in the complement of marked point.} This case occurs if the additional point constraint was adjacent to the bounded component whise existence is ensured by Lemma \ref{lemma position points k=1}. In that case, every adjacent component to $Q$ has a unique end, and the deformation is obtained by moving $Q$ along the line $L$. If the position of $Q$ is bounded, then so is the cross-ratio. We thus furthermore assume that it is possible to send $Q$ to infinity along the end of $L$ on which it lies, and deforming the ``strings" (\textit{i.e.} paths from $Q$ to the ends inside each component) along with it. By symmetry, assume the end of $L$ on which the $\psi^k L$-point lies is the bottom one. Let us consider a string from the $\psi^k L$-point to an end of $\Gamma$. We now describe the shape of the string according to its slope at $Q$. To do this, we use the following lemma.
	\begin{lem}\label{lemma meeting the points}
	Let $h:\Gamma\to\RR^2$ be a tropical curve meeting the $P$-constraints and the $\psi^k L$-constraint in the setting described at the beginning of Section \ref{sec-WDVV1}, but not the cross-ratio constraint. Assume the complement of marked points has only unbounded components with a unique end, leading to a $1$-parameter family of solutions. Let $\gamma$ be a bounded edge of $\Gamma$ that is deformed in the $1$-parameter family. Then:
		\begin{enumerate}[label=(\Roman*)]
		\item either $\gamma$ lies on some string of $\Gamma$,
		\item or $\gamma$ is adjacent to some string, and its extremity not in the string lies in the point region.
		\end{enumerate}
	\end{lem}
	
	\begin{proof}
	The description of the deformation by moving the strings ensures that any moving edge either belong to one of the strings, or is adjacent to it: every other edge is fixed by the point conditions. We now assume $\gamma$ is adjacent to a string. The line containing $\gamma$ splits $\RR^2$ in two. If the point region was contained in one of the two halves, the other half would contain an unbounded end and no marked point, leading to a second end in the connected component of $\Gamma\backslash\P$ containing $\gamma$, which is absurd. So the line spanned by $\gamma$ passes through the point region. If the extremity of $\gamma$ was not in the point region, we get a trivalent vertex and can do the same reasoning with the adjacent edges to get another end in the component. Thus, the extremity has to lie in the point region.
	\end{proof}
	
	We now describe the strings. The corresponding pictures ca be found in Figure \ref{figure slope of the edges adjacent to strings}. In each case, we consider the component cut at the marked points.
		\begin{enumerate}[label=(\alph*)]
		\item If the slope of the string at $Q$ is $(u,v)$ with $u\geqslant 2$, by balancing, the component needs to have at least two ends of slope $(1,1)$, which is impossible because there is only one end.
		\item Assume the slope is $(1,v)$. Assume there is a vertex $V$ on the string, and let $(a,b)$ be the slope of the edge adjacent to the string. By Lemma \ref{lemma meeting the points}, we must have $a\leqslant 0$ since its extremity lies in the point region. If we had $a<0$, the balancing would force the component to contain at least two ends of slope $(1,1)$, which is impossible because there is only one end. Hence, $a=0$. However, for a vertical edge, all the marked points are on its left since the line $L$ has been chosen south-east of the point region. Thus, there cannot be any vertex on the string, which is thus only an end of slope $(1,1)$.
		\item If the slope is $(0,v)$, it completes into a tropical curve of some degree with a bottom end of weight $v$.
		\item If the slope is $(u,v)$ with $u\leqslant -2$, the component would contain at least two ends of slope $(-1,0)$, which is impossible because there is only one end.
		\item If the slope is $(-1,v)$, let $V$ be a vertex on the string. Let $(a,b)$ be the slope of the edge adjacent to the string at $V$, which is a bounded edge. By Lemma \ref{lemma meeting the points}, we must have $a\geqslant 0$. If $a>0$, as before, we would have at least two ends of slope $(-1,0)$ in the component, which is impossible. So we must have $a=0$. the bounded edge is thus completed by a tropical curve with a bottom end of weight $b$ in the point region.
		\end{enumerate}
	We need to finish with the description of the position of the $L_1$-marking, $L_2$-marking and $P_1$-marking. As it lies in the point region, the $P_1$-marking has to belong to some component $\Gamma_i$. The $L_1$ and $L_2$-marking can either belong to a component $\Gamma_i$ or to the string that deforms to infinity. It is easy to check that the only possibility of dispatching these conditions that leads to a non-constant cross-ratio in each case ($\Xi_A$ or $\Xi_B$) is as given in the lemma.
	\end{enumerate}
	\end{enumerate}
	\end{proof}

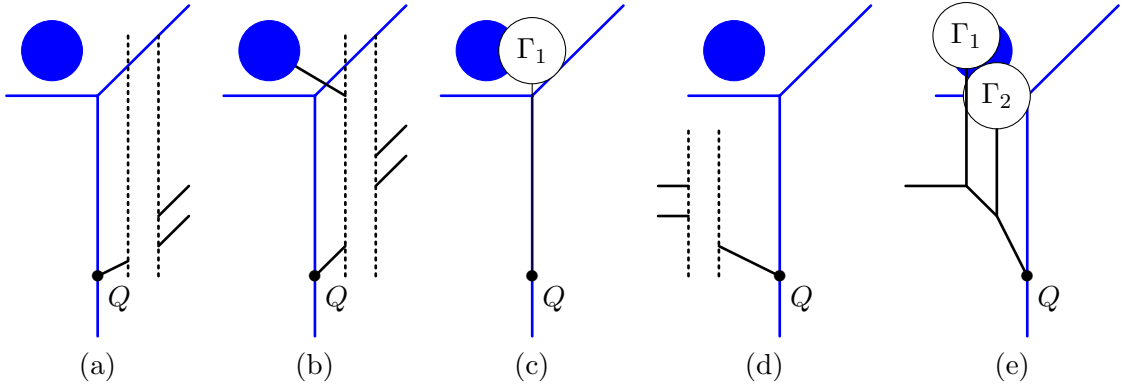
\begin{figure}
\begin{center}
\begin{tabular}{ccccc}
\begin{tikzpicture}[line cap=round,line join=round,>=triangle 45,x=0.4cm,y=0.4cm]
\draw[line width=1pt,blue] (0,0)--++(-3,0)++(3,0)--++(3,3)++(-3,-3)--++(0,-6)--++(0,-2);
\draw (0,-6) node {$\bullet$} node[below right] {$Q$};
\draw[blue,fill=blue] (-1.5,1.5) circle[radius=1];
\draw[line width=1pt,dotted] (1,-6)--++(0,8);
\draw[line width=1pt,dotted] (2,-6)--++(0,8);
\draw[line width=1pt] (0,-6)--++(1,0.5);
\draw[line width=1pt] (2,-4)--++(1,1);
\draw[line width=1pt] (2,-5)--++(1,1);
\end{tikzpicture} & \begin{tikzpicture}[line cap=round,line join=round,>=triangle 45,x=0.4cm,y=0.4cm]
\draw[line width=1pt,blue] (0,0)--++(-3,0)++(3,0)--++(3,3)++(-3,-3)--++(0,-6)--++(0,-2);
\draw (0,-6) node {$\bullet$} node[below right] {$Q$};
\draw[line width=1pt] (-1.5,1.5)--(1,0);
\draw[blue,fill=blue] (-1.5,1.5) circle[radius=1];
\draw[line width=1pt,dotted] (1,-6)--++(0,8);
\draw[line width=1pt,dotted] (2,-6)--++(0,8);
\draw[line width=1pt] (0,-6)--++(1,1);
\draw[line width=1pt] (2,-2)--++(1,1);
\draw[line width=1pt] (2,-3)--++(1,1);
\end{tikzpicture} & \begin{tikzpicture}[line cap=round,line join=round,>=triangle 45,x=0.4cm,y=0.4cm]
\draw[line width=1pt,blue] (0,0)--++(-3,0)++(3,0)--++(3,3)++(-3,-3)--++(0,-6)--++(0,-2);
\draw (0,-6) node {$\bullet$} node[below right] {$Q$};
\draw[blue,fill=blue] (-1.5,1.5) circle[radius=1];
\curve (1) at (0,1.5) {$\Gamma_1$};
\draw (0,-6) to (1);
\end{tikzpicture} & \begin{tikzpicture}[line cap=round,line join=round,>=triangle 45,x=0.4cm,y=0.4cm]
\draw[line width=1pt,blue] (0,0)--++(-3,0)++(3,0)--++(3,3)++(-3,-3)--++(0,-6)--++(0,-2);
\draw (0,-6) node {$\bullet$} node[below right] {$Q$};
\draw[blue,fill=blue] (-1.5,1.5) circle[radius=1];
\draw[line width=1pt,dotted] (-3,-6)--++(0,5);
\draw[line width=1pt,dotted] (-2,-6)--++(0,5);
\draw[line width=1pt] (0,-6)--++(-2,1);
\draw[line width=1pt] (-3,-4)--++(-1,0);
\draw[line width=1pt] (-3,-3)--++(-1,0);
\end{tikzpicture} & \begin{tikzpicture}[line cap=round,line join=round,>=triangle 45,x=0.4cm,y=0.4cm]
\draw[line width=1pt,blue] (0,0)--++(-3,0)++(3,0)--++(3,3)++(-3,-3)--++(0,-6)--++(0,-2);
\draw (0,-6) node {$\bullet$} node[below right] {$Q$};
\draw[blue,fill=blue] (-1.5,1.5) circle[radius=1];
\draw[line width=1pt] (0,-6)--++(-1,2)--++(-1,1)--++(-2,0);
\curve (2) at (-1,0) {$\Gamma_2$};
\curve (1) at (-2,2) {$\Gamma_1$};
\draw[line width=1pt] (-1,-4) to (2);
\draw[line width=1pt] (-2,-3) to (1);
\end{tikzpicture} \\
(a) & (b) & (c) & (d) & (e) \\
\end{tabular}
\caption{\label{figure slope of the edges adjacent to strings}Illustration for proof of Lemma \ref{lemma shape of curves for big cross-ratio r=1}.}
\end{center}
\end{figure}

		\subsection{Statement of the recursive formula}
	
Using Lemma \ref{lemma shape of curves for big cross-ratio r=1}, it is now possible to write down the recursive formula.

	\begin{theo}\label{theorem recursive formula r=1}
	One has the following identity:
	$$\boxed{ \begin{array}{>{\displaystyle}r>{\displaystyle}c>{\displaystyle}l}
	\ang{\psi^k L}_d = & \sum_{\substack{d_1+d_2=d \\ d_1,d_2\geqslant 1}} & \bino{3d-3-k}{3d_1-2} d_1^2 d_2 N_{d_1}(d_2\ang{\psi^k L}_{d_2}+\ang{\psi^{k-1}P}_{d_2}) \\
	& & - \bino{3d-3-k}{3d_1-1} d_1^3 d_2 N_{d_1}\ang{\psi^k L}_{d_2} \\
	& & + \bino{3d-3-k}{3d_1-2} d_1^2 d_2 N_{d_1}\ang{\psi^{k-1}P}_{d_2} \\
	& & - \bino{3d-3-k}{3d_1-1} d_1^3 N_{d_1} \ang{\psi^{k-1}P}_{d_2} \\
	& & + \bino{3d-3-k}{3d_1-3} d_1N_{d_1}(d_2\ang{\psi^{k-1}L}_{d_2}+\ang{\psi^{k-2}P}_{d_2}) \\
	& & - \bino{3d-3-k}{3d_2-2} d_1^2 N_{d_1}\ang{\psi^{k-1}L}_{d_2} \\
	& + 3\sum &  \frac{r^{S-s}}{(r!)^2 l!} \frac{1}{\sigma} \bino{3d-k-2}{\widehat{k_1},\cdots,\widehat{k_S}} \left( \frac{d}{3d-k-2}\sum_{i=1}^S \frac{\widehat{k_i}\widehat{d_i}}{\widehat{w_i}}-\sum_{i=1}^S \frac{\widehat{d_i}^2}{\widehat{w_i}} \right)\prod_{i=1}^S\widehat{w_i}\widehat{N}_{\widehat{d_i}}(\widehat{w_i})  \\
	\end{array} }$$
	
	where the second sum is over the following data:
	\begin{itemize}[label=-]
	\item positive integers $S\geqslant 1$, $l\geqslant 1$ and $r,s\geqslant 0$ with
	$$\left\{ \begin{array}{l}
	2r+l + s =k+2, \\
	s \leqslant S,\\
	s+r\leqslant l \leqslant d.\\
\end{array}	\right. $$
	\item integers families $(d_i,w_i)_{1\leqslant i\leqslant S-s},(\widetilde{d_j},\widetilde{w_j})_{S-s +1\leqslant j\leqslant S}$ with $1\leqslant\widehat{w_i}\leqslant \widehat{d_i}$ satisfying the conditions
	$$ S\leqslant\sum_{i=1}^S \widehat{w_i}=l -r, \ \sum_{i=1}^S \widehat{d_i}=d-r.$$
	\end{itemize}
	
The term $\sigma$ is the number of automorphisms of the splitting $(d_i,w_i),(\widetilde{d_j},\widetilde{w_j})$ (as in the Caporaso-Harris formula counting irreducible curves, see Theorem 4.5 in \cite{gathmann2007caporaso}), and $k_i=3d_i-w_i$, $\widetilde{k_j}=3\widehat{d_j}-\widetilde{w_j}-1$. The hat denotes either with tilde or without tilde according to the index in $[\![1;S]\!]$.
	\end{theo}
	
	Factorizing the binomial coefficients, the formula also writes	
	$$\boxed{ \begin{array}{>{\displaystyle}r>{\displaystyle}c>{\displaystyle}l}
	\ang{\psi^k L}_d = & \sum_{\substack{d_1+d_2=d \\ d_1,d_2\geqslant 1}} & -\bino{3d-3-k}{3d_1-1} d_1^3 N_{d_1}(d_2\ang{\psi^k L}_{d_2}+\ang{\psi^{k-1}P}_{d_2}) \\
	& &+ \bino{3d-3-k}{3d_1-2} d_1^2 N_{d_1}( d_2^2\ang{\psi^k L}_{d_2} + 2d_2\ang{\psi^{k-1}P}_{d_2} -\ang{\psi^{k-1}L}_{d_2} ) \\
	& & +\bino{3d-3-k}{3d_1-3} d_1N_{d_1}(d_2\ang{\psi^{k-1}L}_{d_2}+\ang{\psi^{k-2}P}_{d_2}) \\
	& + 3\sum &  \frac{r^{S-s}}{(r!)^2 l!} \frac{1}{\sigma} \bino{3d-k-2}{\widehat{k_1},\cdots,\widehat{k_S}} \left( \frac{d}{3d-k-2}\sum_{i=1}^S \frac{\widehat{k_i}\widehat{d_i}}{\widehat{w_i}}-\sum_{i=1}^S \frac{\widehat{d_i}^2}{\widehat{w_i}} \right)\prod_{i=1}^S\widehat{w_i}\widehat{N}_{\widehat{d_i}}(\widehat{w_i})  \\
	\end{array} }$$
	
	\begin{rem}
	In the first form, the formula is comprised of six terms in the sum over $d_1,d_2$ with $d_1+d_2=d$ and a second sum. We explain the origin of these contributions:
	\begin{itemize}[label=$\ast$]
	\item The first terms are obtained as in the proof of Kontsevich's formula \cite{GM053}: a big cross-ratio forces a contracted edge that splits the curve into $C_1$ and $C_2$, only one of them containing the $\psi^k L$-constraint. Both curves $C_1$ and $C_2$ can be fixed by the constraints, leading to the first two summands. The difference is that the contracted edge can be adjacent to the $\psi^k L$-constraint, so that only $C_1$ or $C_2$ is fixed by the constraints, and it determines the position of the other, leading to the other four summands.
	\item As the big cross-ratio is not enough to ensure the appearance of a contracted edge, we have the correction term, which consists in the deformation patterns described by Lemma \ref{lemma shape of curves for big cross-ratio r=1}:
		\begin{itemize}[label=-]
		\item Curves $(\Gamma_i,w_i)$ appear in strings of type (b) , curves $(\widetilde{\Gamma_j},\widetilde{w_j})$ for strings of type (c) (see Figure \ref{figure shape of strings for big cross-ratio}). In total, there are $S$ of them.
		\item There are $l$ ends (as \textit{lower ends}) of slope $(0,-1)$ and $r$ ends of slope $(1,1)$. By balancing, there are $r$ strings of type (b) leaving on the left.
		\item There are $s\leqslant S$ strings of type (c), and thus exactly $s$ $(\widetilde{\Gamma_j},\widetilde{w_j})$. The remaining $S-s$ components dispatch over the $r$ strings of type (b).
		\end{itemize}
	\end{itemize}
	\end{rem}
	
	\begin{rem}
	The term $\sigma$ is equal to $\prod _{d,w}n_{(d,w)}!\prod _{\widetilde{d},\widetilde{w}}n_{(\widetilde{d},\widetilde{w})}!$, where $n_\bullet$ is the number of $\bullet$ in the splitting.
	\end{rem}

	\begin{proof}
	We evaluate the intersection number between $f(\M_{0,n+3}(\RR^2,d)\cap\psi_Q^k)$ and $\Xi_A/\Xi_B$ for a big value of $\lambda$ on the following two combinatorial types: $(L_1L_2//LP)$ and $(PL_1//LL_2)$. Lemma \ref{lemma shape of curves for big cross-ratio r=1} describes the tropical curves corresponding to this intersection, which we now gather.
	
\textbf{Step 1, the $(L_1L_2//LP)$ left side:} 
	\begin{enumerate}[label=(\Roman*)]
	\item We assume $\Gamma$ has a contracted edge. This edge splits the curve in two components $C_1$ and $C_2$ of respective degrees $d_1+d_2=d$. The marked points are distributed among the two components. We have several possibilities according to whether the contracted edge is adjacent to the vertex with the $\psi^k L$-constraint or not, and whether $d_1=0$ or $d_1\neq 0$. Let $n_1$ (resp. $n_2$) be the number of marked point on $C_1$ (resp. $C_2$).
		\begin{enumerate}[label=(\alph*)]
		\item If $d_1=0$, $C_1$ is contracted to the intersection point of $L_1$ and $L_2$, so that $C_2$ is of degree $d$ and has to pass through $L_1\cap L_2$. The contribution is equal to
		$$(Ia)^\mathrm{left}=\ang{\psi^k L}_d.$$
		\item Assume now that $d_1\neq 0$, and the contracted edge is not adjacent to the $\psi^k L$-point. As $C_1$ is subject to $n_1$ point constraints, and $C_2$ to $n_2$ point constraints as well as the $\psi^k L$ constraint, one has
	$$\left\{\begin{array}{l}
	n_1\leqslant 3d_1-1 , \\
	n_2\leqslant 3d_2-1-k.\\
\end{array}	 \right.$$
	Adding the two relations, we see that $n=n_1+n_2\leqslant 3(d_1+d_2)-2-k=n$, so we have in fact equality. Thus, we can pick independently $C_1$ and $C_2$ subject to point constraints ($N_{d_1}\ang{\psi^k L}_{d_2}$), pick an intersection point between them where to insert the contracted edge ($d_1d_2$), and the intersection point between $L_1$, $L_2$ and $C_1$ ($d_1^2$). This contributes a term
	$$(Ib)^\mathrm{left}=\sum_{d_1+d_2=d} \bino{3d-3-k}{3d_1-1} d_1^3d_2 N_{d_1}\ang{\psi^k L}_{d_2}.$$
		\item We now assume that the contracted edge is adjacent to the $\psi^k L$-point, so that the valency of the vertex inside $C_2$ is only $k+2$ instead of $k+3$. Now, $C_2$ is subject to $n_2$ point condition and a $\psi^{k-1}L$-constraint. We now have
		$$\left\{\begin{array}{l}
	n_1\leqslant 3d_1-1 , \\
	n_2\leqslant 3d_2-1-(k-1).\\
\end{array}	 \right.$$
		Adding the two rows, we do not get an equality anymore, as both sides differ by one. So we are in one of the following situations:
		\begin{itemize}[label=$\circ$]
		\item We have $n_1=3d_1-1$. In this situation, we pick $C_1$ fixed by the points constraints ($N_{d_1}$), an intersection point with $L_1$ and $L_2$ ($d_1^2$). The curve $C_2$ has a $\psi^{k-1}P$-constraint ($\ang{\psi^{k-1}P}_{d_2}$) at one of the intersection point between $C_1$ and $L$, of which there are $d_1$. We get a contribution of
		$$(Ic)^\mathrm{left}=\sum_{d_1+d_2=d} \bino{3d-3-k}{3d_1-1} d_1^3N_{d_1}\ang{\psi^{k-1}P}_{d_2}.$$
		\item We have $n_1=3d_1-2$. In this situation, $C_2$ is fixed by the $P$-constraints and $\psi^{k-1}L$-constraint ($\ang{\psi^k L}_{d_2}$). The curve $C_1$ has then to pass through the $\psi^k L$-marking ($N_{d_1}$). We then pick the intersection points with $L_1$ and $L_2$ ($d_1^2$). We get a contribution of
		$$(Ic')^\mathrm{left}=\sum_{d_1+d_2=d} \bino{3d-3-k}{3d_1-2} d_1^2 N_{d_1}\ang{\psi^{k-1}L}_{d_2}.$$
		\end{itemize}
		\end{enumerate}
	
	\item We now compute the correction term on the left side, which corresponds to tropical curves for which the $\psi^k L$-point constraint goes to infinity on one of the ends of $L$. We assume that the correction goes in one of the directions, \textit{e.g.} the bottom one, and later multiply by $3$. The curves are described by Lemma \ref{lemma shape of curves for big cross-ratio r=1}. The deformation pattern is described by the following data:
		\begin{itemize}[label=$\ast$]
		\item The number $r$ of ends in direction $(1,1)$ (type (a)), also equal to the number of strings of type (b) going to the left by balancing.
		\item The number $l$ of ends in direction $(0,-1)$,
		\item the family of weights and degrees $(\widetilde{d_j},\widetilde{w_j})_j$ of curves in strings of type (c). We have $\widetilde{w_j}\leqslant\widetilde{d_j}$, and the family has $s$ elements.
		\item The family of weights and degrees $(d_i,w_i)_i$ of the curves adjacent to strings of type (b). We have $w_i\leqslant d_i$, and the family has $S-s$ elements.
		\item How the $S-s$ curves are grouped together among the $r$ strings of type (b).
		\end{itemize}
		Moreover, they have to satisfy the following equations, coming from the $\psi^k L$-constraint and the balancing condition:
		$$\left\{\begin{array}{l}
		2r+ l + s =k+2 \text{ by the }\psi^kL\text{ constraint,} \\
		\sum_i w_i+\sum_j \widetilde{w_j} +r=l \text{ by vertical balancing,}\\
		\sum_i d_i + \sum_j d_j =d-r \text{ by degree conservation.}\\
\end{array}		 \right.$$
It also implies that $d\geqslant l\geqslant s +r$. Concretely, pick $r,l,s,S$ fulfilling the inequality, then choose the $(\widetilde{d_j},\widetilde{w_j})$ and finish with the $(d_i,w_i)$ along with their grouping.

Removing the strings, the curve is cut into several connected components $h_i:\Gamma_i\to\RR^2$ of degrees adding up to $d-r$, and the marked point $P_1$ belongs to one of them. Then, each of the lines $L_1$ and $L_2$ intersects the tropical curve $\Gamma$ at the curves $\Gamma_i$, and the strings going to infinity. Thanks to Lemma \ref{lemma shape of curves for big cross-ratio r=1}, the only possibility of location of the intersection points leading to an increasing cross-ratio $(L_1L_2//LP)$ when the string goes to infinity is when the intersection points of $L_1$ and $L_2$ with $\Gamma$ are chosen inside the same component different from the one containing $P_1$.

We then have to distribute the remaining marked points among the components. Let us momentarily forget about the disjunction between tilde and not tilde components, and index them by $[\![1;S]\!]$. We first have a factor $\frac{1}{\sigma}$ as in the Caporaso-Harris formula \cite{gathmann2007caporaso} Theorem 4.5 for irreducible curves accounting for the symmetries of the splitting. Then, for each component, let $\widehat{k_i}$ be the number of points needed to mark it: $k_i=3d_i-w_i$ or $\widetilde{k_i}=3\widetilde{d_i}-\widetilde{w_i}-1$.

$$(II)^\mathrm{left} = \frac{1}{\sigma}\left(\sum_{i=1}^S \bino{n-1}{\widehat{k_1},\cdots,\widehat{k_i}-1,\cdots,\widehat{k_S}} \left(\sum_{j\neq i}\frac{\widehat{d_j}^2}{\widehat{w_j}}\right)\right)\prod \widehat{w_i} \widehat{N}_{\widehat{d_i}}(\widehat{w_i}),$$
where the first sum accounts for the component that contains $P_1$, and the binomial coefficient for the distribution of the remaining marked points. The second sum accounts for the choice of the component that contains intersection points with $L_1$ and $L_2$. The multiplicity then contains a term $\widehat{d_j}^2$ for the number of possible intersection points, and a $\frac{1}{\widehat{w_j}}$ since the latter does not appear in the computation of the multiplicity, so it removes the contribution from $\prod \widehat{w_i}$. The fact that factors of $\widehat{w_i}$ show up in the multiplicity except for the one leading to the component with $L_1$ and $L_2$ can be seen be computing the determinant of the map $f$ in local coordinates similar as has been done in Propositions \ref{prop-mult1} and \ref{prop-mult2}. 

We now practice some standard computation to get an easier expression (with hats everywhere):
\begin{align*}
 & \sum_{i=1}^S \bino{n-1}{k_1,\cdots,k_i-1,\cdots,k_S} \left(\left(\sum_{j}\frac{d_j^2}{w_j}\right)-\frac{d_i^2}{w_i}\right)\\
  = &  \left( \sum_j \frac{d_j^2}{w_j}\right)\left( \sum_{i=1}^S \bino{n-1}{k_1,\cdots,k_i-1,\cdots,k_S}\right)-\sum_{i=1}^S \bino{n}{k_1,\cdots,k_S}\frac{k_i}{n}\frac{d_i^2}{w_i} \\
  = &  \bino{n}{k_1,\cdots,k_S} \sum_{i=1}^S\frac{d_j^2}{w_j} \left( 1-\frac{k_j}{n}\right) . \\
\end{align*}

Finally, we have to account for the grouping of the components into the $r$ strings going into the $\psi^k L$-point. As the curves are distinguished by the marked points that they contain, it is a combinatorial factor. Moreover, we have a $\frac{1}{e!}$ term coming from the automorphisms of the curve, where $e$ is the number of ends adjacent to the $\psi^k L$-point in direction $(-1,0)$. We have to compute the sum
$$\sum_{e=0}^{r-1}\frac{1}{e!}\left\{\begin{matrix}
 S-s \\
 r-e \\
\end{matrix}\right\} = \frac{r^{S-s}}{r!},$$
where the braces denote second kind Stirling numbers: number of ways of splitting a set of here $S-s$ elements into a fixed number of unlabelled sets. The equality comes from counting the number of maps from $[\![S-s]\!]$ to $[\![1;r]\!]$ up to the action of the symmetric group $S_r$.
	\end{enumerate}

\textbf{Step 2, the $(PL_1//LL_2)$ right side:} 
	\begin{enumerate}[label=(\Roman*)]
	\item We assume $\Gamma$ has a contracted edge, and proceed as in step 1 with similar notations: $\Gamma$ is split in $C_1$ and $C_2$ with respective degrees $d_1$ and $d_2$. The contracted edge can be adjacent to the $\psi^k L$-constraint or not, leading to two different possibilities.
		\begin{enumerate}[label=(\alph*)]
		\item Assume the contracted edge is not adjacent to the $\psi^k L$-constraint. We have
		$$\left\{\begin{array}{l}
	n_1\leqslant 3d_1-1 , \\
	n_2\leqslant 3d_2-1-k.\\
\end{array}	 \right.$$
		As the sum becomes an equality, we have equality in each term. The solving for $C_1$ yields $N_{d_1}$ while the solving for $C_2$ gives $\ang{L_2,\psi^{k}L}_{d_2}$. We have $d_1d_2$ for the intersection points between $C_1$ and $C_2$, an additional $d_1$ for intersection points between $C_1$ and $L_1$, and the divisor equation to replace $\ang{L_2,\psi^k L}_{d_2}$ by $d_2\ang{\psi^k L}_{d_2}+\ang{\psi^{k-1}P}_{d_2}$. In total, the contribution is
		$$(Ia)^\mathrm{right}=\sum_{d_1+d_2=d} \bino{3d-3-k}{3d_1-2} d_1^2 d_2 N_{d_1}(d_2\ang{\psi^k L}_{d_2}+\ang{\psi^{k-1}P}_{d_2}).$$

		\item Now, assume that the contracted edge is adjacent to the $\psi^k L$-point, now $C_2$ is only subject to a $\psi^{k-1}L$-condition. We have
		$$\left\{\begin{array}{l}
	n_1\leqslant 3d_1-1 , \\
	n_2\leqslant 3d_2-1-(k-1),\\
\end{array}	 \right.$$
		and we do not get an equality anymore when adding the two rows. As in step 1, we get two cases.
			\begin{itemize}[label=$\circ$]
			\item If $n_1=3d_1-1$, $C_1$ is fixed by the point constraints ($N_{d_1}$), it has $d_1$ intersection points with $L_1$. The curve $C_2$ has a $\psi^{k-1}P$-constraint at one of the intersection points between $C_1$ and $L$ ($\ang{\psi^{k-1}P}_{d_2}$), of which there are $d_1$. It has also $d_2$ intersection points with $L_2$. The contribution is
			$$(Ib)^\mathrm{right}=\sum_{d_1+d_2=d}\bino{3d-3-k}{3d_1-2} d_1^2 d_2 N_{d_1}\ang{\psi^{k-1} P}_{d_2}.$$
			\item If $n_1=3d_1-2$, $C_2$ is fixed by the $n_2$ $P$-constraints and the $\psi^{k-1}L$-constraint, yielding $\ang{L_2,\psi^{k-1}L}_{d_2}=d_2\ang{\psi^{k-1}L}_{d_2}+\ang{\psi^{k-2}P}_{d_2}$ by the divisor equation. Then, $C_1$ has to pass through the $\psi^k L$-point ($N_{d_1}$). It then has $d_1$ intersection points with $L_1$. We get the contribution
			$$(Ib')^\mathrm{right}=\sum_{d_1+d_2=d}\bino{3d-3-k}{3d_1-3}d_1 N_{d_1}(d_2\ang{\psi^{k-1}L}_{d_2}+\ang{\psi^{k-2}P}_{d_2}),$$
		\end{itemize}
		
		\end{enumerate}
	
	\item We now proceed as in Step 1 (II) to compute the correction term coming from the curves with the $\psi^k L$-point going to infinity. The curves are still described by Lemma \ref{lemma shape of curves for big cross-ratio r=1}, and exactly by the same data as in Step 1 (II). The difference comes from the relative position of the intersection points between $L_1,L_2$ and $\Gamma$, the $\psi^k L$-point, and $P_1$. The marked point $P_1$ lies in one of the $\Gamma_i$. This time, the only possibility leading to a big cross-ratio $(PL_1//LL_2)$ is when the intersection point with $L_1$ is in the same component as $P_1$, and the intersection point with $L_2$ is either on the string near infinity or in a different component.
	
	We similarly momentarily forget about the difference between tilde and non-tilde components to count the curves. Everything is to understand with a hat which we delete to avoid burdening notations.
	
	$$(II)^\mathrm{right} = \frac{1}{\sigma} \left( \sum_{i=1}^S \bino{n-1}{k_1,\cdots,k_i-1,\cdots,k_S} \left(r\frac{d_i}{w_i}+\sum_{j\neq i} \frac{d_id_j}{w_i}\right) \right)\prod w_i N_{d_i}(w_i).$$
	
	In the inner sum, $r\frac{d_i}{w_i}$ corresponds to $L_2$ intersecting the string, and the sum over $L_2$ intersecting the other components. We now carry out some simplifications.
	
\begin{align*}
 & \sum_{i=1}^S \bino{n-1}{k_1,\cdots,k_i-1,\cdots,k_S} \left(r\frac{d_i}{w_i}+\sum_{j\neq i} \frac{d_id_j}{w_i}\right) \\
= & \bino{n}{k_1,\cdots,k_S}\sum_{i=1}^S \frac{k_i}{n}\left(r\frac{d_i}{w_i}+\frac{d_i}{w_i}\left(\sum_{j}d_j\right) - \frac{d_i^2}{w_i}\right) \\
= & \bino{n}{k_1,\cdots,k_S}\sum_{i=1}^S \frac{k_i}{n}\left(r\frac{d_i}{w_i}+\frac{d_i}{w_i}(d-r) - \frac{d_i^2}{w_i}\right) \\
= & \bino{n}{k_1,\cdots,k_S}\sum_{i=1}^S \frac{k_i}{n}\left(\frac{dd_i}{w_i} - \frac{d_i^2}{w_i}\right). \\
\end{align*}
The combinatorial factor corresponding to the grouping of $\Gamma_i$ into the strings is the same: $\frac{r^{S-s}}{r!}$.
	
	\end{enumerate}

\textbf{Last step: }We now equalize the two intersection numbers by adding all contributions:
$$(Ia)^\mathrm{left}+(Ib)^\mathrm{left}+(Ic)^\mathrm{left}+(Ic')^\mathrm{left}+3(II)^\mathrm{left}=(Ia)^\mathrm{right}+(Ib)^\mathrm{right}+(Ib')^\mathrm{right}+3(II)^\mathrm{right}.$$
Notice that there is a factor $3$ for the correction terms since the strings can go in each of the three directions of ends of $L$. We notice that the summand
$$\frac{1}{n}\bino{n}{k_1,\cdots,k_S}\sum_{i=1}^S \frac{k_i d_i^2}{w_i},$$
appears on both sides of the equality. It is is not a surprise since it appeared by overcounting the curves where the intersection points with $L_1$, $L_2$ and $P_1$ lie in the same $\Gamma_i$, which is forbidden by each of the cross-ratio constraint.
	\end{proof}

%
%

\begin{rem}
Notice thatin Theorem \ref{theorem recursive formula r=1} we have
$$\left\{ \begin{array}{l}
2r+l+s=k+2,\\
r\leqslant l-1 \text{ because }l-r\geqslant S\geqslant 1,\\
s\leqslant l-r.\\
\end{array}\right.$$
Thus, we get that $k+2\leqslant 3l-1$.
\end{rem}

	\subsection{Practical cases of the formula}
	
		\subsubsection{The WDVV recursion for $\ang{\psi L}_d$, i.e.\ the $k=1$ case.}
		
	\begin{coro}\label{cor-WDVVk=1}
	Theorem \ref{theorem recursive formula r=1} simplifies to the following recursive formula for $k=1$:
	$$\boxed{ \begin{array}{>{\displaystyle}r>{\displaystyle}>{\displaystyle}l}
	\ang{\psi L}_d = \sum_{\substack{d_1+d_2=d \\ d_1,d_2 \geqslant 1}} & - \bino{3d-4}{3d_1-1}d_1^3 N_{d_1}(d_2\ang{\psi L}_{d_2}+N_{d_2})\\
	& + \bino{3d-4}{3d_1-2} d_1^2 d_2 N_{d_1}(d_2\ang{\psi L}_{d_2} + N_{d_2}) \\
	& + \bino{3d-4}{3d_1-3} d_1 d_2^2 N_{d_1}N_{d_2} . \\  
	\end{array} }$$
	\end{coro}
	
	\begin{proof}
	We use that for $k=1$, one has $\ang{\psi^{k-1} P}_d=N_d$, $\ang{\psi^{k-1} L}=dN_d$ and $\ang{\psi^{k-2} P}=0$. Concerning the correction sum, we have to sum over the integers $(r,l,s)$ with $2r+l+s=k+2=3$. We have $3l-1\geqslant 3$, so $l$ is at least $2$. It solves for:
	\begin{itemize}[label=$\ast$]
	\item $(r,l,s)=(0,2,1)$, for which the contribution cancels. This is no surprise because there is a unique string of type (c), so it is not possible to split the elements involved in the cross-ratio.
	\item $(r,l,s)=(0,3,0)$, which is impossible since $r=0$ forces $s=S\geqslant 1$.
	\end{itemize}
	\end{proof}
	
	\begin{expl}
	Using the formula from Corollary \ref{cor-WDVVk=1} and the initial value $\ang{\psi L}_1=2$, we get that $\ang{\psi L}_2=4$ and $\ang{\psi L}_3=60$, thus confirming our direct computations from Example \ref{ex-k=1}.
	\end{expl}
		
		\subsubsection{The WDVV recursion for $\ang{\psi^2 L}_d$, i.e.\ the $k=2$ case.}
		
	\begin{coro}\label{cor-WDVVk=2}
	Theorem \ref{theorem recursive formula r=1} simplifies to the following recursive formula for $k=2$:
	$$\boxed{ \begin{array}{>{\displaystyle}r>{\displaystyle}c>{\displaystyle}l}
	\ang{\psi^2 L}_d = &\sum_{\substack{d_1+d_2=d \\ d_1,d_2 \geqslant 1}} & - \bino{3d-5}{3d_1-1}d_1^3 N_{d_1}(d_2\ang{\psi^2 L}_{d_2}+\ang{\psi P}_{d_2}) \\
	& & + \bino{3d-5}{3d_1-2}d_1^2 N_{d_1}(d_2^2\ang{\psi^2 L}_{d_2}+2d_2\ang{\psi P}_{d_2} -\ang{\psi L}_{d_2})\\
	& & + \bino{3d-5}{3d_1-3}d_1 N_{d_1}(d_2\ang{\psi L}_{d_2}+N_{d_2})\\
	& + & \frac{3}{2}\left( (d-1)^2 N_{d-1} + \sum_{d_1+d_2=d}\bino{3d-5}{3d_1-1}d_2(d_1-d_2)N_{d_1}N_{d_2} \right).\\
	\end{array} }$$
	\end{coro}
	
	\begin{proof}
	We just put $k=2$ in the formula to get the first sum over $d_1+d_2=d$, replacing $\ang{\psi^0 P}_d$ by $N_d$. Concerning the correcting sum, we look for integers $(r,l,s)$ with $2r+l +s=k+2=4$, and $3l-1\geqslant 4$, so $l$ is at least $2$. This solves for:
		\begin{itemize}[label=$\ast$]
		\item $(r,l,s)=(1,2,0)$, the string pattern is as in Figure \ref{figure shape of strings going to infinity k=2} (a) and the contribution is
		\begin{align*}
	& 3\cdot\frac{1}{2}\cdot\bino{3d-4}{3d-4} \left( \frac{d}{3d-4}(3d-4)(d-1)-(d-1)^2\right)N_{d-1}(1) \\
	= & \frac{3}{2} (d-1)^2 N_d, \\
		\end{align*}
		since $N_{d-1}(1)=(d-1) N_d$.
		\item $(r,l,s)=(0,2,2)$, the string pattern is as in Figure \ref{figure shape of strings going to infinity k=2} (b) and the contribution is
		$$\frac{3}{2}\sum_{d_1+d_2=d}\bino{3d-5}{3d_1-1}d_2(d_1-d_2)N_{d_1}N_{d_2}.$$
		The symmetry is removed by labelling the components: the first contains $P_1$ (and is of degree $d_1$) and the second (of degree $d_2$) contains the intersection point with $L_2$.
		\item $(r,l,s)=(0,3,1)$, which contributes $0$ since $r=0$ forces $S=s$, and the contribution cancels.
		\item $(r,l,s)=(0,4,0)$, which is impossible since $r=0$ forces $s=S\geqslant 1$.
		\end{itemize}
	\end{proof}
	
	\begin{expl}
	Using the formula from Corollary \ref{cor-WDVVk=2} and the initial value $\ang{\psi^2 L}_1=0$, we recover $\ang{\psi^2 L}_2=\frac{9}{2}$ and $\ang{\psi^2 L}_3=54$, thus confirming our direct computation from Example \ref{ex-k=2}.
	\end{expl}
	
	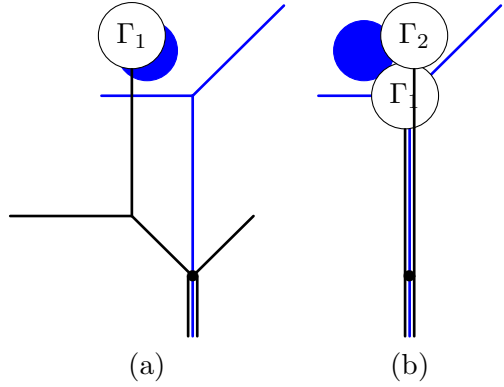
\begin{figure}
	\begin{center}
	\begin{tabular}{cc}
	\begin{tikzpicture}[line cap=round,line join=round,>=triangle 45,x=0.4cm,y=0.4cm]
\draw[line width=1pt,blue] (0,0)--++(-3,0)++(3,0)--++(3,3)++(-3,-3)--++(0,-6)--++(0,-2);
\draw[line width=1pt] (0.15,-6)--++(0,-2);
\draw[line width=1pt] (-0.15,-6)--++(0,-2);
\draw[blue,fill=blue] (-1.5,1.5) circle[radius=1];
\draw[line width=1pt] (0,-6)--++(-2,2)--++(-4,0);
\draw[line width=1pt] (0,-6)--++(2,2);
\curve (1) at (-2,2) {$\Gamma_1$};
\draw[line width=1pt] (-2,-4) to (1);
\draw (0,-6) node {$\bullet$};
\end{tikzpicture} & \begin{tikzpicture}[line cap=round,line join=round,>=triangle 45,x=0.4cm,y=0.4cm]
\draw[line width=1pt,blue] (0,0)--++(-3,0)++(3,0)--++(3,3)++(-3,-3)--++(0,-6)--++(0,-2);
\draw[line width=1pt] (0.15,-6)--++(0,-2);
\draw[line width=1pt] (-0.15,-6)--++(0,-2);
\draw[blue,fill=blue] (-1.5,1.5) circle[radius=1];
\curve (1) at (-0.15,0) {$\Gamma_1$};
\curve (2) at (0.15,2) {$\Gamma_2$};
\draw[line width=1pt] (-0.15,-6) to (1);
\draw[line width=1pt] (0.15,-6) to (2);
\draw (0,-6) node {$\bullet$};
\end{tikzpicture} \\
(a) & (b) \\
\end{tabular}		
	\caption{\label{figure shape of strings going to infinity k=2}Shape of the different strings going to infinity for $k=2$.}
	\end{center}
	\end{figure}
		
		\subsubsection{The WDVV recursion for $\ang{\psi^3 L}_d$, i.e.\ the $k=3$ case.}

	\begin{coro}\label{cor-WDVVk=3}
	Theorem \ref{theorem recursive formula r=1} simplifies to the following recursive formula for $k=3$:
	$$\boxed{ \begin{array}{>{\displaystyle}r>{\displaystyle}c>{\displaystyle}l}
	\ang{\psi^3 L}_d = &\sum_{\substack{d_1+d_2=d \\ d_1,d_2 >0}} & -\bino{3d-6}{3d_1-1} d_1^3 N_{d_1}(d_2\ang{\psi^3 L}_{d_2}+\ang{\psi^2 P}_{d_2}) \\
	& & + \bino{3d-6}{3d_1-2} d_1^2 N_{d_1}(d_2^2 \ang{\psi^3 L}_{d_2}+2d_2\ang{\psi^2 P}_{d_2}-\ang{\psi^2 L}_{d_2} ) \\
	& & + \bino{3d-6}{3d_1-3}d_1 N_{d_1}(d_2 \ang{\psi^2 L}_{d_2}+\ang{\psi P})\\
	& + & \frac{1}{2}(d-1) (N_{d-1}(2)+3N_{d-1})\\
	& + & \frac{1}{2}\sum_{d_1+d_2=d-1}\bino{3d-6}{3d_1-2}d_1 d_2 N_{d_1} N_{d_2} \left( d_1(d_2+1)-d_2^2 \right) \\
	& + & \frac{1}{2}\sum_{d_1+d_2=d} \left( \bino{3d-6}{3d_1-3} -\bino{3d-6}{3d_1-4}\right) d_2(d_1-d_2)\widetilde{N}_{d_1}(2)N_{d_2} .   \\
 	\end{array} }$$
	\end{coro}
	
	\begin{proof}
As before, for the first sum, we just replace $k$ by $3$. Then, we look for the correction term, with $2r+l+s=5$. We have $3l-1\geqslant 5$, so $l\geqslant 2$. Moreover, we have $s\leqslant l-r$ and $r\leqslant l-1$. This solves for
	\begin{itemize}[label=$\ast$]
	\item $(r,l,s)=(1,2,1)$, the pattern is as in Figure \ref{figure shape of strings going to infinity k=3} (a) and the contribution is
	$$\frac{3}{2}(d-1)N_{d-1}.$$
	\item $(r,l,s)=(1,3,0)$, we can have $S=2$ or $S=1$, giving the patterns presented in Figure \ref{figure shape of strings going to infinity k=3} (b) and (b'). The automorphisms contribute for $\frac{1}{3!}$. Assuming the first component is the one containing $P_1$, we can write the contributions to be respectively
	$$\frac{1}{2}\sum_{d_1+d_2=d-1} \bino{3d-6}{3d_1-2}d_1d_2N_{d_1}N_{d_2}(d_1(d_2+1)-d_2^2) ,$$
	and
	$$\frac{1}{2}(d-1)N_{d-1}(2) .$$
	\item $(r,l,s)=(0,3,2)$, the string pattern is as in Figure \ref{figure shape of strings going to infinity k=3} (c).  We can assume the first component to contain $P_1$, but the adjacent edge might be of weight $1$ or $2$. We thus get the contribution to be
	\begin{align*}
	 & \frac{1}{2} \sum_{d_1+d_2=d}\left( \bino{3d-6}{3d_1-3}d_2(d_1-d_2)\widetilde{N}_{d_1}(2)N_{d_2} + \bino{3d-6}{3d_1-2}d_2(d_1-d_2) N_{d_1}\widetilde{N}_{d_2}(2) \right) \\
	= & \frac{1}{2}\sum_{d_1+d_2=d} \left( \bino{3d-6}{3d_1-3} -\bino{3d-6}{3d_1-4}\right) d_2(d_1-d_2)\widetilde{N}_{d_1}(2)N_{d_2}, \\
	\end{align*}
	where the second row is obtained by switching $d_1$ and $d_2$ in the second part of the sum.
	\item $(r,l,s)=(0,4,1)$ contributes for $0$.
	\item $(r,l,s)=(0,5,0)$ is impossible since $r=0$ forces $s=S\geqslant 1$.
	\end{itemize}
	\end{proof}

	\begin{figure}
	\begin{center}
	\begin{tabular}{cccc}
	\begin{tikzpicture}[line cap=round,line join=round,>=triangle 45,x=0.4cm,y=0.4cm]
\draw[line width=1pt,blue] (0,0)--++(-3,0)++(3,0)--++(3,3)++(-3,-3)--++(0,-6)--++(0,-2);
\draw[line width=1pt] (0.15,-6)--++(0,-2);
\draw[line width=1pt] (-0.15,-6)--++(0,-2);
\draw[blue,fill=blue] (-1.5,1.5) circle[radius=1];
\draw[line width=1pt] (0,-6)--++(-4,0);
\draw[line width=1pt] (0,-6)--++(2,2);
\curve (1) at (0,2) {$\Gamma_1$};
\draw[line width=1pt] (0,-6) to (1);
\draw (0,-6) node {$\bullet$};
\end{tikzpicture} & \begin{tikzpicture}[line cap=round,line join=round,>=triangle 45,x=0.4cm,y=0.4cm]
\draw[line width=1pt,blue] (0,0)--++(-3,0)++(3,0)--++(3,3)++(-3,-3)--++(0,-6)--++(0,-2);
\draw[line width=1pt] (0.15,-6)--++(0,-2);
\draw[line width=1pt] (0,-6)--++(0,-2);
\draw[line width=1pt] (-0.15,-6)--++(0,-2);
\draw[blue,fill=blue] (-1.5,1.5) circle[radius=1];
\draw[line width=1pt] (0,-6)--++(2,2)++(-2,-2)--++(-1,2)--++(-1,1)--++(-2,0);
\curve (1) at (-2,0) {$\Gamma_1$};
\curve (2) at (-1,2) {$\Gamma_2$};
\draw[line width=1pt] (-2,-3) to (1);
\draw[line width=1pt] (-1,-4) to (2);
\draw (0,-6) node {$\bullet$};
\end{tikzpicture} & \begin{tikzpicture}[line cap=round,line join=round,>=triangle 45,x=0.4cm,y=0.4cm]
\draw[line width=1pt,blue] (0,0)--++(-3,0)++(3,0)--++(3,3)++(-3,-3)--++(0,-6)--++(0,-2);
\draw[line width=1pt] (0.15,-6)--++(0,-2);
\draw[line width=1pt] (0,-6)--++(0,-2);
\draw[line width=1pt] (-0.15,-6)--++(0,-2);
\draw[blue,fill=blue] (-1.5,1.5) circle[radius=1];
\draw[line width=1pt] (0,-6)--++(2,2)++(-2,-2)--++(-1,2)--++(-3,0);
\curve (1) at (-1,2) {$\Gamma_1$};
\draw[line width=1pt] (-1,-4) to node[midway,left] {$2$} (1);
\draw (0,-6) node {$\bullet$};
\end{tikzpicture} & \begin{tikzpicture}[line cap=round,line join=round,>=triangle 45,x=0.4cm,y=0.4cm]
\draw[line width=1pt,blue] (0,0)--++(-3,0)++(3,0)--++(3,3)++(-3,-3)--++(0,-6)--++(0,-2);
\draw[line width=1pt] (0.15,-6)--++(0,-2);
\draw[line width=1pt] (0,-6)--++(0,-2);
\draw[line width=1pt] (-0.15,-6)--++(0,-2);
\draw[blue,fill=blue] (-1.5,1.5) circle[radius=1];
\curve (1) at (-0.15,0) {$\Gamma_1$};
\curve (2) at (0.15,2) {$\Gamma_2$};
\draw[line width=1pt] (-0.15,-6) to node[midway,left] {$2$} (1);
\draw[line width=1pt] (0.15,-6) to (2);
\draw (0,-6) node {$\bullet$};
\end{tikzpicture} \\
(a) & (b) & (b') & (c) \\
\end{tabular}		
	\caption{\label{figure shape of strings going to infinity k=3}Shape of the different strings going to infinity for $k=3$.}
	\end{center}
	\end{figure}
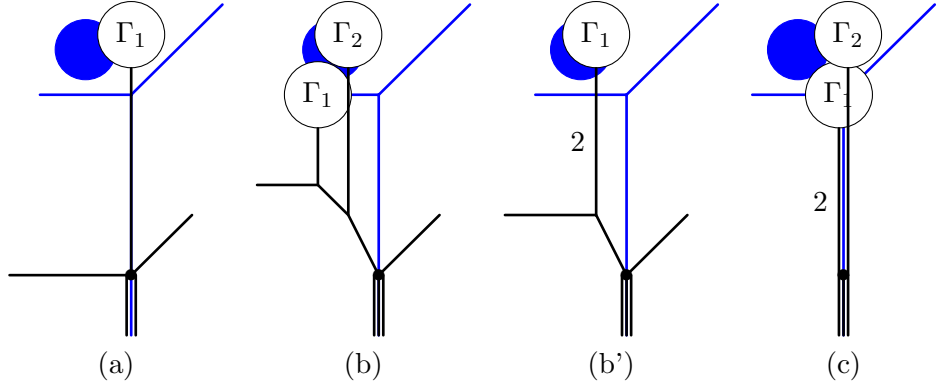

\begin{example}
Using the formula of Corollary \ref{cor-WDVVk=3}, we compute $\ang{\psi^3 L}_2=\frac{5}{2}$. This can also be confirmed by a direct computation as suggested in Subsection \ref{subsec-furthercases}.
\end{example}

\section{WDVV equation for $\ang{\psi L,\psi L}$}\label{sec-WDVV2}

In this section we prove a WDVV-type equation for the invariants $\ang{\psi L \psi L}_{d}$. We proceed analogously to Section \ref{sec-WDVV1}. To avoid heavy notation, we restrict to the case where each power of $\psi$ is just one. We consider the following evaluation map
$$f:\M_{0,n+4}(\RR^2,d)\cap\psi_1\cap \psi_2\longrightarrow \RR^2\times\RR^2\times\RR^2\times \RR^2\times(\RR^2)^n\times\M_{0,4}.$$
We aim to intersect its image with a cycle of the form $\Xi=L\times L'\times L_1\times L_2\times\{P_1\}\times\cdots\times\{P_n\}\times\{\lambda\}$, where $L,L',L_1$ and $L_2$ are generic tropical lines, $P_1,\dots,P_n$ are generic points inside $\RR^2$ and $\lambda\in\M_{0,4}$ is a very big cross-ratio inside some combinatorial type of $\M_{0,4}$. We require marking $1$ (also denoted $Q$), which is coupled with a $\psi$-condition, to meet the line $L$, and marking $2$ (also denoted $Q'$), which is also coupled with a $\psi$-condition, to meet the line $L'$. The lines $L_1$ and $L_2$ are merely helping constructions to produce the recursion from conditions with a very large cross-ratio.

We choose $n$ such that $n+2=3d-2$, so that the domain has dimension equal to the codimension of $\Xi$. The marked points remembered by the forgetful map are the markings associated to $L_1,L_2,L$ and $L'$. We care about the combinatorial types $L_1L_2//LL'$ and $L_1L//L_2L'$. We call the respective cycles $\Xi_A$ and $\Xi_B$.

\begin{theo}\label{thm-WDVV2}
	We have the following identity:

	$$\boxed{ \begin{array}{>{\displaystyle}r>{\displaystyle}c>{\displaystyle}l}
	\ang{\psi L,\psi L}_d = &\sum_{\substack{d_1+d_2=d \\ d_1,d_2 \geqslant 1}} & - \bino{3d-4}{3d_1-1} d_1^3 N_{d_1}(d_2 \ang{\psi L, \psi L}_{d_2}+2\ang{\psi L}_{d_2})\\
	& & -2\cdot  \bino{3d-4}{3d_1-2} d_1^2  N_{d_1}(d_2\ang{\psi L}_{d_2}+N_{d_2}) \\
			& & + \bino{3d-4}{3d_1-2} d_1 d_2 (d_1\cdot \ang{\psi L}_{d_1}+N_{d_1})\cdot (d_2\cdot \ang{\psi L}_{d_2}+N_{d_2}) \\
	& & +2\cdot  \bino{3d-4}{3d_1-1} d_1^2N_{d_1}\cdot (d_2\cdot \ang{\psi L}_{d_2}+N_{d_2})\\
	& & + 2 \cdot  \bino{3d-4}{3d_1-2} d_1 d_2 N_{d_1}\cdot (d_2\cdot \ang{\psi L}_{d_2}+N_{d_2})\\
	& & + \bino{3d-4}{3d_1-2} d_1 d_2 N_{d_1}N_{d_2}\\
	& -&  3(d-1)^3\cdot N_{d-1}.\\
	\end{array} }$$

	\end{theo}

\begin{proof}
The proof is analogous to the proof of Theorem \ref{theorem recursive formula r=1}. 

	We evaluate the intersection number between $f(\M_{0,n+4}(\RR^2,d)\cap\psi_Q\cap \psi_{Q'})$ and $\Xi_A/\Xi_B$ for a big value of $\lambda$ on the following two combinatorial types: $(L_1L_2//LL')$ and $(LL_1//L'L_2)$. 
	
\textbf{Step 1, the $(L_1L_2//LL')$ left side:} 
	\begin{enumerate}[label=(\Roman*)]
	\item We assume $\Gamma$ has a contracted edge. This edge splits the curve in two components $C_1$ and $C_2$ of respective degrees $d_1+d_2=d$. The marked points are distributed among the two components. We have several possibilities according to whether the contracted edge is adjacent to one or two marked points with a $\psi L$-constraint or not, and whether $d_1=0$ or $d_1\neq 0$. Let $n_1$ (resp. $n_2$) be the number of marked point on $C_1$ (resp. $C_2$).
		\begin{enumerate}[label=(\alph*)]
		\item If $d_1=0$, $C_1$ is contracted to the intersection point of $L_1$ and $L_2$, so that $C_2$ is of degree $d$ and has to pass through $L_1\cap L_2$. The contribution is equal to
		$$(Ia)^\mathrm{left}=\ang{\psi L, \psi L}_d.$$
	
	 It is not possible that $d_2=0$, as the marked ends $Q$ and $Q'$, if together, need to be at a $5$-valent vertex, but the contracted bounded edge together with the two markings yields only valency three.
		\item Assume now that $d_1\neq 0$, and the contracted edge is not adjacent to either of the $\psi L$-points. As $C_1$ is subject to $n_1$ point constraints, and $C_2$ to $n_2$ point constraints as well as to both $\psi L$ constraint, one has
	$$\left\{\begin{array}{l}
	n_1\leqslant 3d_1-1 , \\
	n_2\leqslant 3d_2-3.\\
\end{array}	 \right.$$
	Adding the two relations, we see that $n=n_1+n_2\leqslant 3(d_1+d_2)-4=n$, so we have in fact equality. Thus, we can pick independently $C_1$ and $C_2$ subject to point constraints ($N_{d_1}\ang{\psi L, \psi L}_{d_2}$), pick an intersection point between them where to insert the contracted edge ($d_1d_2$), and the intersection point between $L_1$, $L_2$ and $C_1$ ($d_1^2$). This contributes a term
	$$(Ib)^\mathrm{left}=\sum_{d_1+d_2=d} \bino{3d-4}{3d_1-1} d_1^3d_2 N_{d_1}\ang{\psi L, \psi L}_{d_2}.$$
		\item We now assume that the contracted edge is adjacent to precisely one of the $\psi L$-points. The valency of the vertex has to be $4$, of which two arise from the contracted bounded edge and the marked end itself. The remaining two edges point into opposite direction by the balancing condition, that is, the image just looks like a straight line. This image has to pass to the intersection point of $C_1$ and $C_2$, as it is adjacent to the contracted bounded edge, and through $L$ (resp.\ $L'$) as it is adjacent to $Q$ (resp.\ $Q'$).
		 We now have
		$$\left\{\begin{array}{l}
	n_1\leqslant 3d_1-1 , \\
	n_2\leqslant 3d_2-2.\\
\end{array}	 \right.$$
		Adding the two rows, we do not get an equality anymore, as both sides differ by one. So we are in one of the following situations:
		\begin{itemize}[label=$\circ$]
		\item We have $n_1=3d_1-1$. In this situation, we pick $C_1$ fixed by the points constraints ($N_{d_1}$), an intersection point with $L_1$ and $L_2$ ($d_1^2$). The curve $C_2$ has to meet one of the intersection point between $C_1$ and $L$, of which there are $d_1$. The curve $C_2$, in addition to passing through the remaining points, also still contains the second $\psi L$ condition ($\ang{ \psi L}_{d_2}$). 

As the  $\psi L$-point adjacent to the contracted bounded edge can be either $Q$ or $Q'$, and in both cases we have the same contribution, we have to multiply with a factor of $2$.
		In total, we get a contribution of
		$$(Ic)^\mathrm{left}=2\cdot \sum_{d_1+d_2=d} \bino{3d-4}{3d_1-1} d_1^3N_{d_1}\ang{\psi L}_{d_2}.$$
		\item We have $n_1=3d_1-2$. In this situation, $C_2$ is fixed by the $P$-constraints and $\psi L$-constraint ($\ang{\psi  L}_{d_2}$). The curve $C_1$ has then to pass through an intersection of $C_2$ with $L$ (resp.\ $L'$) of which there are $d_2$, in addition to the $n_1$ points ($N_{d_1}$). We then pick the intersection points with $L_1$ and $L_2$ ($d_1^2$). 		
		As above, we have to multiply by $2$ and get a contribution of
		$$(Ic')^\mathrm{left}=2\cdot \sum_{d_1+d_2=d} \bino{3d-4}{3d_1-2} d_1^2 d_2 N_{d_1}\ang{\psi L}_{d_2}.$$
		\end{itemize}
		
		\item We now assume that the contracted bounded edge is adjacent to both of the $\psi L$-points. The valency of this edge has to be $5$ in total, counting both contracted marked ends and the contracted bounded edge. The remaining $2$ edges are again opposite to each other and thus, the image looks like a straight line passing through the intersection point of $L$ and $L'$ through which also $C_1$ has to pass.
		
		 We  have
		$$\left\{\begin{array}{l}
	n_1\leqslant 3d_1-2 , \\
	n_2\leqslant 3d_2-2.\\
\end{array}	 \right.$$
	Thus, we have equality again. 	The curve $C_1$ is fixed by the point conditions together with the intersection point of $L$ and $L'$ ($N_{d_1}$), we can choose $d_1$ intersection points of $C_1$ with $L_i$ for $i=1,2$, and $C_2$ has to pass through the remaining points and the intersection point of $L$ and $L'$ ($N_{d_2}$). The vertex of valency five to which both $\psi$-conditions are adjacent contributes a factor of $2$.
	Altogether, we have
		$$(Id)^\mathrm{left}=2\cdot \sum_{d_1+d_2=d} \bino{3d-4}{3d_1-2} d_1^2N_{d_1}N_{d_2}.$$
		
	\end{enumerate}
	
	\item We now compute the correction term on the left side, which corresponds to tropical curves that do not contain any contracted edge. Momentarily forgetting about the cross-ratio constraint, we get a $1$-parameter family of solutions. Each combinatorial type contributing a solution gives a $1$-parameter family. We look for combinatorial types for which this $1$-parameter family is unbounded, as we intersect with a very big cross-ratio. Since we precisely lack one marked point compared to the situation of Lemma \ref{lemma position points r=2} giving the relative position of the marked points on the curves, we are in one of the following situations:
	\begin{itemize}[label=$\ast$]
	\item There is an unbounded component containing two ends. We get a $1$-parameter family by deforming the path from one end to another. This situation is already described in \cite{GM053}: the string is only comprised of two ends, and the unique adjacent bounded edge does not contribute to the cross-ratio, so that these combinatorial types do not contribute any solution.
	
	\item There is one of the $\psi L$-marking all of whose adjacent components are unbounded. We get a $1$-parameter family by moving the $\psi L$-marking along the line and the paths to the ends in the adjacent components along with it. In that case, we are as if there was only a unique $\psi L$-constraint, and the shape of the deformations that are not bounded are described by Lemma \ref{lemma shape of curves for big cross-ratio r=1}. Considering the exponent of the $\psi$-power is $1$, assuming the marking lies on the bottom end of its tropical line, the unique possibility is that there are vertical ends and a vertical edge of weight $2$ attached to it. The deformation makes the vertex go down. Unfortunately, the unique bounded edge does not contribute to the cross-ratio, and these solutions do not contribute either.
	
	\item Last, there is a unique bounded component which is adjacent to the two $\psi L$-markings. We get a $1$-parameter family by moving each $\psi L$-marking along their lines, the path between them inside the bounded component and the paths between the markings and the ends in the other adjacent components.
	
	Up to symmetry, assume that the marking $Q'$ moves along the bottom end of $L'$ which is right to $L$. We prove step by step that the only possible deformation pattern is as depicted on the left of Figure \ref{fig-deformk=2intotwo}. 
	\begin{itemize}[label=-]
	\item The vertex $Q'$ cannot be a flat horizontal vertex. By balancing, there is therefore at least one edge going down, which hence needs to be an bottom end.
	\item Similarly, as the vertex $Q'$ cannot be flat and vertical, still by balancing, there needs to be an edge going right, which hence is an end of slope $(1,1)$.
	\item As the vertex is trivalent, the remaining edge has slope $(-1,0)$.
	\item This latter edge cannot be an end since the curve is connected. As it misses the point region, the only possibility is that it meets $L$ at the other $\psi L$-marking: $Q$.
	\item By balancing, there needs to be a bottom end adjacent to $Q$. The balancing the remaining edge has slope $(-1,1)$.
	\item As this new edge still misses the point region, the only possibility is to split into a vertical edge that goes inside the point ragion and an end of slope $(-1,1)$. We exactly get the picture from Figure \ref{fig-deformk=2intotwo}.
\end{itemize}
	\end{itemize}

	
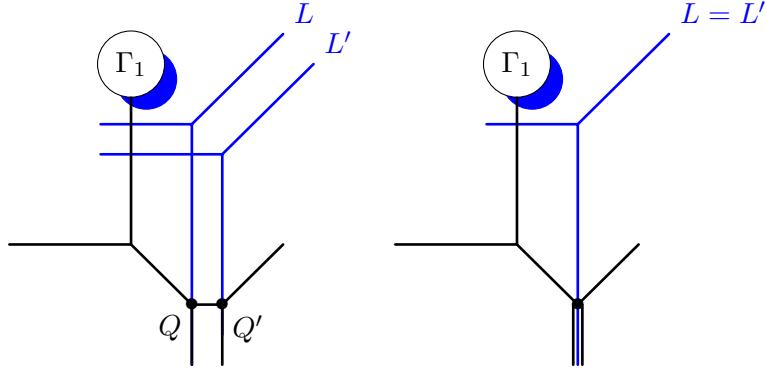
\begin{figure}
\begin{center}

\begin{tabular}{cc}
\begin{tikzpicture}[line cap=round,line join=round,>=triangle 45,x=0.4cm,y=0.4cm]
\draw[line width=1pt,blue] (0,0)--++(-3,0)++(3,0)--++(3,3) node[above right] {$L$} ++(-3,-3)--++(0,-6)--++(0,-2);
\draw[line width=1pt,blue] (1,-1)--++(-4,0)++(4,0)--++(3,3) node[above right] {$L'$} ++(-3,-3)--++(0,-6);
\draw[blue,fill=blue] (-1.5,1.5) circle[radius=1];
\draw[line width=1pt] (0,-6)--++(-2,2)--++(-4,0);
\draw[line width=1pt] (0,-8)--++(0,2)--++(1,0)--++(0,-2)++(0,2)--++(2,2);
\curve (1) at (-2,2) {$\Gamma_1$};
\draw[line width=1pt] (-2,-4) to (1);
\draw (0,-6) node {$\bullet$} node[below left] {$Q$};
\draw (1,-6) node {$\bullet$} node[below right] {$Q'$};
\end{tikzpicture} & \begin{tikzpicture}[line cap=round,line join=round,>=triangle 45,x=0.4cm,y=0.4cm]
\draw[line width=1pt,blue] (0,0)--++(-3,0)++(3,0)--++(3,3) node[above right] {$L=L'$} ++(-3,-3)--++(0,-6)--++(0,-2);
\draw[line width=1pt] (0.15,-6)--++(0,-2);
\draw[line width=1pt] (-0.15,-6)--++(0,-2);
\draw[blue,fill=blue] (-1.5,1.5) circle[radius=1];
\draw[line width=1pt] (0,-6)--++(-2,2)--++(-4,0);
\draw[line width=1pt] (0,-6)--++(2,2);
\curve (1) at (-2,2) {$\Gamma_1$};
\draw[line width=1pt] (-2,-4) to (1);
\draw (0,-6) node {$\bullet$};
\end{tikzpicture} \\
\end{tabular}
\end{center}

\caption{The deformation of a $k=2$-string as in Figure \ref{figure shape of strings going to infinity k=2} to a string with two $\psi L$ constraints at infinity.}\label{fig-deformk=2intotwo}
\end{figure}

	The remaining curve is of degree $d-1$, passes through $n=3d-4=3(d-1)-1$ points and has a fixed down end of weight one, contributing a factor of $(d-1)\cdot N_{d-1}$, where $(d-1)$ accounts for the choice of the fixed down end. Furthermore, we need to pick the intersections with $L_1$ and $L_2$, giving another $(d-1)^2$ choices. Altogether, we obtain a contribution of 
	$$(Ie)^\mathrm{left}=3(d-1)^3\cdot N_{d-1}.$$


	\end{enumerate}

\textbf{Step 2, the $(LL_1//L'L_2)$ right side:} 

	\begin{enumerate}[label=(\Roman*)]
	\item As in step 1, we assume $\Gamma$ has a contracted edge which splits the curve in two components $C_1$ and $C_2$ of respective degrees $d_1+d_2=d$. Let $n_1$ (resp. $n_2$) be the number of marked point on $C_1$ (resp. $C_2$).
		\begin{enumerate}[label=(\alph*)]
		\item We cannot have $d_i=0$, as the marked ends $Q$ resp.\ $Q'$ need to be at a $4$-valent vertex, but the contracted bounded edge together with the two markings (one for $Q$ resp.\ $Q'$, one for $L_i$) yields only valency three.
		\item Assume now that $d_1\neq 0$, and the contracted edge is not adjacent to either of the $\psi L$-points. As $C_1$ is subject to $n_1$ point constraints and one $\psi L$ constraint, and $C_2$ to $n_2$ point constraints and one $\psi L$ constraint, we have
	$$\left\{\begin{array}{l}
	n_1\leqslant 3d_1-2 , \\
	n_2\leqslant 3d_2-2.\\
\end{array}	 \right.$$
	Adding the two relations, we see that $n=n_1+n_2\leqslant 3(d_1+d_2)-4=n$, so we have in fact equality. Thus, we can pick independently $C_1$ and $C_2$ subject to point constraints.
As both curves also meet a line condition without $\psi$, $L_i$, we can use the divisor equation Lemma \ref{lem-divisoreq} to express these factors as $d_i\cdot \ang{\psi L}_{d_i}+N_{d_i}$ for $i=1,2$.
Furthermore, we pick an intersection point between $C_1$ and $C_2$ where to insert the contracted edge ($d_1d_2$). This contributes a term
	$$(Ib)^\mathrm{right}=\sum_{d_1+d_2=d} \bino{3d-4}{3d_1-2} d_1 d_2 (d_1\cdot \ang{\psi L}_{d_1}+N_{d_1})\cdot (d_2\cdot \ang{\psi L}_{d_2}+N_{d_2}) .$$
		\item We now assume that the contracted edge is adjacent to the $\psi L$-point in $C_1$, but not to the $\psi L$-point in $C_2$.
The situation where the contracted edge is adjacent to the $\psi L$-point in $C_2$ but not the one in $C_1$ is completely analogous and symmetric (by exchanging $d_1$ and $d_2$, but we sum over all pairs satisfying $d_1+d_2=d$), which accounts for a factor of $2$ in our contributions below.

		The valency of the vertex has to be $4$, of which two arise from the contracted bounded edge and the marked end itself. The remaining two edges point into opposite direction by the balancing condition, that is, the image just looks like a straight line. This image has to pass to the intersection point of $C_1$ and $C_2$, as it is adjacent to the contracted bounded edge, and through $L$ as it is adjacent to $Q$.
		 We now have
		$$\left\{\begin{array}{l}
	n_1\leqslant 3d_1-1 , \\
	n_2\leqslant 3d_2-2.\\
\end{array}	 \right.$$
		Adding the two rows, we do not get an equality anymore, as both sides differ by one. So we are in one of the following situations:
		\begin{itemize}[label=$\circ$]
		\item We have $n_1=3d_1-1$. In this situation, we pick $C_1$ fixed by the points constraints ($N_{d_1}$), and an intersection point with $L_1$ ($d_1$). The curve $C_2$ has to meet one of the intersection points between $C_1$ and $L$, of which there are $d_1$. The curve $C_2$, in addition to passing through the remaining points, also still contains the second $\psi L$ condition, and passes through $L_2$, so using the divisor equation \ref{lem-divisoreq} again we obtain a factor of $d_2\cdot \ang{\psi L}_{d_2}+N_{d_2}$. 
We get a contribution of
		$$(Ic)^\mathrm{right}=2\cdot \sum_{d_1+d_2=d} \bino{3d-4}{3d_1-1} d_1^2N_{d_1}\cdot (d_2\cdot \ang{\psi L}_{d_2}+N_{d_2}).$$
		\item We have $n_1=3d_1-2$. In this situation, $C_2$ is fixed by the $P$-constraints and $\psi L$-constraint, and passes through $L_2$ in addition ($d_2\cdot \ang{\psi L}_{d_2}+N_{d_2}$). The curve $C_1$ has then to pass through an intersection of $C_2$ with $L$ of which there are $d_2$, in addition to the $n_1$ points ($N_{d_1}$). We finally pick the intersection point of $C_1$ with $L_1$  ($d_1$). 		
		We obtain
		$$(Ic')^\mathrm{right}=2 \cdot \sum_{d_1+d_2=d} \bino{3d-4}{3d_1-2} d_1 d_2 N_{d_1}\cdot (d_2\cdot \ang{\psi L}_{d_2}+N_{d_2}).$$
		\end{itemize}

	\item We finally assume that the contracted edge is adjacent to the $\psi L$-point in $C_1$ and to the $\psi L$-point in $C_2$. The images of both vertices adjacent to the contracted bounded edge look like straight lines, i.e.\ the intersection of $C_1$ and $C_2$ has to pass through $L\cap L'$.
		 We have
		$$\left\{\begin{array}{l}
	n_1\leqslant 3d_1-2 , \\
	n_2\leqslant 3d_2-2.\\
\end{array}	 \right.$$
		
		Thus, we have equality again. 	The curve $C_1$ is fixed by the point conditions together with the intersection point of $L$ and $L'$ ($N_{d_1}$), we can choose $d_i$ intersection poihts of $C_i$ with $L_i$ for $i=1,2$, and $C_2$ has to pass through the remaining points and the intersection point of $L$ and $L'$ ($N_{d_2}$). 
	Altogether, we have
		$$(Id)^\mathrm{right}= \sum_{d_1+d_2=d} \bino{3d-4}{3d_1-2} d_1 d_2 N_{d_1}N_{d_2}.$$

	\end{enumerate}

On this side, we cannot have correction terms coming from strings going to infinity: the deformation has already been described on the left side, and as the $\psi L$-markings are close to each other, the cross-ratio is always bounded.


\end{enumerate}

As the last step, we set the two sides equal and solve for our wanted enumerative invariant, obtaining the formula in the statement.
\end{proof}

\begin{example}
Using the initial value $\ang{\psi L,\psi L}_1=2$ (obtained from a line with vertex at the intersection of $L$ and $L'$, and both marked points $Q$ and $Q'$ adjacent to this vertex, yielding a $5$-valent vertex contributing a multiplicity factor of $2$), we can compute 
$\ang{\psi L,\psi L}_2=17$ and $\ang{\psi L,\psi L}_3=302$, confirming the computation in Examples \ref{ex-twolines1} and \ref{ex-twolines2} obtained from Theorem \ref{thm-2lines}.
\end{example}

\bibliographystyle{plain}
\bibliography{biblio}

%
%
%

\end{document}